%% file: main.tex
\documentclass[11pt,reqno]{amsart}
\input{macros}
\hypersetup{
  pdftitle={A strongly aperiodic monotile in three dimensions},
  pdfauthor={Ioannis Tsiokos}
}

\begin{document}

\title[A strongly aperiodic monotile in three dimensions]{A strongly aperiodic monotile in three dimensions}

\author{Ioannis Tsiokos}
\address{}

\date{16 September 2026 (revision 2; mathematical baseline commit \texttt{d90313a717})}

% ledger: Q1, Q6, Q7, T1, T2, T3, T4, T6, T8, O1, O4, H5, H6, M1, M2, M3, D4, Q23 (abstract; OUTLINE.md); Q-G author's account (the provenance sentence)
\begin{abstract}
Socolar and Taylor asked for a single, simply connected three-dimensional prototile that forces
nonperiodicity by shape alone, admitting no weakly nonperiodic tiling; the Schmitt--Conway--Danzer
biprism and the three-dimensional Socolar--Taylor tile admit screw motions or a periodic stacking
direction. We exhibit a rational polyhedral $3$-ball $Q$, which we call Chair44 (R44): a
seven-cube chair whose $24$ exposed unit
panels carry tiny square-pyramid features, and prove, as a proof submission, that $Q$ admits tilings
of $\R^3$ by congruent copies, reflections allowed, and that every such tiling has no translational
period and a symmetry group of order at most $24$; every tiling is homochiral and carries a unique
infinite hierarchy of nested supertiles. The solid was designed to a reading of the
aperiodic-monotile phenomenon reached with the Six Birds emergence calculus (Section~3.3), and
the construction turns on a single finite test, checked by machine: the tile's own contact rule
survives coarsening, so that the decoded parent tiling obeys the tile's rule and no other. The
proof combines a written geometric argument, that the features force
every tiling onto a registered lattice, with exhaustive finite enumerations; the companion census is
replayed by two independent implementations, every finite gate is kernel-checked in Lean~4 (modulo a
named compiler hook per \texttt{native\_decide} theorem), and the written geometric lemmas and the
logical assembly are Lean theorems as well, so that the theorem is kernel-checked modulo the named
compiler hooks; the written proofs remain as exposition.
\end{abstract}

\maketitle

\frontmatternote
\begin{center}
  {\small\copyright\ 2026 Ioannis Tsiokos. Licensed under
  \href{https://creativecommons.org/licenses/by/4.0/}{CC BY 4.0}.}
\end{center}
% The repository PDF displays its Zenodo DOI.  The arXiv source builder
% deliberately omits this local include, so the arXiv version has no DOI line.
\IfFileExists{publication_doi.tex}{\input{publication_doi}}{}

\input{sec1_intro}
\input{sec2_solid}
\input{sec3_finding}
\input{sec4_companions}
\input{sec5_registration}
\input{sec6_hierarchy}
\input{sec7_aperiodicity}
\input{sec8_mechanization}
\clearpage
\input{sec9_context}
\input{sec9_remarks}

\section*{Acknowledgements}
% ledger: D6, H13, M8; reviewers and tools by role; the framework is cited in sec:found only
The repository archives two external adversarial review rounds of the development and one of
this manuscript; the latter corrected the count of fixed seeds and supplied the first
non-exhausting and the first symmetric fixed points (\cref{sec:hierarchy,sec:limitperiodic}),
and its report is filed with the paper's review record. Automated reasoning agents also
checked the development, certificates and manuscript under the recorded review protocol. The
remaining errors are ours. We thank the authors of the hat and Spectre papers for the form of
the definitions, of the theorems and of the disclosure that this paper follows, and the authors
of Lean~4 and Mathlib.

\section*{Use of AI systems}
% ledger: H13, H14, M8, D4; source LINEAGE.md (packet table), provenance/HASH_CHAIN.md, lean/R44/MECHANIZATION_PLAN.md (agent roles), paper/review/PROTOCOL.md; owner statement 2026-09-09
This paper and the work it reports were produced with artificial-intelligence systems under the
author's direction; the author takes responsibility for every statement, and no AI system is an
author. \emph{The discovery.} The solid was found by an OpenAI reasoning model (ChatGPT, Astra)
that was given the Six Birds framework, the author's emergence calculus
\cite{Tsi26F6,Tsi26F4,Tsi26NDO}, together with the construction record of the repository, and
that worked from the reading of the aperiodic-monotile phenomenon described in
\cref{sec:reading,sec:found}: the framework was the input, the solid the output. Its six
hash-chained packets (\cref{sec:found}, \texttt{provenance/}) contain the periodic control, the
frame change, the $44$-contact atlas and the written alignment proof; the construction record that
preceded them was likewise produced by automated agents under the author's direction. \emph{The
verification.} The Lean formalization, the finite certificates and the negative controls were
implemented by automated coding agents (OpenAI Codex) and reviewed by agents of a different
model family (xAI Grok; OpenAI Codex on a second model) under the protocols recorded in the
repository, and the two external adversarial review rounds of \cref{sec:lean} and the external
review of this manuscript were carried out by such agents. \emph{The manuscript.} This text was largely written by Claude Fable~5.1
(Anthropic) from the author's plan, the repository's proofs and certificates and the literature
corpus, and was reviewed step by step by an independent automated reviewer (OpenAI Codex,
gpt-6-astra) whose verdicts are archived with the review package; the author directed, read and
edited every part.

\appendix
\input{appA_data}
\input{appB_census}
\input{appC_lean}
\input{appD_calibration}
\input{appE_geometry}
\clearpage
\input{notation}
\clearpage

\bibliographystyle{plainnat}
\bibliography{refs}

\end{document}

%% file: macros.tex
\usepackage[T1]{fontenc}
\usepackage[utf8]{inputenc}
\DeclareUnicodeCharacter{2014}{---}
\DeclareUnicodeCharacter{2192}{\ensuremath{\to}}
\DeclareUnicodeCharacter{21A6}{\ensuremath{\mapsto}}
\DeclareUnicodeCharacter{03C7}{\ensuremath{\chi}}
\DeclareUnicodeCharacter{2200}{\ensuremath{\forall}}
\DeclareUnicodeCharacter{2227}{\ensuremath{\wedge}}
\DeclareUnicodeCharacter{2264}{\ensuremath{\le}}
\usepackage{amsmath,amssymb,amsthm,mathtools}
\usepackage{graphicx}
\usepackage{booktabs}
\usepackage{longtable}
\usepackage{array}
\usepackage{xcolor}
\usepackage{enumitem}
\usepackage{caption}
\usepackage{microtype}
\usepackage{tikz}
\usetikzlibrary{positioning,arrows.meta}
\usepackage[numbers,sort&compress]{natbib}
\usepackage[colorlinks=true,linkcolor=blue!50!black,citecolor=blue!50!black,urlcolor=blue!50!black]{hyperref}
\usepackage[capitalise,noabbrev]{cleveref}

\graphicspath{{figures/}}

\theoremstyle{plain}
\newtheorem{theorem}{Theorem}[section]
\newtheorem{lemma}[theorem]{Lemma}
\newtheorem{proposition}[theorem]{Proposition}
\newtheorem{corollary}[theorem]{Corollary}
\theoremstyle{definition}

\theoremstyle{remark}

\newcommand{\R}{\mathbb{R}}
\newcommand{\Z}{\mathbb{Z}}
\newcommand{\Per}{\operatorname{Per}}
\newcommand{\Sym}{\operatorname{Sym}}
\newcommand{\solid}{Q}
\newcommand{\carrier}{P}
\newcommand{\atlas}{\mathcal{A}_{44}}
\DeclareUrlCommand\file{\urlstyle{tt}}
\newcommand{\lean}[1]{\begingroup\def\_{\char`\_\allowbreak}\texttt{#1}\endgroup}

\newcommand{\tier}[1]{\textsc{\small #1}}
\newcolumntype{L}[1]{>{\raggedright\arraybackslash}p{#1}}

\definecolor{outlinegray}{gray}{0.40}

\newcommand{\frontmatternote}{\noindent{\small\emph{Proof submission.} The finite checks, the
nineteen former hypothesis statements, and the logical assembly of the four tiling clauses are
Lean theorems, kernel-checked modulo the named compiler hooks. \Cref{app:lean} records the empty
hypothesis structure and the dependencies; the written proofs remain as exposition.}\par\medskip}

%% file: sec1_intro.tex
\section{Introduction}\label{sec:intro}

\subsection{The einstein problem in \texorpdfstring{$\R^3$}{R3}}\label{sec:problem}

% ledger: Q1, Q2, Q23, T1, T3, O1, O4, O12 (Figure 1); source 1009.1419 L88; 2509.12216 L249-L253; form 2303.10798 L13, L43
Socolar and Taylor, presenting the three-dimensional version of their hexagonal tile, set out
what remained to be found: ``It is still interesting, however, to search for a single, simply
connected 2D or 3D prototile that forces maximal nonperiodicity by shape alone, or one that does
not permit any weakly nonperiodic tilings'' \cite{ST12}. Writing after the hat, Kaplan restated
the demand from the other side: ``Some people regard this screw motion as uncomfortably close to
translation, and demand a strongly aperiodic 3D monotile, one that admits tilings whose symmetries
never include an infinite cyclic subgroup of any kind. Many of the ideas and algorithms we used to
prove the hat's aperiodicity could be adapted to work in 3D space, but personally I am daunted by
the prospect of deducing the behaviour of any candidate shapes that might be discovered there''
\cite{Kap25}. The solid $\solid$ of \cref{fig:tile} is such a prototile: a closed polyhedral
$3$-ball with rational vertex coordinates that admits tilings of $\R^3$ by congruent copies, and
every tiling it admits has a finite symmetry group (\cref{thm:main,thm:full} below).

\begin{figure}[ht]
  \centering
  \includegraphics[width=\linewidth]{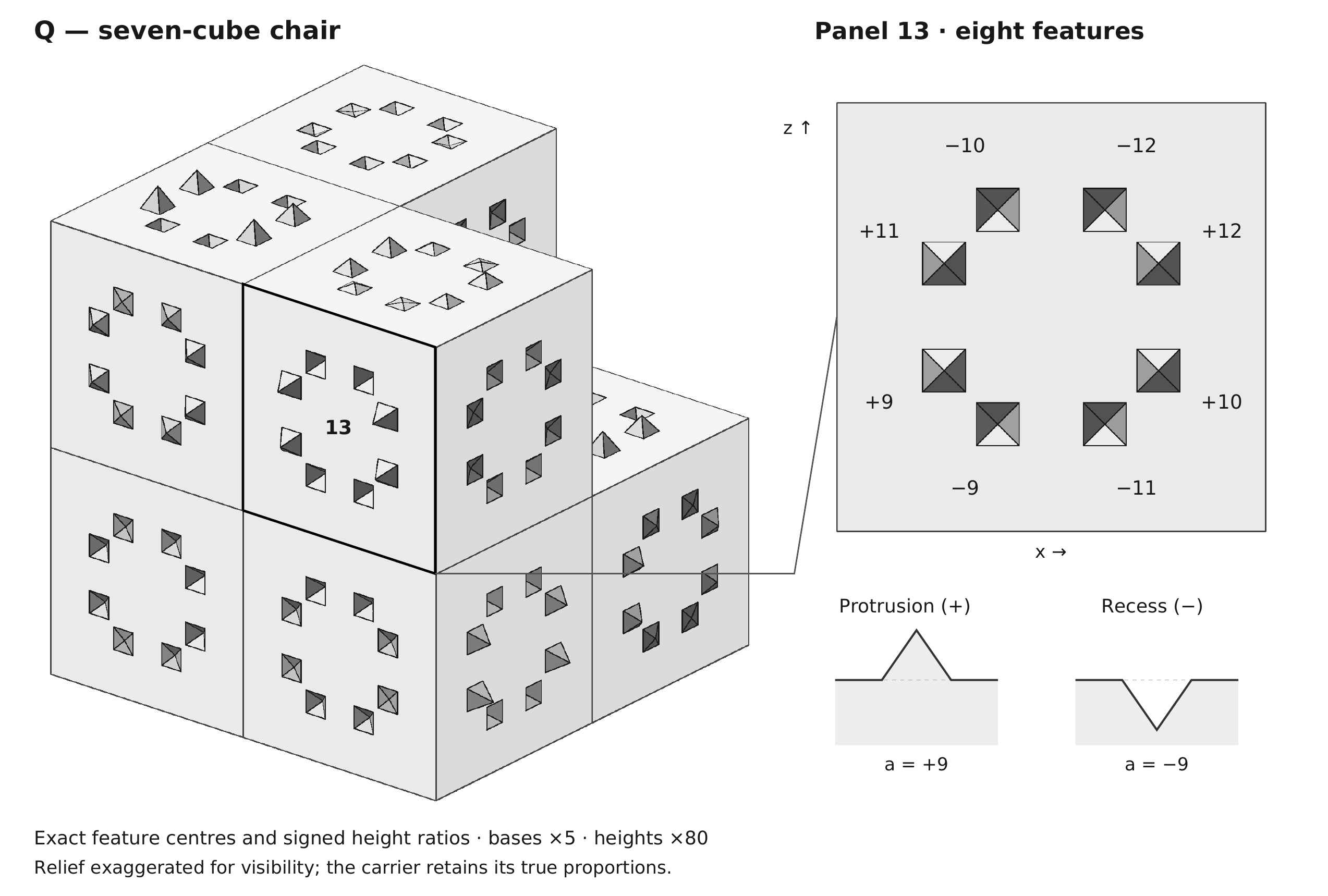}
  \caption{The solid $\solid$, with feature bases enlarged $\times 5$ and signed heights
  $\times 80$ for visibility. The seven-cube carrier, feature centres, polarities and relative
  heights are preserved. The outlined panel~$13$ and its inset show the same eight features;
  signed labels give $a$, whose true height is $a/10000$ (positive for a protrusion, negative
  for a recess). The sections illustrate its $+9$ and $-9$ features. The true base side is
  $1/50$; the drawing is not at true feature scale.}
  \label{fig:tile}
\end{figure}

% ledger: Q21; source 2303.10798 L18-L24
The study of aperiodic tilings begins with Wang's tiles and his conjecture that any set of
Wang tiles admitting a tiling admits a periodic one \cite{Wan61}; Berger disproved it with an
aperiodic set of $20{,}426$ Wang tiles and derived the undecidability of the domino problem
\cite{Ber66}. Robinson gave a six-tile aperiodic construction \cite{Rob71} and Penrose a two-tile
one \cite{Pen78}; these counts use different permitted motions and matching conventions from Wang
tiles, for which Jeandel and Rao proved that the minimum is eleven \cite{JR21}. The natural endpoint
of this reduction is a single shape that admits tilings but only non-periodic ones, an aperiodic
monotile or ``einstein'' \cite{SMKGS24}.

% ledger: Q12, Q8, Q21; source 2303.10798 L28-L35; 1903.01158 L11; 2211.15847 L28
Every candidate before the hat conceded something. Gummelt's decagon covers the plane only with
overlaps \cite{Gum96}; Penrose's $1+\varepsilon+\varepsilon^2$ set lets an arbitrarily high
fraction of the area lie in one tile, but the other tiles remain necessary \cite{Pen97}; the
Taylor--Socolar hexagon needs markings that let
non-adjacent tiles exchange information, and a geometric version of the rule needs a disconnected
tile, a tile with cutpoints, or a three-dimensional shape tiling a thickened plane
\cite{Tay10,ST11,ST12}; the Walton--Whittaker tile uses orientational rules \cite{WW23}; and
Mampusti and Whittaker say of their dendritic construction that ``the monotile we present here
is not an einstein in the technical sense'' \cite{MW20}. In a different direction, Greenfeld and
Tao disproved the periodic tiling conjecture with a single tile that tiles $\Z^d$ by translations
only, for a dimension $d$ that ``will be extremely large'' and with a tile that need not be
connected; whether $\Z^3$ admits such a tile they leave open \cite{GT24}, and Greenfeld and
Kolountzakis showed that the tile can be taken connected \cite{GK23}.

% ledger: Q9, Q10, Q25; source 2303.10798 L13, L43; 2305.17743 L45
In the plane the question is settled. The hat, a polykite, is an aperiodic monotile whose tilings
use both the tile and its mirror image \cite{SMKGS24}; the Spectre tiles aperiodically with a
single handedness even when reflections are allowed \cite{SMKGS24b}. Both papers give a
hierarchical, computer-assisted proof, and the hat paper a second, computer-free one. Simplified
proofs and the structure of the tilings followed quickly \cite{AA25,BGS25,BGMM25,Soc23}.

% ledger: Q7, Q6, Q13, Q2, Q24; source 2303.10798 L33, L37; 1009.1419 L14-L16, L30, L69-L72; 1003.4279 L171; 1003.4909 L45-L48
Three dimensions have their own history. The second part of Hilbert's eighteenth problem asked
whether anisohedral polyhedra exist in $\R^3$ \cite{Hil02}; Reinhardt found one \cite{Rei28},
Heesch a planar one \cite{Hee35}. The Schmitt--Conway--Danzer biprism, a convex rhombic biprism
going back to an unpublished 1988 example of Schmitt \cite{Sen96,Rad20}, admits, with reflections
forbidden, only non-periodic tilings: its tilings are periodic layers stacked with a twist, and it
admits a tiling whose symmetry group contains a screw motion; with reflections allowed it admits
a periodic tiling; Socolar and Taylor call the structure ``heterogeneously periodic'' \cite{ST12}. Their own three-dimensional tile enforces the
hexagonal matching rules by shape alone and can be made a topological sphere, but its tilings form
``corrugated slabs of the limit-periodic pattern that stack periodically in the direction
orthogonal to the nonperiodic tiling plane'' \cite{ST11}, which makes it ``a weakly nonperiodic
tiling by Goodman-Strauss's definition'' \cite{ST12}. Fletcher turned \v{C}ul\'ik and Kari's set of
$21$ Wang cubes \cite{CK95} into a single cubic prototile whose $21$ types are $21$ orientations,
but the matching rules become an atlas of legal one-corona patches rather than the shape of the
tile \cite{Fle10}. Kaplan's survey
calls the biprism ``both an existence proof and a cautionary tale about the definition of
aperiodicity in 3D space'' \cite{Kap25}.

% ledger: Q3, Q4, Q5, Q26, T4, T5, N3; source 2303.10798 L33-L34, L64; 2409.15880 L43-L45, L88, L234, L236
That definition is the point on which the three-dimensional problem turns. The hat paper,
``following Mozes'', calls a set of tiles \emph{strongly aperiodic} if it admits tilings but none
with any infinite cyclic symmetry, and \emph{weakly aperiodic} if it admits tilings but none with
a cocompact discrete group of symmetries \cite{SMKGS24,Moz97}. In the Euclidean plane the two
notions coincide for normal tiles \cite[Theorem~3.7.1]{GS16}; in the hyperbolic plane weak
aperiodicity comes readily, as with B\"or\"oczky's monotile \cite{Bor74}, while Goodman-Strauss
constructed a strongly aperiodic set there \cite{GS05}. In $\R^3$ the notions differ: with
reflections forbidden, the biprism is weakly but not strongly aperiodic. We prove the strong form: every tiling by $\solid$ has a symmetry group of
order at most $24$, and a finite group has no element of infinite order (\cref{cor:strong}). Coulbois, Gajardo,
Guillon and Lutfalla use a finer taxonomy in which a tile all of whose tilings have finite
symmetry group is \emph{mildly} aperiodic and ``strongly aperiodic'' is reserved for trivial
stabilizers; in their terms the hat is mildly aperiodic, ``strongly aperiodic monotiles are
unknown in all settings'', and their Question~32 asks for one \cite{CGGL24}. In their terms
$\solid$ is mildly aperiodic and not strongly aperiodic: an explicit tiling by $\solid$ has
symmetry group of order eight, and others have trivial symmetry group
(\cref{sec:aperiodicity,sec:limitperiodic}).

% ledger: O10, O1, O4, N4, T10
$\solid$ is not convex, as the biprism is, and it is not simple: it is a seven-cube chair whose
$24$ exposed unit panels each carry eight square pyramids of height between $1/10000$ and
$12/10000$ (\cref{sec:solid}). It is, however, a closed topological $3$-ball with rational
coordinates, and no matching rule of any kind is imposed: the tilings in \cref{thm:full} are all
the ways congruent copies of the solid can fill space without overlapping.

\subsection{Main result}\label{sec:main}

% ledger: T1-T4, T6, T8, T10, D4, T12; form 2303.10798 L42-L43; 2305.17743 L45
In this paper we prove the following, as a proof submission in the sense explained at the end of
this subsection.

\begin{theorem}\label{thm:main}
The solid $\solid$ of \cref{fig:tile} is a strongly aperiodic monotile of $\R^3$: $\solid$ admits
tilings of $\R^3$ by congruent copies, and no such tiling has a symmetry of infinite order.
\end{theorem}

\begin{theorem}\label{thm:full}
Let $\solid$ be the rational polyhedral solid of \cref{sec:solid}. Then
\begin{enumerate}[label=\textup{(\arabic*)}]
  \item there is a tiling of $\R^3$ by isometric copies of $\solid$;
  \item for every tiling $T$ by isometric copies of $\solid$, $\Per(T)=\{0\}$ and $|\Sym(T)|\le 24$;
  \item every tiling $T$ is homochiral;
  \item in every tiling $T$ each tile lies in a unique infinite hierarchy of $2^n$-scaled chair
    supertiles in registered poses, nested through the eight child poses.
\end{enumerate}
Isometries include reflections. No face-to-face, lattice, common-orientation,
connected-contact-graph or local-finiteness hypothesis is imposed.
\end{theorem}

% ledger: T10, T11, D1, Q3; source 2303.10798 L58
Here $\Per(T)$ is the set of translation vectors $v$ with $T+v=T$ and $\Sym(T)$ the group of
isometries of $\R^3$ that permute the tiles of $T$ (\cref{sec:terminology}). The absence of
hypotheses in \cref{thm:full} is deliberate. A tiling is any family of isometric copies of
$\solid$ with pairwise disjoint interiors whose union is $\R^3$: the copies need not meet
face-to-face, need not sit on a lattice, need not share an orientation, and their contact graph
need not be connected. Local finiteness, which the hat paper notes is automatic for monohedral
plane tilings by closed disks, is proved here from the geometry of $\solid$
(\cref{sec:companions}) rather than assumed.

% ledger: T6, T7, D5, D7, Q7; source 2305.17743 L22; 2303.10798 L52-L54; 1009.1419 L72
Reflections are allowed in \cref{thm:full}, and clause (3) says what they do: nothing. By
longstanding tradition in the tiling literature, ``shapes are considered congruent if they are
equivalent under any Euclidean isometry, including those that reverse orientation''
\cite{SMKGS24}, and the hat's status as a monotile depends on that convention, since every hat
tiling uses both handednesses. $\solid$ needs no such caveat. Its tilings are \emph{homochiral}
in the Spectre paper's sense, ``for every pair $T_i,T_j\in\mathcal{T}$, there is an
orientation-preserving isometry mapping $T_i$ to $T_j$'' \cite{SMKGS24b}: a tiling by $\solid$ is
wholly left-handed or wholly right-handed, and a reflected and an unreflected copy never occur in
the same tiling. By the same paper's definition, a tile that admits only homochiral non-periodic
tilings is a \emph{strictly chiral aperiodic monotile}; $\solid$ is one, in $\R^3$
(\cref{cor:chiral}). The
Schmitt--Conway--Danzer biprism, by contrast, is chiral but admits a periodic tiling once
reflections are allowed \cite{ST12}.

% ledger: T8, T9, N1, N2, A15, R7; source 2305.17743 L43, L45; 2303.10798 L253
Clause (4) is the mechanism. After one ambient isometry every tile of every tiling occupies one
of $24$ proper cubic frames at an integer position (\cref{sec:registration}); the seven-cube
chairs so placed group uniquely into $2$-scaled chairs, those into $4$-scaled chairs, and so on
without end (\cref{sec:hierarchy}). The hierarchy is unique at every level, and that is what
forces non-periodicity, as in the Spectre paper \cite{SMKGS24b}. We do not claim that the
supertiles of one tile exhaust space: as in the hat paper, tilings in which a tile sits at a
corner of its supertile at every level, along ``fault lines'', are not excluded
(\cref{sec:hierarchy}).

% ledger: D4, T12, M1, M2, N8; source README wording (proof submission)
\emph{Proof submission.} The finite parts of the argument, the enumerations of
\cref{sec:computer}, the written geometric lemmas, and the logical assembly from those lemmas to
\cref{thm:full} are theorems checked by the Lean~4 proof assistant, modulo the named compiler
hooks of \cref{app:lean}. The written lemmas concern the tube construction, feature companions,
global coverage and registration, carrier recovery, and small-collar realization. They are new,
written out in full, independently reviewed, and listed by name with the Lean theorems that
prove them in \cref{tab:hypotheses}; there are $19$ of them. Until the development and this
manuscript have been refereed we present the theorem as a proof submission.

\subsection{What an aperiodic monotile is, and how that found \texorpdfstring{$\solid$}{Q}}\label{sec:reading}

% ledger: H1, O5, A15; Lean registration_of_tiling, period_grid_from_registered_poses, period_halving
The proof below follows a reading of the phenomenon that we state first, because it is what
produced the solid. Consider a labelling $x\colon\Z^3\to A$ of the cubic lattice by a finite
alphabet. For a nonzero $p\in\Z^3$, $x$ is $p$-periodic exactly when it factors through
$\Z^3\to\Z^3/\langle p\rangle$; this quotient is infinite. For tilings by $\solid$ the connection
to labellings is
registration (\cref{sec:registration}): after one ambient isometry every tile of every tiling
sits in one of $24$ proper cubic frames at an integer position, and, because $\solid$ has no
self-isometry, every translational period of the tiling is then an integer vector
(\cref{sec:aperiodicity}). We use only this direction; the paper does not identify the periods of
a tiling with those of a labelling in any packaged form. For $\solid$, we prove existence and
exclude every nonzero translational period directly, using registration and period halving. The
obstruction is productive: the local pieces do fit together globally; what fails is fitting them
together periodically.

% ledger: H3, Q11; source 1608.07165 L80, L308; 2305.17743 L43
A hierarchy certifies aperiodicity only when it is forced. Grouping the tiles of a tiling into
supertiles proves nothing by itself: on the cubic grid an observer may group cells into
$2\times2\times2$ blocks in eight ways, and the grouping adds information the tiling does not
carry. What proves non-periodicity is a grouping that the tiling determines and that every
symmetry of the tiling therefore preserves. Then a symmetry of the tiling is a symmetry of the
grouped tiling at every level, so a translation that was a symmetry would survive coarsening to
every scale; if each coarsening halves it, it was never there. This is the recognizability,
unique-composition and period-halving mechanism used here \cite{GS18,SMKGS24b}, and we use it as
a design target.

% ledger: H4; source packets P3-P5 SOURCE_AUDIT (structural interpretation); history/cascade/findings.md L44
With marked tiles the rule that decides which neighbours are legal is external to the pieces; a
geometric monotile has to put that rule into what can physically fit. For a marked system with a
forced hierarchy to become one bare solid, three things must hold: a solid realizing the marked
system exists; every tiling by the solid decodes into a legal configuration of the marked
system; and a period of the tiling is a period of the decoded configuration. The second is the
one that is easy to lose. It is not enough that the intended tilings survive: every tiling
admitted by the bare solid must decode into the intended admissible marked system. One shape does
not mean one role: copies of a
single solid occupy different positions in the forced structure, and the diversity that a set of
marked tiles carries in its catalogue is carried here by the relations between congruent copies.

% ledger: H5, R7; Lean parent_atlas_eq_fine
This reading yields a finite test. Write the legal contacts between neighbouring copies of the
solid as a language of contacts, and decode a tiling into the tiling by its parents. The halving
argument needs the parent tiling to be again a legal tiling of the same solid, that is, the
language of contacts among decoded parents must lie inside the language of contacts among the
tiles themselves. Equality of the two languages is the form we check, and it is sufficient for
the argument, not necessary: a strictly smaller decoded language would also do. When the decoded
language is strictly larger the argument fails, and whether a periodic parent tiling then exists
must be checked case by case; for the first bare solid we built, it did (\cref{sec:found}). For
$\solid$ the test is one finite identity, verified by machine: the halved parent atlas equals the
fine atlas of $44$ legal contacts (\cref{sec:designtest}).

% ledger: H2, H6, H11 (pointer sentence only; framework vocabulary confined to sec:found)
The reading was reached with the Six Birds emergence calculus \cite{Tsi26F6,Tsi26F4,Tsi26NDO};
\cref{sec:found} quotes its statement of the phenomenon and relates it, clause by clause and at
each clause's level of verification, to the theorems proved here. Nothing in that section is used
in the proofs.

% ledger: H11, N7, O8, R10, A7, A11, A15, R6, R7, T13; source 2303.10798 L51; figure roadmap
The proof is the discharge of that reading, and \cref{fig:roadmap} shows its shape. Existence
comes from a substitution: eight rotated copies of the chair form a doubled chair, and iterating
the dissection gives a registered tiling of the chair which the features accept
(\cref{sec:finding}). Universality comes from the features: each pyramid forces a companion
pyramid of the opposite polarity on a neighbouring tile, the companion forces the neighbour's pose
into a discrete set, and an exhaustive census reduces that set to the $44$ legal contacts
(\cref{sec:companions}); from there every tiling is registered (\cref{sec:registration}), every
registered tiling has a unique parent tiling at every level (\cref{sec:hierarchy}), and a
translational period would have to halve at every level, hence vanish (\cref{sec:aperiodicity}).
As in the hat paper, the computer-assisted census is needed only to show that all tilings follow
the hierarchy; it is not needed to show that a tiling exists.

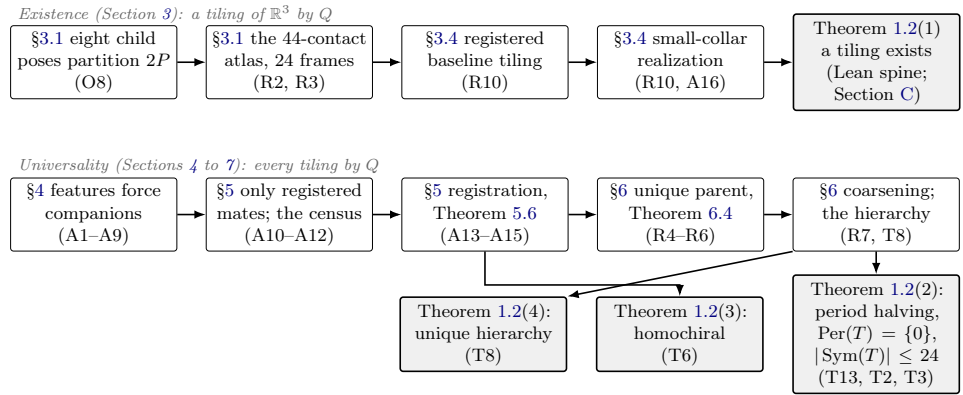
\begin{figure}[ht]
  \centering
  \resizebox{\linewidth}{!}{\input{fig_roadmap}}
  \caption{Roadmap of the proof. Each box names its section and the rows of the claims ledger
  (\cref{app:lean}) it rests on.}
  \label{fig:roadmap}
\end{figure}

\subsection{Computer assistance and verification}\label{sec:computer}

% ledger: M1-M5, M7, T12, Q14, Q20, R2, R4, R6, R7, A11, L1; source 2303.10798 L51, L158; 1808.07768 L51, L253; 1506.06492 L94
Several steps of the proof are exhaustive finite enumerations, listed in \cref{tab:census}: the
$2{,}388$ candidate contact poses of which exactly $44$ are legal; the $33$ first shells and $28$
parent conflicts behind the unique-parent theorem; the $697\to116\to44$ parent contacts of the
coarsening; the companion census of $6{,}862$ mate poses reduced to $44$ through $299{,}975$
collision boxes; and the $168$-label substitution with its primitivity exponent and modular
coincidence depth. The companion census, the largest of these, was checked by two independent
Python implementations. Every enumeration is stated as a
theorem in Lean~4 and checked by its kernel: those that decide directly by \lean{decide}, and
those too large for the kernel's evaluator by \lean{native\_decide}, which adds, per theorem, one
named compiler trust hook that we list in \cref{app:lean}. No floating-point number enters any
checked statement: certificates are integers and exact rationals, the Python replay uses exact
arithmetic, and Lean decides integer literals. The logical assembly from the enumerations and the
named written lemmas to \cref{thm:full} is itself a Lean theorem, and the written lemmas of
\cref{app:solidchecks,sec:finding,sec:companions,sec:registration} are themselves Lean theorems, so that
the main theorem is kernel-checked modulo the named compiler hooks listed in \cref{app:lean}; the
written proofs remain as exposition. Negative
controls (must-fail variants of the inputs and of the formal statements, each checked to fail as
expected) accompany the positive build. A one-command replay of the enumerations from the Python
standard library runs in about a minute (\cref{app:data}). The only comparable formalization we
know of is Myers's staging repository for the hat and Spectre, still in progress
\cite{Mye24lean,Lean4}. \emph{Code.} The repository is given in \cref{app:data}; its release tag is
\texttt{paper-v1} (intended; tag creation and licensing remain pending owner actions).

\subsection{Terminology}\label{sec:terminology}

% ledger: D1, D2, D3, D9; source 2303.10798 L25, L57-L58, L61; 1003.4909 L25
Terminology follows Grünbaum and Shephard as the hat paper uses it \cite{GS16,SMKGS24}, with
``closed topological disk'' replaced by ``closed topological $3$-ball''. A tile is a closed set;
a tiling by a set of tiles is a collection of images of tiles under isometries with pairwise
disjoint interiors whose union is the whole space; the set, or the single tile, \emph{admits} the
tiling. A tiling is monohedral if all its tiles are congruent, congruences including mirror
reflections, and locally finite if every ball meets only finitely many tiles. A \emph{monotile}
is a tile admitting a monohedral tiling; adapting the hat paper's reservation of the terms to the
present setting, an \emph{aperiodic monotile} or \emph{einstein} here is a closed topological
$3$-ball admitting tilings but only non-periodic ones, by geometry alone, without any
non-geometric matching rule constraining tile adjacencies. \emph{Registered} tilings, the \emph{atlas} of legal contacts,
and the \emph{features}, \emph{panels} and \emph{carrier} of $\solid$ are our own terms, defined
in \cref{sec:solid,sec:finding}; a \emph{supertile of level $n$} is a $2^n$-scaled chair in a
registered pose (\cref{sec:hierarchy}). Here a patch means a finite subfamily of a packing;
unlike the hat paper's disk-patch convention \cite{SMKGS24}, we impose no topological-ball
condition. A $1$-corona consists of the tiles touching the chosen tile \cite{Fle10}. The notation
is collected in \cref{app:notation}, \cref{tab:notation}.

% ledger: D1 (symmetry), Q3, T5, Q5, O5; source 2303.10798 L34, L64; 2409.15880 L43-L45, L73-L76
We quote the hat paper's definitions. ``The symmetry group of a tiling is the group of those
isometries that act as a permutation on the tiles of the tiling. A tiling is weakly periodic if
its symmetry group has an element of infinite order \dots\ A tiling is strongly periodic if the
symmetry group has a discrete subgroup with cocompact action on the space tiled \dots\ A set of
tiles (or a single tile) is weakly aperiodic if it admits a tiling but does not admit a strongly
periodic tiling, and strongly aperiodic if it admits a tiling but does not admit a weakly periodic
tiling.'' And: ``Following Mozes [Moz97], we say a set of tiles is strongly aperiodic if it admits
tilings but none with any infinite cyclic symmetry'' \cite{SMKGS24,Moz97}. Since a finite group
has no element of infinite order, ``strongly aperiodic'' in this sense is what
\cref{thm:full}(2) gives. In the finer taxonomy of Coulbois, Gajardo, Guillon
and Lutfalla \cite[\S2.1]{CGGL24} a tile all of whose tilings have finite symmetry group is
\emph{mildly} aperiodic, and \emph{strongly} aperiodic is reserved for trivial stabilizers; in
their terms $\solid$ is mildly aperiodic, and it is not strongly aperiodic in that stricter
sense, since an explicit tiling has symmetry group of order eight (\cref{sec:aperiodicity}).
They also distinguish the symmetry group of a tiling from the
stabilizer of its family of placements; the two coincide when the tile has no nonidentity self-isometry, as
$\solid$ has none (\cref{app:asymmetry}), and we write $\Sym(T)$ throughout.

% ledger: D5, Q10; source 2305.17743 L22, L43
From the Spectre paper we take, verbatim, that ``a monohedral tiling $\mathcal{T}$ is a
homochiral tiling if for every pair $T_i,T_j\in\mathcal{T}$, there is an orientation-preserving
isometry mapping $T_i$ to $T_j$'' and that a strictly chiral aperiodic monotile is ``a tile that
admits only homochiral non-periodic tilings''; and that ``a set of tiles is hierarchical if, in
every tiling admitted by those tiles, every tile is nested within an infinite hierarchy of
ever-larger supertiles. If these hierarchies are uniquely determined, then the tilings that
contain them must be non-periodic'' \cite{SMKGS24b}. Fault lines are defined in
\cref{sec:hierarchy}, following the hat paper.

%% file: fig_roadmap.tex
% fig_roadmap.tex — Fig. 1.2, proof roadmap (FIGURES.md 1.2; hat Figs 1.2/1.3 role).
% Two flowcharts: existence (top) and universality (bottom). Each box carries its section and
% the CLAIMS_LEDGER.md ids it stands on; tiers are those of the ledger rows at the pinned commit.
\begin{tikzpicture}[
  node distance=4mm and 5mm,
  box/.style={draw, rounded corners=1.5pt, align=center, font=\footnotesize, inner sep=3pt,
              minimum height=9mm, text width=27mm},
  goal/.style={box, thick, fill=black!6},
  lab/.style={font=\scriptsize\itshape, text=black!60},
  arr/.style={-{Latex[length=2mm]}, thick}
]
  % --- existence ---
  \node[lab] (exlab) {Existence (\cref{sec:finding}): a tiling of $\R^3$ by $\solid$};
  \node[box, below=2mm of exlab.west, anchor=north west] (e1)
    {\S\ref{sec:substitution} eight child poses partition $2\carrier$ \\ (O8)};
  \node[box, right=of e1] (e2)
    {\S\ref{sec:substitution} the $44$-contact atlas, $24$ frames \\ (R2, R3)};
  \node[box, right=of e2] (e3)
    {\S\ref{sec:existence} registered baseline tiling \\ (R10)};
  \node[box, right=of e3] (e4)
    {\S\ref{sec:existence} small-collar realization \\ (R10, A16)};
  \node[goal, right=of e4] (e5)
    {\cref{thm:full}(1) \\ a tiling exists (Lean spine; \cref{app:lean})};
  \draw[arr] (e1) -- (e2); \draw[arr] (e2) -- (e3); \draw[arr] (e3) -- (e4); \draw[arr] (e4) -- (e5);

  % --- universality ---
  \node[lab, below=9mm of e1.south west, anchor=north west] (unlab)
    {Universality (\cref{sec:companions,sec:registration,sec:hierarchy,sec:aperiodicity}): every tiling by $\solid$};
  \node[box, below=2mm of unlab.west, anchor=north west] (u1)
    {\S\ref{sec:companions} features force companions \\ (A1--A9)};
  \node[box, right=of u1] (u2)
    {\S\ref{sec:registration} only registered mates; the census \\ (A10--A12)};
  \node[box, right=of u2] (u3)
    {\S\ref{sec:registration} registration, \cref{thm:registration} \\ (A13--A15)};
  \node[box, right=of u3] (u4)
    {\S\ref{sec:hierarchy} unique parent, \cref{thm:parent} \\ (R4--R6)};
  \node[box, right=of u4] (u5)
    {\S\ref{sec:hierarchy} coarsening; the hierarchy \\ (R7, T8)};
  \draw[arr] (u1) -- (u2); \draw[arr] (u2) -- (u3); \draw[arr] (u3) -- (u4); \draw[arr] (u4) -- (u5);

  \node[goal, below=of u5] (u6)
    {\cref{thm:full}(2): period halving, $\Per(T)=\{0\}$, $|\Sym(T)|\le 24$ \\ (T13, T2, T3)};
  \node[goal, left=of u6] (u7)
    {\cref{thm:full}(3): homochiral \\ (T6)};
  \node[goal, left=of u7] (u8)
    {\cref{thm:full}(4): unique hierarchy \\ (T8)};
  \draw[arr] (u5) -- (u6);
  \draw[arr] (u3.south) |- ([yshift=2mm]u7.north) -- (u7.north);
  \draw[arr] (u5.south west) -- (u8.north east);
\end{tikzpicture}

%% file: sec2_solid.tex
\section{The solid Chair44 (R44), \texorpdfstring{$\solid$}{Q}}\label{sec:solid}

\subsection{Start with the chair}

% ledger: O1, O9, O6, O12 (Figure 3), R1; source solid/r44_solid.json and solid/native_panels.csv
Take a $2\times2\times2$ block of unit cubes and remove one corner cube.
The remaining seven cubes form the \emph{carrier}, $\carrier$. Chair44 is this
chair with eight small bumps or dents on each exposed unit square. We call those
squares \emph{panels}; there are $24$ in all, including the three inside the notch.
\Cref{fig:carrier} gives the complete panel recipe. Its small signed numbers tell
you which feature to put at each position: $+j$ means a pyramid pointing out of
the solid, and $-j$ means a pyramid-shaped dent of the same depth. A bump and a
matching dent have equal magnitudes and opposite signs.

Fix coordinates by taking the removed cube to be $[1,2]^3$. Thus
\[
  \carrier=\bigcup_{a\in\{0,1\}^3,\ a\ne(1,1,1)}\bigl(a+[0,1]^3\bigr).
\]
The six outer faces contribute $4+4+4+3+3+3=21$ panels; the notch contributes
three more. The large numbers $0,\ldots,23$ in the diagram identify panels,
not heights. They retain the numbering of the exact solid file.

\subsection{Read the panel diagram}

% ledger: O1, O6, R1; source solid/native_panels.csv; coordinate conventions of the existing panel diagram
Each of the six diagrams in \cref{fig:carrier} collects panels with one outward
normal. The two axes named above a diagram increase to the right and upwards,
respectively. For example, ``normal $-y$; axes $x,z$'' means that $x$ increases
rightwards and $z$ upwards. The layouts use these global coordinates. To compare
them with a physical model, look at a panel from outside with the second named
axis pointing up. The diagrams for normals $+x$, $-y$ and $+z$ then look as drawn;
those for $-x$, $+y$ and $-z$ look reflected left to right. This change of viewpoint
does not change which feature belongs at each coordinate.

To locate a panel on the chair, read its position in that diagram's $2\times2$
grid. A left or bottom square has centre coordinate $1/2$ along the corresponding
axis; a right or top square has centre coordinate $3/2$. The remaining coordinate
is $0$ for a negative outward normal and $2$ for a positive one, except for the
three notch panels: panel~$4$ lies at $x=1$, panel~$14$ at $y=1$, and panel~$23$
at $z=1$. This locates every panel and fixes the in-panel positions of all eight features.
In particular, the square in the upper right of each positive-normal diagram is
an inset notch panel, not a face at the missing outer corner.

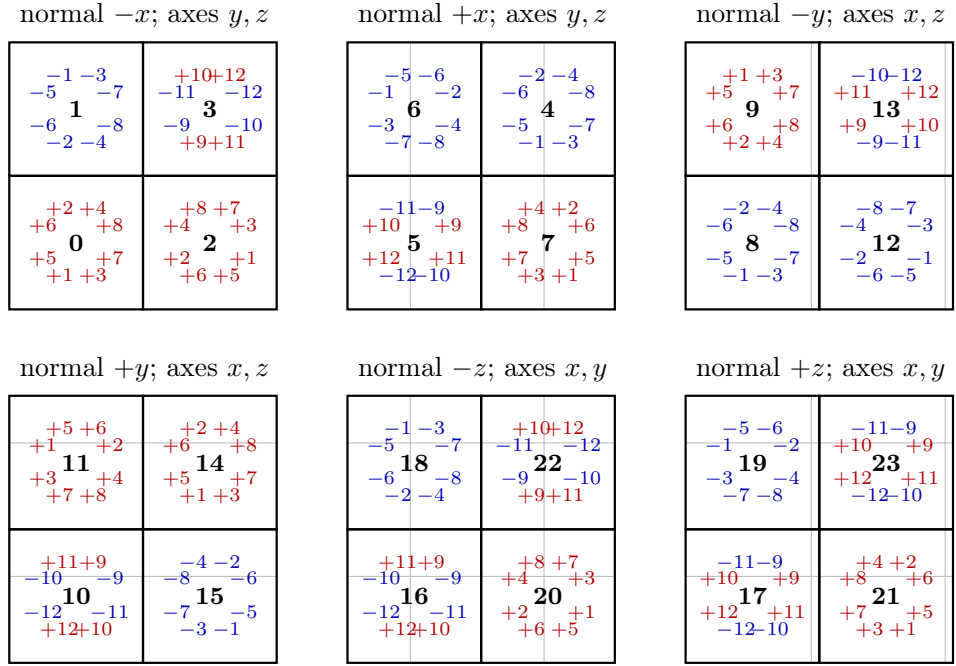
\begin{figure}[htbp]
  \centering
  \resizebox{\linewidth}{!}{\input{generated/fig_carrier_panels}}
  \caption{Complete panel recipe for Chair44, with the original panel IDs and signed
  heights. In each grid the first named axis increases rightwards and the second
  upwards. The large central number is the panel ID; the eight small signed numbers
  are feature coefficients. Panels $4,14,23$ are the three inset faces of the notch.
  Positive coefficients denote bumps and negative coefficients denote dents.
  The panels are laid out in global coordinates; the diagrams are not exterior views.}
  \label{fig:carrier}
\end{figure}

You can make a marked paper model by copying these signed numbers onto the
corresponding squares of a seven-cube chair, keeping the indicated coordinate
orientation. That model is a guide to the shape. The actual tile has no printed
labels or imposed matching rules: the pyramids themselves determine what can fit.
When two panels meet face to face, their outward normals point in opposite
directions. Compare the features at the points of space where they coincide:
at each such point a bump must meet a dent of the same magnitude.

\subsection{From the labels to the solid}

% ledger: O1, O6, O9, O12, R1; source solid/r44_solid.json and native_panels.csv; Lean solid_mesh_exact
All panels use the same eight feature positions. Relative to the panel centre,
in the order of the two axes named in the diagram, these are
\[
  \bigl(\pm\tfrac18,\pm\tfrac14\bigr),\qquad
  \bigl(\pm\tfrac14,\pm\tfrac18\bigr).
\]
Each feature has an axis-parallel square base of side $1/50$ (half-width
$\eta=1/100$). Its signed height along the outward normal is $a/10000$, where
$a$ is the small signed number at that position. The rest of the panel stays flat.
These instructions specify all $192$ features without requiring a mesh file.

\emph{For example, read panel~$13$.} It is the upper-right square of the $-y$
diagram, so its centre is $(3/2,0,3/2)$ and its in-panel axes are $x,z$.
The coefficient at offset $(-1/8,-1/4)$ is $-9$. Its base centre is therefore
$(11/8,0,5/4)$, and its apex is $(11/8,9/10000,5/4)$: since the outward normal
is $-y$, the negative sign cuts a dent into the solid, towards positive $y$.
The same panel's $+9$ at offset $(-1/4,-1/8)$ points outwards instead.
\Cref{fig:panel} shows both the plan and oblique views of this panel.

\begin{figure}[ht]
  \centering
  \includegraphics[width=\linewidth]{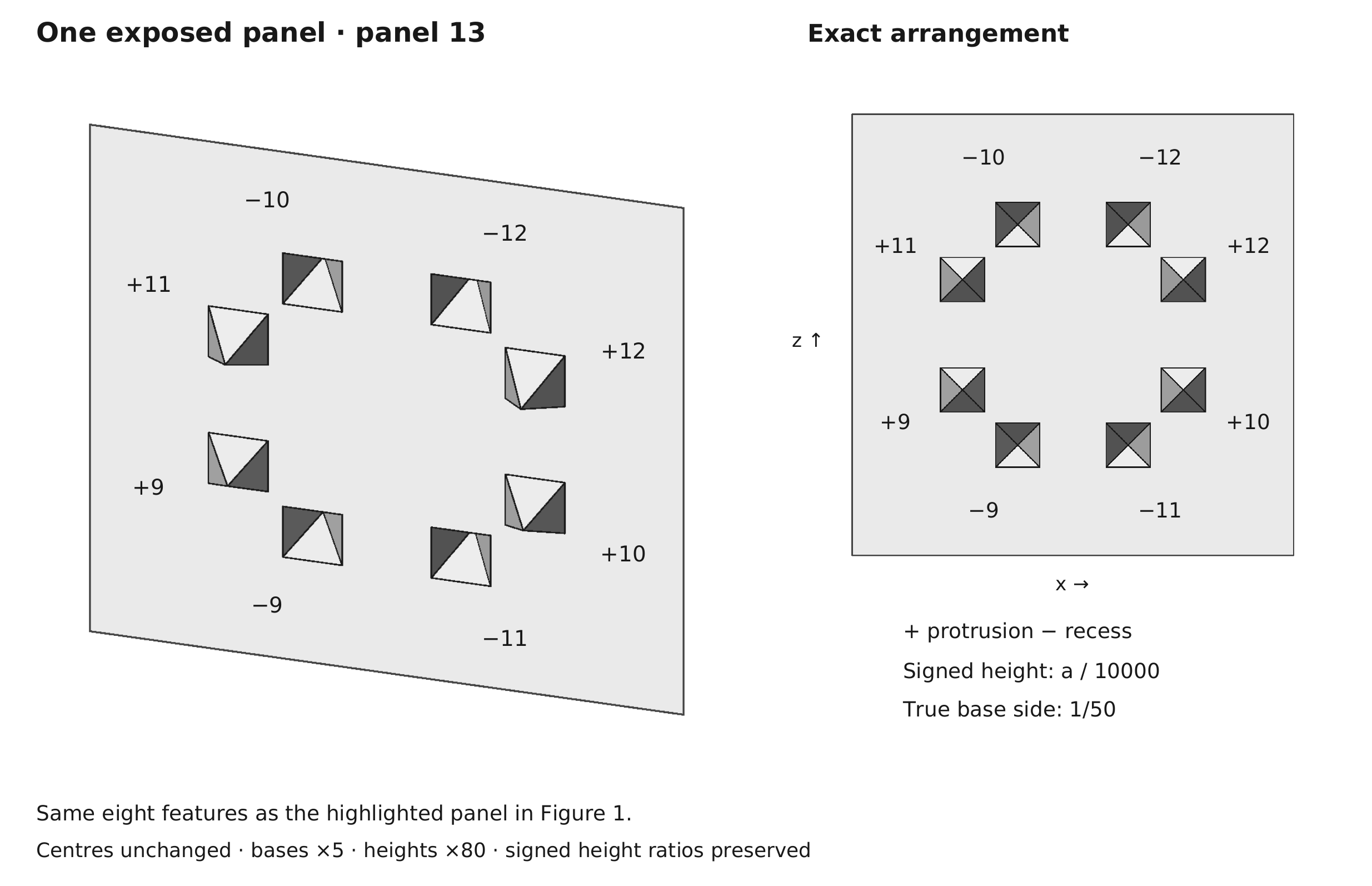}
  \caption{Panel~$13$, the panel outlined in \cref{fig:tile}, with its eight
  square-pyramid features. The outward normal is $-y$ and the plan view uses
  global axes $x,z$. Labels are the same signed coefficients as in
  \cref{fig:carrier}. For visibility, bases are enlarged $\times5$ and signed
  heights $\times80$ about their unchanged centres; the true base side is
  $1/50$ and the true signed height is $a/10000$.}
  \label{fig:panel}
\end{figure}

\subsection{Why these features?}

% ledger: R1; Lean contact_closure_30, profile_canonical
The bare chair can tile periodically. The features distinguish arrangements that
its flat panels alone cannot distinguish. Their heights were chosen so that the
contacts needed by the eight-child construction of \cref{sec:substitution} fit,
including the contacts created when that construction is repeated. Each required
pair of coincident features imposes an equation $a_u=-a_v$. The resulting equations
leave twelve independent signed groups; assigning them the distinct magnitudes
$1,\ldots,12$ gives the height table in \cref{fig:carrier}. The exact reconstruction
of that table is in \cref{app:profile}. Proving that \emph{every} tiling follows
the intended hierarchy requires the registration and recognition arguments in
\cref{sec:registration,sec:hierarchy}; the panel recipe alone does not prove it.

% ledger: O2, O3, O4, O5, O6, O7, O9, A16, T14; supporting geometric checks retained in appE_geometry.tex
The numerical dimensions make the features small and separate while retaining
the distinct angles used to force registration. Their purpose in the proof is
explained in \cref{app:angles,app:tubes}; readers following the construction can
keep the single height rule $a/10000$ in mind. The resulting solid has rational
vertices, volume $7$ and the topology of a closed $3$-ball, and no nonidentity
self-isometry. Its volume is unchanged because the bump and dent volumes cancel.
The topology, asymmetry and mesh checks are recorded in \cref{app:solidchecks}.
The parameters can also vary within the stated open conditions
(\cref{app:family}); the explicit solid described here uses the fixed values above.

%% file: generated/fig_carrier_panels.tex
% generated by paper/scripts/carrier_panels.py from solid/native_panels.csv
\begin{tikzpicture}[x=1cm,y=1cm,font=\scriptsize]
  \node[anchor=south west,font=\small] at (0.0,3.4499999999999997) {normal $-x$; axes $y,z$};
  \draw[gray!50] (0.0,0.0) grid[step=1.7] (3.4,3.4);
  \draw[thick] (0.0,0.0) rectangle (1.7,1.7);
  \node[font=\footnotesize\bfseries] at (0.85,0.85) {0};
  \node[font=\tiny,text=red!70!black] at (0.6375,0.425) {$+1$};
  \node[font=\tiny,text=red!70!black] at (0.6375,1.275) {$+2$};
  \node[font=\tiny,text=red!70!black] at (1.0625,0.425) {$+3$};
  \node[font=\tiny,text=red!70!black] at (1.0625,1.275) {$+4$};
  \node[font=\tiny,text=red!70!black] at (0.425,0.6375) {$+5$};
  \node[font=\tiny,text=red!70!black] at (0.425,1.0625) {$+6$};
  \node[font=\tiny,text=red!70!black] at (1.275,0.6375) {$+7$};
  \node[font=\tiny,text=red!70!black] at (1.275,1.0625) {$+8$};
  \draw[thick] (0.0,1.7) rectangle (1.7,3.4);
  \node[font=\footnotesize\bfseries] at (0.85,2.55) {1};
  \node[font=\tiny,text=blue!70!black] at (0.6375,2.125) {$-2$};
  \node[font=\tiny,text=blue!70!black] at (0.6375,2.9749999999999996) {$-1$};
  \node[font=\tiny,text=blue!70!black] at (1.0625,2.125) {$-4$};
  \node[font=\tiny,text=blue!70!black] at (1.0625,2.9749999999999996) {$-3$};
  \node[font=\tiny,text=blue!70!black] at (0.425,2.3375) {$-6$};
  \node[font=\tiny,text=blue!70!black] at (0.425,2.7625) {$-5$};
  \node[font=\tiny,text=blue!70!black] at (1.275,2.3375) {$-8$};
  \node[font=\tiny,text=blue!70!black] at (1.275,2.7625) {$-7$};
  \draw[thick] (1.7,0.0) rectangle (3.4,1.7);
  \node[font=\footnotesize\bfseries] at (2.55,0.85) {2};
  \node[font=\tiny,text=red!70!black] at (2.3375,0.425) {$+6$};
  \node[font=\tiny,text=red!70!black] at (2.3375,1.275) {$+8$};
  \node[font=\tiny,text=red!70!black] at (2.7625,0.425) {$+5$};
  \node[font=\tiny,text=red!70!black] at (2.7625,1.275) {$+7$};
  \node[font=\tiny,text=red!70!black] at (2.125,0.6375) {$+2$};
  \node[font=\tiny,text=red!70!black] at (2.125,1.0625) {$+4$};
  \node[font=\tiny,text=red!70!black] at (2.9749999999999996,0.6375) {$+1$};
  \node[font=\tiny,text=red!70!black] at (2.9749999999999996,1.0625) {$+3$};
  \draw[thick] (1.7,1.7) rectangle (3.4,3.4);
  \node[font=\footnotesize\bfseries] at (2.55,2.55) {3};
  \node[font=\tiny,text=red!70!black] at (2.3375,2.125) {$+9$};
  \node[font=\tiny,text=red!70!black] at (2.3375,2.9749999999999996) {$+10$};
  \node[font=\tiny,text=red!70!black] at (2.7625,2.125) {$+11$};
  \node[font=\tiny,text=red!70!black] at (2.7625,2.9749999999999996) {$+12$};
  \node[font=\tiny,text=blue!70!black] at (2.125,2.3375) {$-9$};
  \node[font=\tiny,text=blue!70!black] at (2.125,2.7625) {$-11$};
  \node[font=\tiny,text=blue!70!black] at (2.9749999999999996,2.3375) {$-10$};
  \node[font=\tiny,text=blue!70!black] at (2.9749999999999996,2.7625) {$-12$};
  \node[anchor=south west,font=\small] at (4.3,3.4499999999999997) {normal $+x$; axes $y,z$};
  \draw[gray!50] (4.3,0.0) grid[step=1.7] (7.699999999999999,3.4);
  \draw[thick] (6.0,1.7) rectangle (7.7,3.4);
  \node[font=\footnotesize\bfseries] at (6.85,2.55) {4};
  \node[font=\tiny,text=blue!70!black] at (6.6375,2.125) {$-1$};
  \node[font=\tiny,text=blue!70!black] at (6.6375,2.9749999999999996) {$-2$};
  \node[font=\tiny,text=blue!70!black] at (7.0625,2.125) {$-3$};
  \node[font=\tiny,text=blue!70!black] at (7.0625,2.9749999999999996) {$-4$};
  \node[font=\tiny,text=blue!70!black] at (6.425,2.3375) {$-5$};
  \node[font=\tiny,text=blue!70!black] at (6.425,2.7625) {$-6$};
  \node[font=\tiny,text=blue!70!black] at (7.275,2.3375) {$-7$};
  \node[font=\tiny,text=blue!70!black] at (7.275,2.7625) {$-8$};
  \draw[thick] (4.3,0.0) rectangle (6.0,1.7);
  \node[font=\footnotesize\bfseries] at (5.1499999999999995,0.85) {5};
  \node[font=\tiny,text=blue!70!black] at (4.9375,0.425) {$-12$};
  \node[font=\tiny,text=blue!70!black] at (4.9375,1.275) {$-11$};
  \node[font=\tiny,text=blue!70!black] at (5.3625,0.425) {$-10$};
  \node[font=\tiny,text=blue!70!black] at (5.3625,1.275) {$-9$};
  \node[font=\tiny,text=red!70!black] at (4.725,0.6375) {$+12$};
  \node[font=\tiny,text=red!70!black] at (4.725,1.0625) {$+10$};
  \node[font=\tiny,text=red!70!black] at (5.574999999999999,0.6375) {$+11$};
  \node[font=\tiny,text=red!70!black] at (5.574999999999999,1.0625) {$+9$};
  \draw[thick] (4.3,1.7) rectangle (6.0,3.4);
  \node[font=\footnotesize\bfseries] at (5.1499999999999995,2.55) {6};
  \node[font=\tiny,text=blue!70!black] at (4.9375,2.125) {$-7$};
  \node[font=\tiny,text=blue!70!black] at (4.9375,2.9749999999999996) {$-5$};
  \node[font=\tiny,text=blue!70!black] at (5.3625,2.125) {$-8$};
  \node[font=\tiny,text=blue!70!black] at (5.3625,2.9749999999999996) {$-6$};
  \node[font=\tiny,text=blue!70!black] at (4.725,2.3375) {$-3$};
  \node[font=\tiny,text=blue!70!black] at (4.725,2.7625) {$-1$};
  \node[font=\tiny,text=blue!70!black] at (5.574999999999999,2.3375) {$-4$};
  \node[font=\tiny,text=blue!70!black] at (5.574999999999999,2.7625) {$-2$};
  \draw[thick] (6.0,0.0) rectangle (7.7,1.7);
  \node[font=\footnotesize\bfseries] at (6.85,0.85) {7};
  \node[font=\tiny,text=red!70!black] at (6.6375,0.425) {$+3$};
  \node[font=\tiny,text=red!70!black] at (6.6375,1.275) {$+4$};
  \node[font=\tiny,text=red!70!black] at (7.0625,0.425) {$+1$};
  \node[font=\tiny,text=red!70!black] at (7.0625,1.275) {$+2$};
  \node[font=\tiny,text=red!70!black] at (6.425,0.6375) {$+7$};
  \node[font=\tiny,text=red!70!black] at (6.425,1.0625) {$+8$};
  \node[font=\tiny,text=red!70!black] at (7.275,0.6375) {$+5$};
  \node[font=\tiny,text=red!70!black] at (7.275,1.0625) {$+6$};
  \node[anchor=south west,font=\small] at (8.6,3.4499999999999997) {normal $-y$; axes $x,z$};
  \draw[gray!50] (8.6,0.0) grid[step=1.7] (12.0,3.4);
  \draw[thick] (8.6,0.0) rectangle (10.299999999999999,1.7);
  \node[font=\footnotesize\bfseries] at (9.45,0.85) {8};
  \node[font=\tiny,text=blue!70!black] at (9.237499999999999,0.425) {$-1$};
  \node[font=\tiny,text=blue!70!black] at (9.237499999999999,1.275) {$-2$};
  \node[font=\tiny,text=blue!70!black] at (9.6625,0.425) {$-3$};
  \node[font=\tiny,text=blue!70!black] at (9.6625,1.275) {$-4$};
  \node[font=\tiny,text=blue!70!black] at (9.025,0.6375) {$-5$};
  \node[font=\tiny,text=blue!70!black] at (9.025,1.0625) {$-6$};
  \node[font=\tiny,text=blue!70!black] at (9.875,0.6375) {$-7$};
  \node[font=\tiny,text=blue!70!black] at (9.875,1.0625) {$-8$};
  \draw[thick] (8.6,1.7) rectangle (10.299999999999999,3.4);
  \node[font=\footnotesize\bfseries] at (9.45,2.55) {9};
  \node[font=\tiny,text=red!70!black] at (9.237499999999999,2.125) {$+2$};
  \node[font=\tiny,text=red!70!black] at (9.237499999999999,2.9749999999999996) {$+1$};
  \node[font=\tiny,text=red!70!black] at (9.6625,2.125) {$+4$};
  \node[font=\tiny,text=red!70!black] at (9.6625,2.9749999999999996) {$+3$};
  \node[font=\tiny,text=red!70!black] at (9.025,2.3375) {$+6$};
  \node[font=\tiny,text=red!70!black] at (9.025,2.7625) {$+5$};
  \node[font=\tiny,text=red!70!black] at (9.875,2.3375) {$+8$};
  \node[font=\tiny,text=red!70!black] at (9.875,2.7625) {$+7$};
  \draw[thick] (10.299999999999999,0.0) rectangle (11.999999999999998,1.7);
  \node[font=\footnotesize\bfseries] at (11.149999999999999,0.85) {12};
  \node[font=\tiny,text=blue!70!black] at (10.937499999999998,0.425) {$-6$};
  \node[font=\tiny,text=blue!70!black] at (10.937499999999998,1.275) {$-8$};
  \node[font=\tiny,text=blue!70!black] at (11.362499999999999,0.425) {$-5$};
  \node[font=\tiny,text=blue!70!black] at (11.362499999999999,1.275) {$-7$};
  \node[font=\tiny,text=blue!70!black] at (10.725,0.6375) {$-2$};
  \node[font=\tiny,text=blue!70!black] at (10.725,1.0625) {$-4$};
  \node[font=\tiny,text=blue!70!black] at (11.575,0.6375) {$-1$};
  \node[font=\tiny,text=blue!70!black] at (11.575,1.0625) {$-3$};
  \draw[thick] (10.299999999999999,1.7) rectangle (11.999999999999998,3.4);
  \node[font=\footnotesize\bfseries] at (11.149999999999999,2.55) {13};
  \node[font=\tiny,text=blue!70!black] at (10.937499999999998,2.125) {$-9$};
  \node[font=\tiny,text=blue!70!black] at (10.937499999999998,2.9749999999999996) {$-10$};
  \node[font=\tiny,text=blue!70!black] at (11.362499999999999,2.125) {$-11$};
  \node[font=\tiny,text=blue!70!black] at (11.362499999999999,2.9749999999999996) {$-12$};
  \node[font=\tiny,text=red!70!black] at (10.725,2.3375) {$+9$};
  \node[font=\tiny,text=red!70!black] at (10.725,2.7625) {$+11$};
  \node[font=\tiny,text=red!70!black] at (11.575,2.3375) {$+10$};
  \node[font=\tiny,text=red!70!black] at (11.575,2.7625) {$+12$};
  \node[anchor=south west,font=\small] at (0.0,-1.05) {normal $+y$; axes $x,z$};
  \draw[gray!50] (0.0,-4.5) grid[step=1.7] (3.4,-1.1);
  \draw[thick] (0.0,-4.5) rectangle (1.7,-2.8);
  \node[font=\footnotesize\bfseries] at (0.85,-3.65) {10};
  \node[font=\tiny,text=red!70!black] at (0.6375,-4.075) {$+12$};
  \node[font=\tiny,text=red!70!black] at (0.6375,-3.225) {$+11$};
  \node[font=\tiny,text=red!70!black] at (1.0625,-4.075) {$+10$};
  \node[font=\tiny,text=red!70!black] at (1.0625,-3.225) {$+9$};
  \node[font=\tiny,text=blue!70!black] at (0.425,-3.8625) {$-12$};
  \node[font=\tiny,text=blue!70!black] at (0.425,-3.4375) {$-10$};
  \node[font=\tiny,text=blue!70!black] at (1.275,-3.8625) {$-11$};
  \node[font=\tiny,text=blue!70!black] at (1.275,-3.4375) {$-9$};
  \draw[thick] (0.0,-2.8) rectangle (1.7,-1.0999999999999999);
  \node[font=\footnotesize\bfseries] at (0.85,-1.9499999999999997) {11};
  \node[font=\tiny,text=red!70!black] at (0.6375,-2.375) {$+7$};
  \node[font=\tiny,text=red!70!black] at (0.6375,-1.525) {$+5$};
  \node[font=\tiny,text=red!70!black] at (1.0625,-2.375) {$+8$};
  \node[font=\tiny,text=red!70!black] at (1.0625,-1.525) {$+6$};
  \node[font=\tiny,text=red!70!black] at (0.425,-2.1624999999999996) {$+3$};
  \node[font=\tiny,text=red!70!black] at (0.425,-1.7374999999999998) {$+1$};
  \node[font=\tiny,text=red!70!black] at (1.275,-2.1624999999999996) {$+4$};
  \node[font=\tiny,text=red!70!black] at (1.275,-1.7374999999999998) {$+2$};
  \draw[thick] (1.7,-2.8) rectangle (3.4,-1.0999999999999999);
  \node[font=\footnotesize\bfseries] at (2.55,-1.9499999999999997) {14};
  \node[font=\tiny,text=red!70!black] at (2.3375,-2.375) {$+1$};
  \node[font=\tiny,text=red!70!black] at (2.3375,-1.525) {$+2$};
  \node[font=\tiny,text=red!70!black] at (2.7625,-2.375) {$+3$};
  \node[font=\tiny,text=red!70!black] at (2.7625,-1.525) {$+4$};
  \node[font=\tiny,text=red!70!black] at (2.125,-2.1624999999999996) {$+5$};
  \node[font=\tiny,text=red!70!black] at (2.125,-1.7374999999999998) {$+6$};
  \node[font=\tiny,text=red!70!black] at (2.9749999999999996,-2.1624999999999996) {$+7$};
  \node[font=\tiny,text=red!70!black] at (2.9749999999999996,-1.7374999999999998) {$+8$};
  \draw[thick] (1.7,-4.5) rectangle (3.4,-2.8);
  \node[font=\footnotesize\bfseries] at (2.55,-3.65) {15};
  \node[font=\tiny,text=blue!70!black] at (2.3375,-4.075) {$-3$};
  \node[font=\tiny,text=blue!70!black] at (2.3375,-3.225) {$-4$};
  \node[font=\tiny,text=blue!70!black] at (2.7625,-4.075) {$-1$};
  \node[font=\tiny,text=blue!70!black] at (2.7625,-3.225) {$-2$};
  \node[font=\tiny,text=blue!70!black] at (2.125,-3.8625) {$-7$};
  \node[font=\tiny,text=blue!70!black] at (2.125,-3.4375) {$-8$};
  \node[font=\tiny,text=blue!70!black] at (2.9749999999999996,-3.8625) {$-5$};
  \node[font=\tiny,text=blue!70!black] at (2.9749999999999996,-3.4375) {$-6$};
  \node[anchor=south west,font=\small] at (4.3,-1.05) {normal $-z$; axes $x,y$};
  \draw[gray!50] (4.3,-4.5) grid[step=1.7] (7.699999999999999,-1.1);
  \draw[thick] (4.3,-4.5) rectangle (6.0,-2.8);
  \node[font=\footnotesize\bfseries] at (5.1499999999999995,-3.65) {16};
  \node[font=\tiny,text=red!70!black] at (4.9375,-4.075) {$+12$};
  \node[font=\tiny,text=red!70!black] at (4.9375,-3.225) {$+11$};
  \node[font=\tiny,text=red!70!black] at (5.3625,-4.075) {$+10$};
  \node[font=\tiny,text=red!70!black] at (5.3625,-3.225) {$+9$};
  \node[font=\tiny,text=blue!70!black] at (4.725,-3.8625) {$-12$};
  \node[font=\tiny,text=blue!70!black] at (4.725,-3.4375) {$-10$};
  \node[font=\tiny,text=blue!70!black] at (5.574999999999999,-3.8625) {$-11$};
  \node[font=\tiny,text=blue!70!black] at (5.574999999999999,-3.4375) {$-9$};
  \draw[thick] (4.3,-2.8) rectangle (6.0,-1.0999999999999999);
  \node[font=\footnotesize\bfseries] at (5.1499999999999995,-1.9499999999999997) {18};
  \node[font=\tiny,text=blue!70!black] at (4.9375,-2.375) {$-2$};
  \node[font=\tiny,text=blue!70!black] at (4.9375,-1.525) {$-1$};
  \node[font=\tiny,text=blue!70!black] at (5.3625,-2.375) {$-4$};
  \node[font=\tiny,text=blue!70!black] at (5.3625,-1.525) {$-3$};
  \node[font=\tiny,text=blue!70!black] at (4.725,-2.1624999999999996) {$-6$};
  \node[font=\tiny,text=blue!70!black] at (4.725,-1.7374999999999998) {$-5$};
  \node[font=\tiny,text=blue!70!black] at (5.574999999999999,-2.1624999999999996) {$-8$};
  \node[font=\tiny,text=blue!70!black] at (5.574999999999999,-1.7374999999999998) {$-7$};
  \draw[thick] (6.0,-4.5) rectangle (7.7,-2.8);
  \node[font=\footnotesize\bfseries] at (6.85,-3.65) {20};
  \node[font=\tiny,text=red!70!black] at (6.6375,-4.075) {$+6$};
  \node[font=\tiny,text=red!70!black] at (6.6375,-3.225) {$+8$};
  \node[font=\tiny,text=red!70!black] at (7.0625,-4.075) {$+5$};
  \node[font=\tiny,text=red!70!black] at (7.0625,-3.225) {$+7$};
  \node[font=\tiny,text=red!70!black] at (6.425,-3.8625) {$+2$};
  \node[font=\tiny,text=red!70!black] at (6.425,-3.4375) {$+4$};
  \node[font=\tiny,text=red!70!black] at (7.275,-3.8625) {$+1$};
  \node[font=\tiny,text=red!70!black] at (7.275,-3.4375) {$+3$};
  \draw[thick] (6.0,-2.8) rectangle (7.7,-1.0999999999999999);
  \node[font=\footnotesize\bfseries] at (6.85,-1.9499999999999997) {22};
  \node[font=\tiny,text=red!70!black] at (6.6375,-2.375) {$+9$};
  \node[font=\tiny,text=red!70!black] at (6.6375,-1.525) {$+10$};
  \node[font=\tiny,text=red!70!black] at (7.0625,-2.375) {$+11$};
  \node[font=\tiny,text=red!70!black] at (7.0625,-1.525) {$+12$};
  \node[font=\tiny,text=blue!70!black] at (6.425,-2.1624999999999996) {$-9$};
  \node[font=\tiny,text=blue!70!black] at (6.425,-1.7374999999999998) {$-11$};
  \node[font=\tiny,text=blue!70!black] at (7.275,-2.1624999999999996) {$-10$};
  \node[font=\tiny,text=blue!70!black] at (7.275,-1.7374999999999998) {$-12$};
  \node[anchor=south west,font=\small] at (8.6,-1.05) {normal $+z$; axes $x,y$};
  \draw[gray!50] (8.6,-4.5) grid[step=1.7] (12.0,-1.1);
  \draw[thick] (8.6,-4.5) rectangle (10.299999999999999,-2.8);
  \node[font=\footnotesize\bfseries] at (9.45,-3.65) {17};
  \node[font=\tiny,text=blue!70!black] at (9.237499999999999,-4.075) {$-12$};
  \node[font=\tiny,text=blue!70!black] at (9.237499999999999,-3.225) {$-11$};
  \node[font=\tiny,text=blue!70!black] at (9.6625,-4.075) {$-10$};
  \node[font=\tiny,text=blue!70!black] at (9.6625,-3.225) {$-9$};
  \node[font=\tiny,text=red!70!black] at (9.025,-3.8625) {$+12$};
  \node[font=\tiny,text=red!70!black] at (9.025,-3.4375) {$+10$};
  \node[font=\tiny,text=red!70!black] at (9.875,-3.8625) {$+11$};
  \node[font=\tiny,text=red!70!black] at (9.875,-3.4375) {$+9$};
  \draw[thick] (8.6,-2.8) rectangle (10.299999999999999,-1.0999999999999999);
  \node[font=\footnotesize\bfseries] at (9.45,-1.9499999999999997) {19};
  \node[font=\tiny,text=blue!70!black] at (9.237499999999999,-2.375) {$-7$};
  \node[font=\tiny,text=blue!70!black] at (9.237499999999999,-1.525) {$-5$};
  \node[font=\tiny,text=blue!70!black] at (9.6625,-2.375) {$-8$};
  \node[font=\tiny,text=blue!70!black] at (9.6625,-1.525) {$-6$};
  \node[font=\tiny,text=blue!70!black] at (9.025,-2.1624999999999996) {$-3$};
  \node[font=\tiny,text=blue!70!black] at (9.025,-1.7374999999999998) {$-1$};
  \node[font=\tiny,text=blue!70!black] at (9.875,-2.1624999999999996) {$-4$};
  \node[font=\tiny,text=blue!70!black] at (9.875,-1.7374999999999998) {$-2$};
  \draw[thick] (10.299999999999999,-4.5) rectangle (11.999999999999998,-2.8);
  \node[font=\footnotesize\bfseries] at (11.149999999999999,-3.65) {21};
  \node[font=\tiny,text=red!70!black] at (10.937499999999998,-4.075) {$+3$};
  \node[font=\tiny,text=red!70!black] at (10.937499999999998,-3.225) {$+4$};
  \node[font=\tiny,text=red!70!black] at (11.362499999999999,-4.075) {$+1$};
  \node[font=\tiny,text=red!70!black] at (11.362499999999999,-3.225) {$+2$};
  \node[font=\tiny,text=red!70!black] at (10.725,-3.8625) {$+7$};
  \node[font=\tiny,text=red!70!black] at (10.725,-3.4375) {$+8$};
  \node[font=\tiny,text=red!70!black] at (11.575,-3.8625) {$+5$};
  \node[font=\tiny,text=red!70!black] at (11.575,-3.4375) {$+6$};
  \draw[thick] (10.299999999999999,-2.8) rectangle (11.999999999999998,-1.0999999999999999);
  \node[font=\footnotesize\bfseries] at (11.149999999999999,-1.9499999999999997) {23};
  \node[font=\tiny,text=blue!70!black] at (10.937499999999998,-2.375) {$-12$};
  \node[font=\tiny,text=blue!70!black] at (10.937499999999998,-1.525) {$-11$};
  \node[font=\tiny,text=blue!70!black] at (11.362499999999999,-2.375) {$-10$};
  \node[font=\tiny,text=blue!70!black] at (11.362499999999999,-1.525) {$-9$};
  \node[font=\tiny,text=red!70!black] at (10.725,-2.1624999999999996) {$+12$};
  \node[font=\tiny,text=red!70!black] at (10.725,-1.7374999999999998) {$+10$};
  \node[font=\tiny,text=red!70!black] at (11.575,-2.1624999999999996) {$+11$};
  \node[font=\tiny,text=red!70!black] at (11.575,-1.7374999999999998) {$+9$};
\end{tikzpicture}

%% file: sec3_finding.tex
\section{Finding \texorpdfstring{$\solid$}{Q}}\label{sec:finding}

\subsection{The substitution}\label{sec:substitution}

% ledger: O8, D3; source PROOFS_registered.md R§1; Lean children_partition_2P (T1)
The chair is a rep-tile: eight copies of $\carrier$ form the doubled chair $2\carrier$. Write a
\emph{frame} as a pair $(p,s)$ of a permutation of the coordinate axes and a sign vector, acting
by $G(x)_i=s_i x_{p_i}$; a \emph{pose} is a frame followed by a translation. The eight child poses
are listed in \cref{tab:children} and drawn in \cref{fig:cluster}. Every child frame is a proper
rotation, and the $56$ unit cubes of the eight children partition $2\carrier$ exactly, as sets of
integer cells (a Lean theorem by kernel \lean{decide}). Refinement acts on poses by
\[
  (G,t)\ \longmapsto\ \{(GH,\,2t+Gu) : (H,u)\ \text{a child pose}\},
\]
Thus refinement replaces the doubled support of a chair by eight unit-size chairs; iteration from
the native chair produces patches of $8$, $64$, $512,\dots$ chairs (\cref{fig:patch}).

\begin{table}[ht]
  \centering
  \caption{The eight child poses of the doubled chair (permutation $p$, signs $s$, translation
  $u$, in the coordinates of $2\carrier$). Every frame is a proper rotation.}
  \label{tab:children}
  \begin{tabular}{llll}
    \toprule
    child & $p$ & $s$ & $u$ \\
    \midrule
    000 & $(0,1,2)$ & $(+,+,+)$ & $(0,0,0)$ \\
    001 & $(1,0,2)$ & $(+,+,-)$ & $(0,0,4)$ \\
    010 & $(0,2,1)$ & $(+,-,+)$ & $(0,4,0)$ \\
    011 & $(2,0,1)$ & $(+,-,-)$ & $(0,4,4)$ \\
    100 & $(2,1,0)$ & $(-,+,+)$ & $(4,0,0)$ \\
    101 & $(1,2,0)$ & $(-,+,-)$ & $(4,0,4)$ \\
    110 & $(0,1,2)$ & $(-,-,+)$ & $(4,4,0)$ \\
    central & $(0,1,2)$ & $(+,+,+)$ & $(1,1,1)$ \\
    \bottomrule
  \end{tabular}
\end{table}

\begin{figure}[ht]
  \centering
  \includegraphics[width=0.52\linewidth]{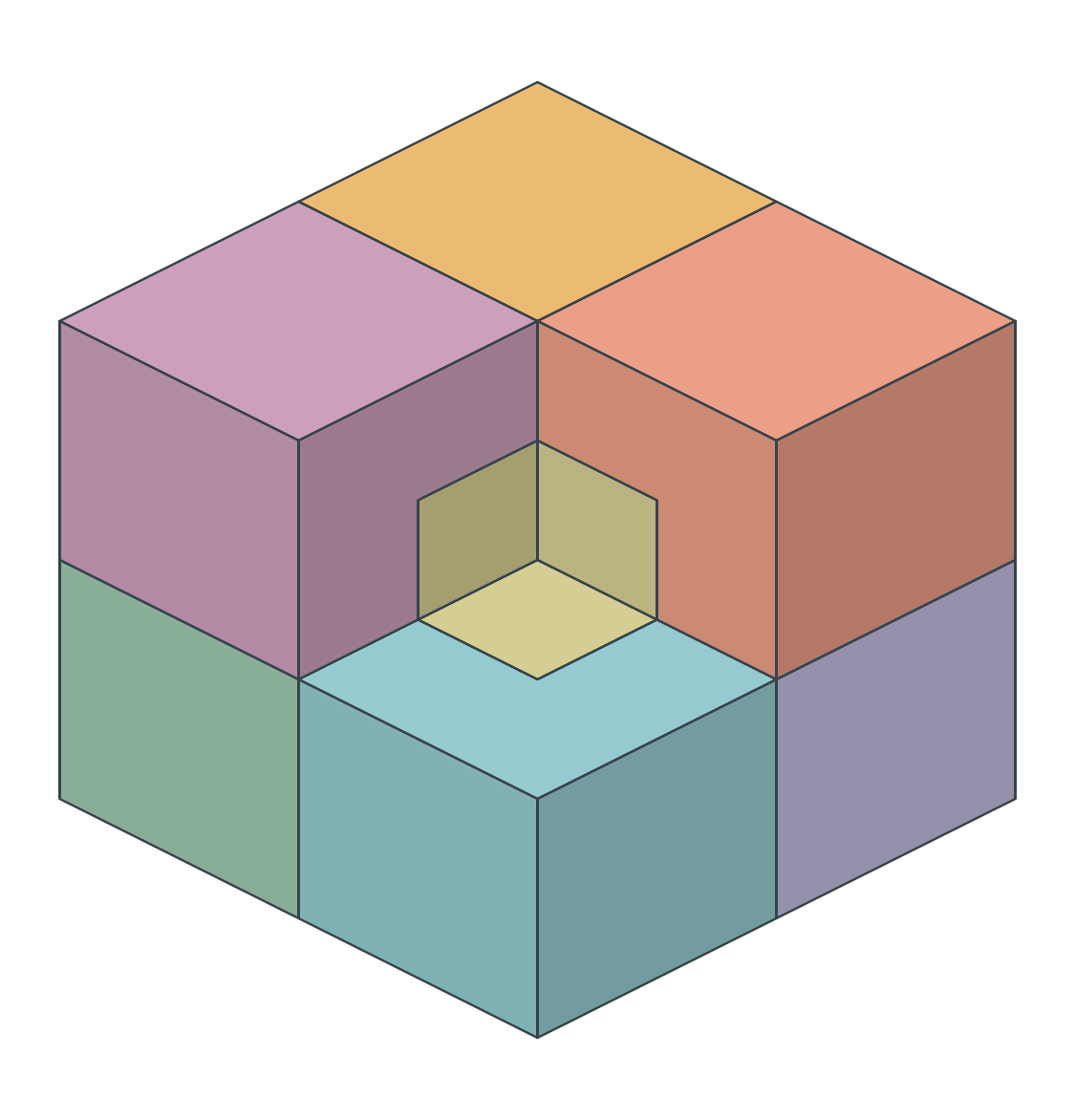}
  \caption{The eight-child dissection: eight properly rotated chair carriers partition
  $2\carrier$. Colours distinguish children; the small boundary features are omitted.}
  \label{fig:cluster}
\end{figure}

\begin{figure}[ht]
  \centering
  \includegraphics[width=0.52\linewidth]{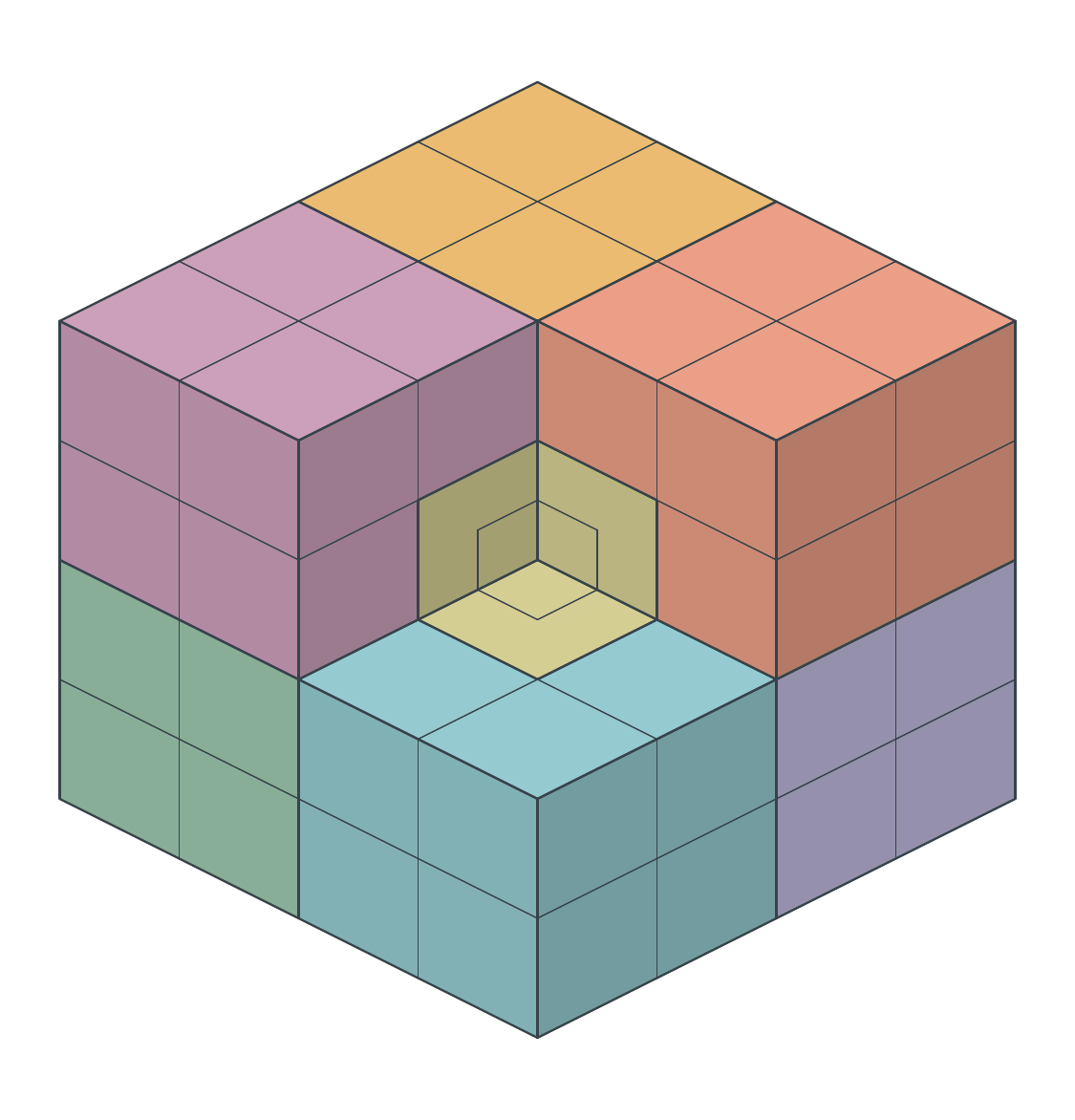}
  \caption{Two refinements: $64$ chair carriers partition $4\carrier$, coloured by their eight
  parents in the same view as \cref{fig:cluster}. Thin lines separate children; thicker lines
  mark parent boundaries and carrier creases. Boundary features are omitted; colour is annotation only.}
  \label{fig:patch}
\end{figure}

% ledger: R2, R3, D3, Q22; Lean atlas_44 (T1n), orientation_group_24 (T1n); source 1003.4909 L48 (Fletcher's atlas); 2303.10798 L96
Call a pose \emph{registered} if its frame is one of the $24$ proper cubic rotations and its
translation is an integer vector, and call a tiling \emph{registered} if, after one ambient
isometry, every tile is in a registered pose. In a registered tiling a chair with its pose
normalized to the identity has $22$ grid cells adjacent to it across its $24$ panels, and a
neighbouring chair covering one of them can sit in any of the $48$ signed frames; subtracting
its seven cells from each shell cell and rejecting overlaps of the two chairs leaves $2{,}388$
candidate face-neighbour poses. The features decide which of them are legal: two pyramids fit
only as a protrusion into a recess of the same magnitude, and because every one of the twelve
magnitudes is used by exactly one signed component, equality of the two profiles across a
shared panel is checked exactly, centre by centre. Exactly $44$ of the $2{,}388$ poses are legal
(a Lean theorem using \lean{native\_decide}). We call this set the \emph{$44$-contact atlas}
$\atlas$ (\cref{fig:atlas}). Its $44$ poses use $19$ distinct relative rotations, and those $19$
generate the whole $24$-element proper cubic group (\lean{native\_decide}). The word ``atlas'' is
Fletcher's, who uses it for a matching rule imposed on a cubic prototile \cite{Fle10}; here it
names a theorem about what the shape permits, and no rule is imposed. To see the substitution
and the atlas in action, the repository includes an interactive browser-based viewer of the solid
and its patches.

\begin{figure}[ht]
  \centering
  {\footnotesize\input{generated/atlas_44_table}}
  \caption{The $44$-contact atlas $\atlas$.}
  \label{fig:atlas}
\end{figure}

\subsection{The design test in finite form}\label{sec:designtest}

% ledger: H5, R7, R3; source PROOFS_registered.md R§8; Lean parent_atlas_eq_fine (T1n)
The test of \cref{sec:reading} asks whether the decoded parents obey the tile's own rule. In a
registered tiling whose chairs have been grouped into doubled chairs, a legal contact between two
fine chairs together with a choice of which child each is in its parent determines the relative
pose of the two parents. Enumerating the $8\times44\times8$ possibilities and removing duplicates
gives $697$ candidate parent contacts; $116$ of them have disjoint parent bodies; and testing every
fine contact across the two parents leaves $44$ legal parent contacts, every one with an even
translation. Dividing the translations by two gives \emph{exactly} the fine atlas $\atlas$, not
merely a set of the same size:
\[
  \tfrac12\,(\text{admitted parent atlas}) \;=\; \atlas
\]
(a Lean theorem using \lean{native\_decide}). This is the identity the whole construction was
designed to satisfy. Together with unique parent recognition and small-collar realization, this
identity permits coarsening (\cref{sec:hierarchy}): divide parent positions and lengths by two
and replace each parent carrier by a fresh copy of $\solid$. Translation covariance sends a period
$p$ to $p/2$, so \cref{sec:aperiodicity}'s halving argument applies at every level. The fact that
the $19$ contact
rotations generate the full proper cubic group, rather than a subgroup, is what makes the parent
frames the same set of frames as the child frames.

\subsection{How the tile was found}\label{sec:found}

% ledger: H7, H8, H9, H14; sources LINEAGE.md; history/cascade/findings.md L40-L48; packets P3-P6 (provenance/packets/); provenance/HASH_CHAIN.md
This section records how the solid was found; nothing in it is used in the proofs. Three
computed steps led to $\solid$, each recorded in a public artifact of the repository.

% ledger: H7; source history/cascade/findings.md L40-L48 (step 150); LINEAGE.md item 1
\emph{Step 1: a marked chair.} The construction record of the repository (steps $140$--$165$)
introduces the seven-cube chair as a rep-8 rep-tile and proves that a single-marking local rule
on the chairs, read off a radius-one box, admits the chair substitution; a SAT solver finds no
periodic tiling satisfying the rule up to period $L=7$. The record calls the all-scale recursion
an ``honest gap'': this is a marked candidate with strong finite evidence, not a proved marked
einstein. Attempts to realize the marks as notches on the chair failed; the verdict at the time
was that the rule is too complex for fixed covariant notches.

% ledger: H8; sources P3 README.md; P4 PROOFS.md L104-L139, L240-L258; P4 README.md
\emph{Step 2: a bare solid that is a periodic control.} The first bare solid put $192$ square
pyramids with signed heights on the $24$ panels, in fixed child frames. It admits an all-scale
aperiodic hierarchy, and it also admits a legal eight-copy periodic tiling. Decoding its tilings
into parent tilings explained why: the parents obey a weaker rule than the tiles, $62$ legal fine
contacts decoding to $398$ parent contacts, so the decoded language is the language of all
bare-chair tilings. The packet further proves that a deterministic translation-covariant recoding
with uniformly bounded neighbourhood and displacement, if it realizes the specified hierarchical
tiling by one compact solid, also realizes a periodic tiling. Requirement two of
\cref{sec:reading} fails, not requirement three: period preservation holds exactly,
$\Per(T)=2\Per(D(T))$ for the decoding $D$; what fails is preservation of the solid's
admissibility rule, since the decoded parents obey the less restrictive bare-chair language.

% ledger: H9, H14; sources P5 PROOFS.md "Claim and evidence status"; LINEAGE.md; provenance/HASH_CHAIN.md; Lean parent_atlas_eq_fine, orientation_group_24 (T1n)
\emph{Step 3: the frame change.} Replacing the fixed child frames by the proper-rotation frames of
\cref{tab:children}, which generate the full $24$-element cubic group, and recomputing the
profile from the closed contact set ($21\to30$ contacts) produced $\solid$, with $44$ legal
contacts and a parent atlas that halves to the same $44$. The decisive finite distinction for the
halving test is equality of the decoded parent atlas with the fine atlas, $44=44$ against
$62\to398$. The unrestricted alignment proof of \cref{sec:companions,sec:registration}, that every
tiling, not only registered ones, is forced onto the lattice, was written last. These stages are
recorded in packets P3--P6 of the repository's six-packet hash chain, each packet naming its
predecessor by hash and replayed and re-verified here; P6 strengthens the unrestricted alignment
argument.

% ledger: H2, H4, H6, H11, H10; Q-G author's account (labelled; carries no claim); framework named once with citations
\emph{The reading behind the steps.} What follows is the author's account of the route, in the
sense in which the hat paper recounts its own discovery, and it carries no claim. In the author's
recollection, the tile was found by asking what an aperiodic monotile is through the Six Birds
framework, an emergence framework developed by the author \cite{Tsi26F6,Tsi26F4,Tsi26NDO}, and
then having an AI reasoning model that was given the framework design a solid to satisfy that
description (see the disclosure after the acknowledgements); the packets that produced $\solid$
name the framework as their structural
compass and state their diagnoses of the control in its terms. The framework's statement of the
phenomenon is its law G11, ``Local-Rule Global-Anti-Symmetry'' \cite{Tsi26F6}, stated for a finite
local constraint system $C$ over configurations $\Z^d\to A$:

\begin{quote}\small
Assume: [H-G11-hierarchy] every globally admissible $C$-configuration carries the declared
hierarchy uniquely at every scale; the hierarchy is locally forced by the rules of $C$; and for
every nonzero period vector $p$ some scale $k$ has forced marker or block structure that cannot be
invariant under translation by $p$; [H-G11-nonempty] at least one global configuration satisfies
all local rules of $C$. Then admissible configurations exist; for every nonzero $p$ the hierarchy
supplies a scale whose forced structure is incompatible with $p$-periodicity; because the forcing
is local, the incompatibility is witnessed on a finite region $R_p$, yielding a per-period defect
certificate; and therefore every admissible configuration is aperiodic.
\end{quote}

\noindent The source lists Berger, Robinson, Penrose, the hat and the Spectre as its instances and
reads its own content as ``P2 constraint gating, P4 hierarchy across forced scales, and P1 descent
obstruction per candidate period''. For $\solid$ the instance is the registered system $C_{44}$:
labellings of $\Z^3$ by the $168$ registered cell labels (a chair in one of $24$ frames at one of
seven cells) obeying the atlas $\atlas$. Registration carries tilings of $\solid$ to $C_{44}$
(\cref{sec:registration}, resting on written lemmas) and the small-collar realization carries
registered baseline tilings back (\cref{sec:existence}, a written lemma); we do not claim a
packaged theorem that the two are one object. Clause by clause: [H-G11-nonempty] is
\cref{thm:full}(1); uniqueness at every scale is \cref{thm:full}(4); local forcing is the
unique-parent theorem of \cref{sec:hierarchy} applied level by level, the same finite rule
determining the parent from a neighbourhood whose radius grows with the level, which is not local
derivability from the unmarked chair tiling; and the period clause is reached not by exhibiting
G11's per-period certificates $R_p$, which our proof does not extract, but by the halving tower
of \cref{sec:aperiodicity}. At their verification levels, all four clauses are Lean theorems of
the pinned development, kernel-checked modulo the named compiler hooks: existence
(\lean{existence}); the hierarchy, with uniqueness among arbitrary geometrically nested registered
families (\lean{geometric\_hierarchy\_canonical}, \lean{geometric\_hierarchy\_unique}); local
forcing, from the finite shell and parent-atlas checks via \lean{native\_decide} together with the
proved local-to-global assembly and the proved registration bridge; and period exclusion, by the
Lean halving tower. The design test of \cref{sec:designtest} is, in the framework's terms,
its sufficiency and composability gate (laws F7 and F28 of \cite{Tsi26F4}). That reading
of the laws is the packets' own; F7 concerns factorization of declared readouts, whereas
preservation of legal tilings under coarsening is the separate geometric theorem of
\cref{sec:hierarchy}. The catalog contains no tiling instance. \Cref{sec:context} develops
this distinction as a complementary interpretation of the proved construction.

% ledger: H10, H6; source non-descending-objects paper (wording discipline)
The framework supplied a reading and a test. It proves nothing about $\solid$: existence is the
substitution's, universality is the companion argument's and the census's, and the framework's
own wording discipline says that its vocabulary certifies no claim. The monotile theorem
is certified by the proof and by the verification record of \cref{sec:computer},
and G11 is not a step of any proof here.

\subsection{Existence}\label{sec:existence}

% ledger: R10, T1, E2, E6, R8; source PROOFS_registered.md R§9; ERRATA E2; Lean nested_substitution_controls (T1n), existence (T1n); small_collar_realization_holds (T1n; formal proof identifies duplicate tube sites by centre)
\begin{theorem}[existence; R\S9, \lean{existence}, \tier{T1n}]\label{thm:existence}
$\solid$ admits a tiling of $\R^3$.
\end{theorem}

\begin{proof}
Every contact that occurs inside a refined patch lies in the closed set of $30$ contacts from
which the profile was built (\cref{app:profile}), so every finite substitution patch is a packing
of copies of $\solid$ whose internal panels match. Two refinements of the native chair contain the
same-frame grandchild at $(2,2,2)$: the central child $(I,(1,1,1))$ has its own $000$ child there.
Put $c_n=2(4^n-1)/3$ and $A_n=\sigma^{2n}(\carrier)-c_n(1,1,1)$, the $2n$-fold refinement
recentred. The grandchild relation makes $A_n$ a literal subpatch of $A_{n+1}$, with the same
frames, and the support of $A_n$ contains the cube $[-c_n,(4^n+2)/3]^3$. These cubes exhaust
$\R^3$, so the nested union is a registered baseline tiling of the whole space by chairs whose
contacts all lie in the closed contact set. The \emph{small-collar realization} lemma then
applies: a locally finite baseline chair tiling whose profiles match on every shared panel is
carried to a tiling by $\solid$ by the glued tube homeomorphism of \cref{app:tubes}, each chair
going exactly to its copy of $\solid$ because every feature tube belongs to the two cells on
either side of its panel, both with unique owners (\cref{sec:registration}, \cref{lem:cover}).
Small-collar realization is the Lean theorem \lean{small\_collar\_realization\_holds}
(\tier{T1n}); its formal proof identifies duplicate tube sites by their centres, so that the
conjugated tube maps agree where two chairs name the same tube, and separates distinct sites by
$21/200>1/16$.
\end{proof}

The same-frame grandchild, the exact dissection and the first two nested patches ($64$ and
$4{,}096$ tiles; $448$ and $28{,}672$ cells) are checked by Lean (\lean{native\_decide}); the
all-$n$ statement is the induction above, assembled in Lean from those checks and the realization
theorem. The construction is explicit; existence is not assumed to make \cref{thm:full}(2)
non-vacuous, and, as in the hat paper, it is proved independently of the census that establishes
universality.

%% file: generated/atlas_44_table.tex
% generated by paper/scripts/atlas_44.py from certificates/candidate_certificate.json (legal_contacts)
\begin{tabular}{rll}
\toprule
\# & frame $(p)\,(s)$ & offset \\
\midrule
1 & $(0,1,2)$ $(-,-,+)$ & $(2,2,-2)$ \\
2 & $(0,1,2)$ $(-,-,+)$ & $(2,2,2)$ \\
3 & $(0,1,2)$ $(-,-,+)$ & $(3,3,-1)$ \\
4 & $(0,1,2)$ $(-,-,+)$ & $(3,3,1)$ \\
5 & $(0,1,2)$ $(-,+,-)$ & $(2,-2,2)$ \\
6 & $(0,1,2)$ $(-,+,-)$ & $(2,2,2)$ \\
7 & $(0,1,2)$ $(+,-,-)$ & $(-2,2,2)$ \\
8 & $(0,1,2)$ $(+,-,-)$ & $(2,2,2)$ \\
9 & $(0,1,2)$ $(+,+,+)$ & $(-1,-1,-1)$ \\
10 & $(0,1,2)$ $(+,+,+)$ & $(1,1,1)$ \\
11 & $(0,2,1)$ $(-,+,+)$ & $(4,0,0)$ \\
12 & $(0,2,1)$ $(+,-,+)$ & $(-1,3,-1)$ \\
13 & $(0,2,1)$ $(+,-,+)$ & $(0,4,0)$ \\
14 & $(0,2,1)$ $(+,+,-)$ & $(0,0,4)$ \\
15 & $(0,2,1)$ $(+,+,-)$ & $(1,1,3)$ \\
16 & $(1,0,2)$ $(-,+,+)$ & $(0,0,0)$ \\
17 & $(1,0,2)$ $(-,+,+)$ & $(4,0,0)$ \\
18 & $(1,0,2)$ $(+,-,+)$ & $(0,0,0)$ \\
19 & $(1,0,2)$ $(+,-,+)$ & $(0,4,0)$ \\
20 & $(1,0,2)$ $(+,+,-)$ & $(-1,-1,3)$ \\
21 & $(1,0,2)$ $(+,+,-)$ & $(0,0,0)$ \\
22 & $(1,0,2)$ $(+,+,-)$ & $(0,0,4)$ \\
\bottomrule
\end{tabular}\hfill
\begin{tabular}{rll}
\toprule
\# & frame $(p)\,(s)$ & offset \\
\midrule
23 & $(1,0,2)$ $(+,+,-)$ & $(1,1,3)$ \\
24 & $(1,2,0)$ $(-,-,+)$ & $(2,2,-2)$ \\
25 & $(1,2,0)$ $(-,-,+)$ & $(2,2,2)$ \\
26 & $(1,2,0)$ $(-,-,+)$ & $(3,3,1)$ \\
27 & $(1,2,0)$ $(-,+,-)$ & $(2,-2,2)$ \\
28 & $(1,2,0)$ $(-,+,-)$ & $(2,2,2)$ \\
29 & $(1,2,0)$ $(-,+,-)$ & $(3,-1,3)$ \\
30 & $(1,2,0)$ $(+,-,-)$ & $(-2,2,2)$ \\
31 & $(1,2,0)$ $(+,-,-)$ & $(2,2,2)$ \\
32 & $(2,0,1)$ $(-,-,+)$ & $(2,2,-2)$ \\
33 & $(2,0,1)$ $(-,-,+)$ & $(2,2,2)$ \\
34 & $(2,0,1)$ $(-,-,+)$ & $(3,3,1)$ \\
35 & $(2,0,1)$ $(-,+,-)$ & $(2,-2,2)$ \\
36 & $(2,0,1)$ $(-,+,-)$ & $(2,2,2)$ \\
37 & $(2,0,1)$ $(+,-,-)$ & $(-2,2,2)$ \\
38 & $(2,0,1)$ $(+,-,-)$ & $(-1,3,3)$ \\
39 & $(2,0,1)$ $(+,-,-)$ & $(2,2,2)$ \\
40 & $(2,1,0)$ $(-,+,+)$ & $(3,-1,-1)$ \\
41 & $(2,1,0)$ $(-,+,+)$ & $(4,0,0)$ \\
42 & $(2,1,0)$ $(+,-,+)$ & $(0,4,0)$ \\
43 & $(2,1,0)$ $(+,+,-)$ & $(0,0,4)$ \\
44 & $(2,1,0)$ $(+,+,-)$ & $(1,1,3)$ \\
\bottomrule
\end{tabular}\hfill

%% file: sec4_companions.tex
\section{Features force companions}\label{sec:companions}

% ledger: A1-A9, A17, O6; source proof/ALIGNMENT_PROOF.md A§1-A§4 (A supersedes R where both treat a step)
This section and the next prove that every tiling by $\solid$ is registered
(\cref{fig:companion}). The argument is
continuous geometry: the pyramids are tiny, but their edges carry dihedral angles that only a
pyramid of the same magnitude and opposite polarity can complete, and that forces every feature
of every tile to be matched, whole, by one feature of one neighbour. We follow the written
alignment proof of the repository, whose lemma labels (A-L2.1 and so on) we give with each
statement together with the Lean theorem that proves it and that theorem's tier
(\cref{tab:hypotheses}); the written proofs are the exposition, and where a formal proof argues
differently we say so after the written one.

% notation; source A§1
Fix a feature $F$ of a tile and write $E(F)$ for its \emph{closed edge graph}: the four edges of
the base square and the four ridges from the base corners to the apex, a compact connected graph
without isolated vertices. In tangent coordinates $u$ on the panel and outward normal coordinate
$z$, the feature of signed height $h=a/10000$ is the graph of the tent function
$f_h(u)=h\max(0,1-\|u\|_\infty/\eta)$, and near it the solid is the hypograph $z\le f_h(u)$.
Throughout, a \emph{packing} is a family of isometric copies of $\solid$ with pairwise disjoint
interiors, and a tiling is a packing whose union is $\R^3$.

\begin{lemma}[local finiteness; A-L2.1, \tier{T1} \lean{local\_finiteness}]\label[lemma]{lem:locfin}
Every packing by congruent copies of $\solid$ is locally finite.
\end{lemma}

\begin{proof}
Choose a ball of positive radius inside the interior of $\solid$ and carry it along with each
tile. The carried balls have pairwise disjoint interiors. A tile meeting a bounded set $K$ has its
ball inside a fixed bounded enlargement of $K$, since $\solid$ has bounded diameter, and only
finitely many disjoint balls of equal positive volume fit there.
\end{proof}

% ledger: A2; Lean cone_sector_budgets_holds (T1n)
\begin{lemma}[cone budgets; A-L2.2, written; Lean \tier{T1n} \lean{cone\_sector\_budgets\_holds}]\label[lemma]{lem:cones}
At a point common to finitely many tiles of a packing, the interiors of their tangent cones are
disjoint, so their solid angles sum to at most $4\pi$, with equality in a tiling. At a generic
point of an edge, a perpendicular cross-section is partitioned into planar sectors whose angles
sum to $2\pi$ in a tiling.
\end{lemma}

\begin{proof}
Near a boundary point a polyhedral set agrees with a translate of its tangent cone, and one
radius serves the finitely many incident tiles; a common open direction would give overlapping
interiors. Cone boundaries are finite unions of planar pieces and have spherical area zero.
Coverage gives equality in a tiling. The planar statement is the same argument in the
perpendicular plane.
\end{proof}

The word \emph{generic} is made precise as follows. Fix a closed root edge and the finitely many
tiles meeting a small neighbourhood of it. Remove their vertices on the edge, their intersections
with non-collinear edges, and the intersections with face planes that meet the root line in a
single point: finitely many points. At every remaining point an incident tile either has a face
containing the line, which contributes a half-plane sector of angle $\pi$, or an edge collinear
with it, which contributes an ordinary dihedral sector. No tile can contain such a point in its
interior, since $\solid$ is the closure of its interior and the root has interior points nearby.
Nothing here assumes that a neighbouring facet coincides with a facet of the root; contacts that
are not face-to-face are handled by the same sectors.

% ledger: A3, O6, A17 (E3); Lean deviations_distinct (T1), mesh_angle_audit (T1n)
Put $t_j=j/100$ for $j=1,\dots,12$, and
\[
  \alpha_j=\arctan t_j,\qquad \beta_j=\arccos\frac{1}{1+t_j^2}.
\]

\begin{lemma}[complete dihedral list; A-L3.1, written; Lean \tier{T1n}
\lean{complete\_dihedral\_list\_holds}; its arithmetic and the mesh audit are separate Lean
theorems, see below]\label[lemma]{lem:dihedral}
A protrusion of magnitude $j$ has base dihedral $\pi+\alpha_j$ and ridge dihedral $\pi-\beta_j$;
a recess has the complementary angles. Every other edge of $\solid$ has dihedral $\pi/2$ or
$3\pi/2$, and flat triangulation edges have angle $\pi$. All $24$ deviations $\alpha_j,\beta_j$
are distinct and less than $\pi/4$.
\end{lemma}

\begin{proof}
The facets of the tent have slopes $(0,0)$, $(\pm t_j,0)$ and $(0,\pm t_j)$; their outward normals
give the formulas and the signs. The tube construction changes no carrier edge, and the disjoint
feature supports create no other kind of edge. Both sequences increase strictly, and a
coincidence $\alpha_i=\beta_j$ would force $1+(i/100)^2=(1+(j/100)^2)^2$, that is
$10000\,i^2=20000\,j^2+j^4$; then $10\mid j$, so $j=10$ and $i^2=201$, impossible. The bound
$\pi/4$ follows from $t_j<1$ and $(1+t_j^2)^2<2$ at $t_{12}=3/25$.
\end{proof}

The unsolvability of $10000\,i^2=20000\,j^2+j^4$ over $1\le i,j\le12$ is a Lean theorem by kernel
\lean{decide} (\lean{deviations\_distinct}). The lemma is about the ideal construction; a separate
audit, a Lean theorem using \lean{native\_decide} (\lean{mesh\_angle\_audit}), classifies every
one of the $6{,}408$ edges of the actual mesh by exact rational outward normals, finds exactly
$1{,}536$ feature edges with $32$ of each sign for each of the $24$ deviations, and checks that
every other edge is flat or has orthogonal normals (\cref{tab:mesh}).

% ledger: A4, A17 (E3); Lean generic_feature_partner_holds (T1n)
\begin{lemma}[one complementary partner; A-L3.2, written; Lean \tier{T1n}
\lean{generic\_feature\_partner\_holds}]\label[lemma]{lem:partner}
At a generic point of a feature edge in any tiling exactly two tiles are incident: the root, and
one tile whose incident edge has the same type (base or ridge), the same magnitude, and the
complementary dihedral angle.
\end{lemma}

\begin{proof}
Every feature sector lies strictly between $3\pi/4$ and $5\pi/4$. Three feature sectors exceed
$2\pi$, as do two feature sectors together with any ordinary sector, whose angle is at least
$\pi/2$. A single feature sector is impossible: its angle is $\pi$ plus a nonzero deviation of
magnitude less than $\pi/4$, while the remaining sectors are multiples of $\pi/2$ (a face
containing the line contributes $\pi$, as noted above), and such a sum is not $2\pi$. The root
supplies one feature sector, so there are exactly two and nothing else; their deviations are
opposite and equal in magnitude, and \cref{lem:dihedral} identifies the type and the magnitude.
\end{proof}

This does not yet say that an edge angle forces a whole feature. That needs a bound at the
vertices of the feature, where sectors are not available.

% ledger: A5, A6; Lean feature_circular_cone_containment_holds (T1n), circular_cone_solid_angle_holds (T1)
\begin{lemma}[solid-angle bound on the feature graph; A-L4.1, written; Lean \tier{T1n}
\lean{feature\_circular\_cone\_containment\_holds}, with the cone formula and
its arithmetic \tier{T1} \lean{circular\_cone\_solid\_angle\_holds}]\label[lemma]{lem:solidangle}
At every point of $E(F)$, including the apex and the base corners, the interior solid angle of
the tile exceeds $4\pi/3$; for $\solid$ it exceeds $7\pi/4$.
\end{lemma}

\begin{proof}
The tent is Lipschitz with constant at most $L=3/25$, on its faces, along its creases, and across
the transition to the flat panel. At any point of the graph the tile's interior tangent cone
contains the circular cone $\{(v,w): w<-L\|v\|_2\}$, whose solid angle is
$\Omega(L)=2\pi\bigl(1-L/\sqrt{1+L^2}\bigr)$. This exceeds $4\pi/3$ exactly when $8L^2<1$, and here
$8L^2=72/625$; it exceeds $7\pi/4$ when $63L^2<1$, and $63L^2=567/625$.
\end{proof}

Only the weaker bound is used below: by \cref{lem:cones}, three distinct tiles cannot share a
point lying on a feature graph of each of them.

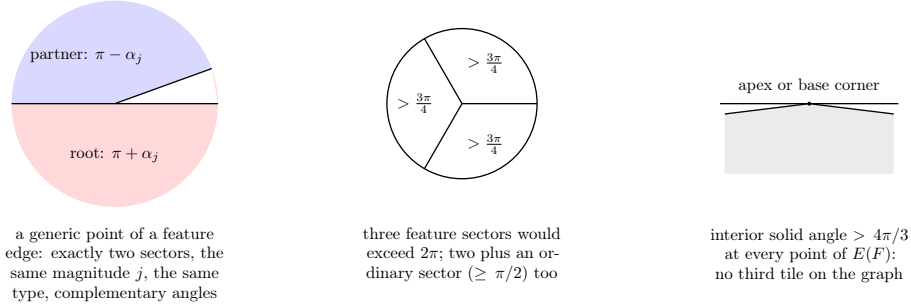
\begin{figure}[ht]
  \centering
  \resizebox{\linewidth}{!}{\input{fig_companion_sectors}}
  \caption{The companion argument: sectors at a generic edge point (\cref{lem:partner}) and the
  solid-angle cone at a vertex of the feature graph (\cref{lem:solidangle}).}
  \label{fig:companion}
\end{figure}

% ledger: A7; Lean connected_feature_companion_holds (T1n)
\begin{theorem}[one tile accompanies the whole feature graph; A-T4.2, written; Lean \tier{T1n}
\lean{connected\_feature\_companion\_holds}]\label{thm:companion}
For each feature $F$ of a tile $R$ in any tiling there is a unique other tile $S$ such that
$E(F)$ is contained in the union of the feature edge graphs of $S$; moreover $E(F)$ lies in one
feature edge graph of $S$.
\end{theorem}

\begin{proof}
By \cref{lem:locfin}, only finitely many tiles other than $R$ meet $E(F)$. For each such tile
$S\ne R$, let $C_S$ be the intersection of $E(F)$ with the union of the feature edge graphs of
$S$; each $C_S$ is closed in $E(F)$. The sets $C_S$ cover $E(F)$: the generic interior points of the root's edges are covered
by \cref{lem:partner}, they are dense in $E(F)$, and a finite union of closed sets containing a
dense subset is everything. The sets $C_S$ are pairwise disjoint: a point in $C_S\cap C_T$ with
$S\ne T$ would lie on a feature graph of $R$, of $S$ and of $T$, which \cref{lem:cones,lem:solidangle}
forbid. A connected space is not the union of two or more nonempty members of a finite disjoint
closed cover, each member being also open. Hence exactly one $C_S$ is nonempty and equals $E(F)$.
The feature graphs of $S$ are disjoint compact sets, and the same connectedness argument applied
to $E(F)$ inside their union puts $E(F)$ in one of them.
\end{proof}

This single argument disposes of every case in which a partner might change: at an interior edge
point, at a subdivision vertex of the mesh, at the apex, at a base corner, or between two
features of the same neighbour. It needs no face-to-face hypothesis and no continuation rule
supplied by the observer.

% ledger: A8, A17 (E1); Lean feature_containment_rigidity_holds (T1n; formal proof by compactness)
\begin{lemma}[containment implies equality; A-L4.3, written, with erratum E1; Lean \tier{T1n}
\lean{feature\_containment\_rigidity\_holds}]\label[lemma]{lem:equal}
The features $F$ and $F'$ of \cref{thm:companion} have the same closed edge graph, the same
magnitude, and opposite polarity; their triangulated surfaces coincide.
\end{lemma}

\begin{proof}
Choose a point in the interior of a base edge of $F$ that avoids the vertices of the companion's
graph and the finite exceptional set of \cref{lem:cones}; such a point exists. There
\cref{lem:partner} gives $F'$ the same magnitude $j$ and the same edge type. A feature of
magnitude $j$ has total edge length $\ell_j=8\eta+4\sqrt{2\eta^2+(j/10000)^2}$. We know $E(F)$ is
a closed subset of $E(F')$; if the inclusion were proper, a point of $E(F')$ outside $E(F)$ would
have positive distance from the closed set $E(F)$, and since $E(F')$ has no isolated vertices it
would contribute a nonzero subsegment outside $E(F)$, giving $\ell_j<\ell_j$ by additivity of
length. So the graphs are equal. In the graph the apex is the unique vertex of degree four and
the base corners have degree three, so equality recovers the base square, the apex, and the four
triangles; complementary base dihedrals force opposite polarity.
\end{proof}

The formal proof argues differently: the graph $E(F)$ is compact and the placement carrying it
into $E(F')$ is an isometry, and an isometric self-map of a compact metric space is onto, so the
containment is an equality without the length count; the premise that the two tiles are distinct
is used there.

% ledger: A9; Lean companion_pose_discrete_holds (T1n)
\begin{corollary}[a feature match discretizes the pose; A-C4.4, written; Lean \tier{T1n}
\lean{companion\_pose\_discrete\_holds}]\label[corollary]{cor:discrete}
Normalize the root's pose to $(I,0)$. Every feature companion has a signed permutation matrix $G$
as linear part and translation $t=c_u-Gc_v\in\tfrac18\Z^3$, where $c_u,c_v$ are native feature
centres with $a_u=-a_v$ and outward base normals $n_u=-Gn_v$.
\end{corollary}

\begin{proof}
A rigid motion matching two complete pyramid graphs carries base-edge directions and base-plane
normal to each other, and those are the coordinate directions, so its orthogonal part is a
signed permutation, of either determinant. The graph identifies the centre; opposite polarity and
a common apex force opposite outward normals; all centres lie in $\tfrac18\Z^3$.
\end{proof}

No global alignment has been concluded yet: only the poses of complete feature companions have
been reduced to a finite list, which the next section examines exhaustively. The mechanism is the
one Socolar and Taylor use for their three-dimensional tile, where plugs enforce a matching rule
by shape \cite{ST11}; here the pyramids enforce, by dihedral angles and solid angles alone, that
every feature is matched by a feature of the opposite polarity on one neighbour.

%% file: fig_companion_sectors.tex
% fig_companion_sectors.tex — Fig. 4.1: the sector budget at a generic feature-edge point (left),
% why three feature sectors cannot meet (middle), and the solid-angle cone at a vertex of the
% feature graph (right). Schematic. Used by sec4_companions.tex.
\begin{tikzpicture}[x=1cm,y=1cm,font=\small,>={Latex[length=1.6mm]}]
  \begin{scope}
    \fill[red!15] (0,0) -- (-2.2,0) arc (180:360:2.2) -- (2.2,0) -- (2.2*0.94,2.2*0.34) arc (20:0:2.2) -- cycle;
    \fill[blue!15] (0,0) -- (2.2*0.94,2.2*0.34) arc (20:180:2.2) -- cycle;
    \draw[thick] (-2.2,0) -- (2.2,0);
    \draw[thick] (0,0) -- (2.2*0.94,2.2*0.34);
    \node at (0,-1.1) {root: $\pi+\alpha_j$};
    \node at (-0.6,1.0) {partner: $\pi-\alpha_j$};
    \node[anchor=north,align=center,text width=5.2cm] at (0,-2.5)
      {a generic point of a feature edge: exactly two sectors, the same magnitude $j$, the same type, complementary angles};
  \end{scope}
  \begin{scope}[shift={(7.4,0)}]
    \draw[thick] (0,0) circle (1.6);
    \foreach \a in {0,120,240} { \draw[thick] (0,0) -- (\a:1.6); }
    \node at (60:1.0) {$>\tfrac{3\pi}{4}$};
    \node at (180:1.0) {$>\tfrac{3\pi}{4}$};
    \node at (300:1.0) {$>\tfrac{3\pi}{4}$};
    \node[anchor=north,align=center,text width=5.2cm] at (0,-2.5)
      {three feature sectors would exceed $2\pi$; two plus an ordinary sector ($\ge\pi/2$) too};
  \end{scope}
  \begin{scope}[shift={(14.8,0)}]
    \draw[thick] (-1.9,0) -- (1.9,0);
    \fill[black!8] (0,0) -- (-1.8,-0.22) -- (-1.8,-1.5) -- (1.8,-1.5) -- (1.8,-0.22) -- cycle;
    \draw[thick] (0,0) -- (-1.8,-0.22); \draw[thick] (0,0) -- (1.8,-0.22);
    \fill (0,0) circle (1.2pt);
    \node[anchor=south] at (0,0.08) {apex or base corner};
    \node[anchor=north,align=center,text width=5.2cm] at (0,-2.5)
      {interior solid angle $>4\pi/3$ at every point of $E(F)$: no third tile on the graph};
  \end{scope}
\end{tikzpicture}

%% file: sec5_registration.tex
\section{Every tiling is registered}\label{sec:registration}

% ledger: A10-A15, R9, M3, E6; source proof/ALIGNMENT_PROOF.md A§5-A§6; Lean retained_core_overlap_holds (T1), mates_census (T1n), registration_of_tiling (spine)
Let $F_0$ be the finite set of poses given by \cref{cor:discrete}: for every ordered pair of
native features of opposite signed height and every signed permutation $G$ with $n_u=-Gn_v$, the
pose $(G,\,c_u-Gc_v)$. There are $6{,}862$ of them, with translations in eighth units, and for
each pose we retain the list of root features it can fully accompany. Each is a candidate for the
pose of a tile sharing a feature with the root; the task is to show that only $44$ of them, all
integral and proper, can occur in a tiling.

\begin{lemma}[eighth-grid overlaps penetrate the cores; A-L5.1; the overlap implication is the
Lean theorem \lean{retained\_core\_overlap\_holds} (\tier{T1}), the quantitative bound the written
estimate in the proof]\label[lemma]{lem:core}
If two poses in $F_0$ have baseline chairs overlapping in positive volume, the corresponding
copies of $\solid$ overlap in positive volume; more precisely, an overlapping pair of unit cubes
yields an open box of side at least $21/200$ inside both interiors, hence a ball of radius
$21/400$.
\end{lemma}

\begin{proof}
Cube endpoints are eighth-integral, so each positive coordinate overlap has length at least
$1/8$; eroding both cubes by $\rho=1/100$ leaves at least $1/8-2/100=21/200$, and the eroded cubes
lie in the interiors of the solids (\cref{app:tubes}).
\end{proof}

% ledger: A11, M3; the census table -> Appendix B; two independent implementations; Lean mates_census
Let $F_1$ be the poses in $F_0$ whose baseline chair does not overlap the root's, and for a root
feature $u$ let $\mathrm{Partners}(u)$ be \emph{all} poses in $F_1$ matching $u$; no other
rejection enters these lists. The census is summarized in \cref{tab:census} and
\cref{app:census}: of the $6{,}862$ poses, $1{,}545$ collide with the root directly, leaving
$5{,}317$ in $F_1$, of which $5{,}234$ have a non-integral translation and $83$ are integral. The
remaining step rejects $5{,}273$ of the $5{,}317$ by the following finite statement, leaving $44$.

\begin{proposition}[collision witnesses; A-FS5.2, \tier{T1n} \lean{mates\_census}, and two
independent replays]\label[proposition]{prop:census}
For each of the $5{,}273$ rejected poses $V\in F_1$ the certificate names a feature $u$ of either
the root $U$ or of $V$, and, in its owner's coordinates: (i) the other tile of the pair is not in
$\mathrm{Partners}(u)$; (ii) $\mathrm{Partners}(u)$ is nonempty; (iii) every $W\in\mathrm{Partners}(u)$
has a positive-volume baseline overlap with the other tile of the pair. There are $299{,}975$ such
option collisions in all.
\end{proposition}

The verifier regenerates $F_0$ and every $\mathrm{Partners}(u)$ from the actual mesh and checks
the three clauses; it uses no flat-panel consistency filter and no already-rejected pose when
building the option sets, so the conclusion is not used to prune its own options. For every
option it exports one pair of overlapping native cubes with their intersection box, and a
second, integer-only checker verifies all $299{,}975$ boxes in units of $1/400$; every box has
sides at least $42/400$ (\cref{fig:collision}).

\begin{figure}[ht]
  \centering
  \resizebox{0.9\linewidth}{!}{\input{fig_collision_box}}
  \caption{Schematic of the retained-core overlap estimate used in \cref{prop:census}; this is
  not a selected certificate witness.}
  \label{fig:collision}
\end{figure}

% ledger: A12; Lean only_registered_mates_holds (T1)
\begin{lemma}[only the 44 survivors occur; A-L5.3, written; Lean \tier{T1}
\lean{only\_registered\_mates\_holds}]\label[lemma]{lem:44}
No rejected pose is a feature companion in any tiling. The $44$ surviving poses are integral and
proper, and they are exactly the atlas $\atlas$ of \cref{sec:substitution}.
\end{lemma}

\begin{proof}
Suppose tiles $U$ and $V$ occur in a tiling in a rejected relative pose. The named feature $u$ has
a complete companion $W$ by \cref{thm:companion,lem:equal}, whose pose lies in
$\mathrm{Partners}(u)$ because the list of complete-feature poses is exhaustive
(\cref{cor:discrete}) and an overlap with the owner is forbidden; by clause (i) $W$ is not the
other tile already placed. Clause (iii) says $W$'s baseline overlaps that other tile's, and
\cref{lem:core} turns this into an overlap of interiors, a contradiction. The verifier compares
the set of $44$ survivors, not merely its size, with the atlas computed from the profile in
\cref{sec:substitution}; they are equal.
\end{proof}

This is a bounded local obstruction with a complete geometric carrier theorem behind it. It is
not a search for periodic cells, not a timeout, not a sample of surface points, and not an
inference that a partial gap between misaligned tiles cannot be filled by a third tile: every
possible filler of the missing companion is exhibited as colliding with a whole tile.

% ledger: A13; Lean baseline_component_covers_grid_holds (T1n)
\begin{lemma}[a component's baseline cells exhaust the grid; A-L6.1, written; Lean \tier{T1n}
\lean{baseline\_component\_covers\_grid\_holds}]\label[lemma]{lem:grid}
Choose any root tile in a tiling and normalize its pose to $(I,0)$. In the graph whose edges join
tiles sharing a complete feature companion, let $\mathcal{C}$ be the root's component. Every tile
in $\mathcal{C}$ has an integer translation and a proper cubic frame, and the baseline chairs of
$\mathcal{C}$ tile the lattice of unit cubes.
\end{lemma}

\begin{proof}
Every edge of the graph has a relative pose in $\atlas$ (\cref{lem:44}), and compositions preserve
integrality and properness. Two baseline chairs of $\mathcal{C}$ cannot share a cell, since their
retained cores would overlap. Let $W$ be the set of unit cells covered by the chairs of
$\mathcal{C}$. Take a cell $a\in W$ and a face-adjacent cell $b$. If $b$ lies in the same chair it is
in $W$. Otherwise the common face is an exposed panel of $a$'s chair, carrying eight features;
the companion of any one of them lies in $\mathcal{C}$ with an integral pose, its base normal is
opposite, and its centre coincides with the root feature's centre at a point strictly inside an
integral unit panel; two integral unit squares in that plane whose relative interiors share the
point are the same square, so the companion owns $b$. Thus $W$ is closed under lattice steps, and
$W=\Z^3$, with exactly one owner per cell.
\end{proof}

% ledger: A14, A17 (E6); Lean component_solids_cover_holds (T1; formal proof by clamping to an interior point)
\begin{lemma}[the component's solids cover $\R^3$; A-L6.2, written, with erratum E6; Lean \tier{T1}
\lean{component\_solids\_cover\_holds}]\label[lemma]{lem:cover}
The copies of $\solid$ in $\mathcal{C}$ cover $\R^3$ with disjoint interiors.
\end{lemma}

\begin{proof}
By \cref{lem:grid}, across every panel shared by two chairs of $\mathcal{C}$ the adjacent cell has one owner, and each
of the eight features on the panel has its companion in that owner, since another owner would
duplicate the cell; so the two profiles agree on the whole panel. Around each paired feature take
its closed normal tube of half-height $\rho$ and apply the tube map of \cref{app:tubes}; opposite
polarity and opposite outward normals give the same map from either side, so this is one map, not
two. The tubes are disjoint except for coincident pairs and are locally finite; the maps are the
identity outside the tubes and on their boundaries, and glue with their inverses to a
homeomorphism $H$ of $\R^3$. Every feature tube belongs to the two cells on either side of its
panel, both with unique owners in $\mathcal{C}$, and a tube not on a tile's exposed boundary is
disjoint from that tile's material; hence $H$ carries each baseline chair $G\carrier+t$ of
$\mathcal{C}$ exactly onto $G\solid+t$, and the images cover $\R^3$ with disjoint interiors.
\end{proof}

The formal proof avoids the assembly-wide gluing. For any tile of the given tiling, choose a
component carrier containing the centre of its inscribed quarter-ball. Clamping produces an
interior point of the component's physical tile at squared distance at most $3/2500<1/100$ from
that centre, hence inside the quarter-ball. Packing disjointness forces the two tiles to
coincide. Thus every tile belongs to the component, and the original tiling's coverage gives
component coverage. This uses an existing tiling and is not an existence proof.

% ledger: A15; Lean unrestricted_alignment_holds (T1n), registration_of_tiling
\begin{theorem}[unrestricted alignment; A-T6.3, written; assembled in Lean as
\lean{registration\_of\_tiling} from \lean{unrestricted\_alignment\_holds},
\tier{T1n}]\label{thm:registration}
Every tiling by $\solid$ is registered: there is one ambient isometry after which every tile has
a proper cubic frame and an integer translation, and every pair of tiles sharing a feature is in a
relative pose of $\atlas$.
\end{theorem}

\begin{proof}
The component $\mathcal{C}$ of \cref{lem:cover} already fills space. A tile outside $\mathcal{C}$
would have nonempty interior, while the boundaries of the tiles of $\mathcal{C}$ form a locally
finite union of polyhedral surfaces containing no ball; an interior point of the extra tile would
then lie inside a tile of $\mathcal{C}$, contradicting disjointness of interiors. So there is no
other component, and the ambient isometry that normalizes the root registers every tile.
\end{proof}

% ledger: A15 preserves periodicity; source 2305.17743 L53; chirality remark R§5
Registration is one ambient isometry applied to the whole tiling, so it preserves periodicity in
the Spectre paper's sense \cite{SMKGS24b}: the registered tiling has the conjugate symmetry
group, and if the ambient isometry has linear part $A$, a period $p$ of the original tiling
becomes the period $Ap$ of the registered tiling, with $\|Ap\|=\|p\|$. This closes the loophole of an integral network of tiles surrounded by tilted or
shifted tiles in residual gaps: the tube homeomorphism proves physical coverage, not merely the
occupancy of a grid. The relative poses in $\atlas$ are all proper, although all $48$ signed
frames were admitted as candidates; a tiling may be reflected as a whole, but its two handedness
sectors never mix, which is the content of \cref{thm:full}(3) proved in \cref{sec:aperiodicity}.
In Lean, a registration is a record of an ambient isometry and, for each placement, a frame index
among the $24$ and an integer shift; the theorem that every tiling has one is assembled from the
three written lemmas of this section, which appear by name in \cref{app:lean}.

%% file: fig_collision_box.tex
% fig_collision_box.tex — Fig. 4.2: a collision witness, schematic in a coordinate plane. Two
% baseline unit cells of the eighth grid overlap in an interval of length >= 1/8 in each
% coordinate; eroding both by rho = 1/100 leaves a core box of side >= 21/200 inside both solids.
% Used by sec5_registration.tex.
\begin{tikzpicture}[x=1cm,y=1cm,font=\small,>={Latex[length=1.6mm]}]
  % two unit cells on the eighth grid, offset by (1/8, 3/8) (drawn at 3.2 cm per unit)
  \def\S{3.2}
  \draw[gray!40,step=0.4] (-0.2,-0.2) grid (4.6,4.6);
  \draw[thick,blue!70!black] (0,0) rectangle (\S,\S);
  \draw[thick,red!70!black] (0.4,1.2) rectangle (0.4+\S,1.2+\S);
  \fill[black!15] (0.4+0.032,1.2+0.032) rectangle (\S-0.032,\S-0.032);
  \draw[very thick] (0.4+0.032,1.2+0.032) rectangle (\S-0.032,\S-0.032);
  \node[blue!70!black,anchor=south west] at (0.05,0.05) {cell of the other tile};
  \node[red!70!black,anchor=north east] at (0.4+\S-0.05,1.2+\S-0.05) {cell of the option $W$};
  \node[anchor=west,align=left] at (4.9,2.6) {overlap of the two unit cells:\\ each coordinate interval has length $\ge 1/8$\\ (all endpoints are eighth-integral)};
  \node[anchor=west,align=left] at (4.9,1.2) {shaded: both cells eroded by $\rho=1/100$,\\ side $\ge 1/8-2/100 = 21/200$; the eroded\\ cells lie inside the two solids, so the\\ solids overlap in a box of that side};
  \node[anchor=north,align=center] at (2.3,-0.4) {Schematic coordinate-plane section of the overlap estimate\\ (the certificate stores the actual boxes in integer units of $1/400$)};
\end{tikzpicture}

%% file: sec6_hierarchy.tex
\section{The unique hierarchy}\label{sec:hierarchy}

% ledger: R4-R8, T8, T9, N1, N2, E4; source PROOFS_registered.md R§7-R§9; ruling R7 (MECHANIZATION_PLAN.md C3); Lean first_shells_33, central_completion, parent_conflicts_28 (T1n), parent_local_to_global, carrier_hierarchy, carrier_hierarchy_unique (spine)
From now on every tiling is registered (\cref{thm:registration}), and the argument is finite
combinatorics on the lattice plus one scale argument. For each of the eight children of
\cref{tab:children} there is one candidate parent through a root chair in the identity pose:
apply the inverse child pose to the list of children. Each candidate is a cluster of eight chairs
whose baseline is a congruent copy of $2\carrier$.

\begin{lemma}[first shells; R\S7, \tier{T1n} \lean{first\_shells\_33}]\label[lemma]{lem:shells}
There are exactly $33$ ways to cover the $22$ shell cells of a root chair by legal neighbours
from $\atlas$ without overlaps and without an illegal contact between chosen neighbours. Each of
the $33$ contains the face-adjacent sibling signature of exactly one candidate parent.
\end{lemma}

The $33$ arrangements (\cref{fig:shells}) are necessary local configurations, not claims that each extends to a
tiling; a second implementation, using a set-based recursion instead of the generator's bitmask
recursion, reproduces all $33$.

% ledger: R5; Lean central_completion (T1n)
\begin{lemma}[central completion; R\S7, \tier{T1n} \lean{central\_completion}]\label[lemma]{lem:central}
In one of the $33$ arrangements the root is the central child of its proposed parent and all
seven outer children are present. In the other $32$ the proposed parent's centre is a neighbour of
the root; checking every first shell at that centre for compatibility with the arrangement, $18$
of the $32$ have no compatible completion and cannot occur in a tiling, $14$ have completions,
and every completion makes the proposed centre genuinely central.
\end{lemma}

\begin{lemma}[parent conflicts; R\S7, \tier{T1n} \lean{parent\_conflicts\_28}]\label[lemma]{lem:conflicts}
Each of the $28$ unordered pairs of distinct candidate parents through one root contains an
incompatible pair of child poses.
\end{lemma}

\begin{theorem}[unique parent; R\S7 Theorem 7.1; \lean{parent\_local\_to\_global},
\lean{unique\_parent}: \tier{T1n}, unconditional]\label{thm:parent}
Every tiling by $\solid$ partitions uniquely into the eight-chair parent clusters, and the
parent of each tile is determined by a finite neighbourhood of it.
\end{theorem}

\begin{proof}
Registration gives the complete local inventory around a tile: its first shell is one of the $33$
of \cref{lem:shells}, which proposes a unique parent. By \cref{lem:central} that parent's seven
other children are present in every tiling in which the shell occurs. By \cref{lem:conflicts}
two parents cannot share a tile, so the local assignments are globally disjoint and consistent,
and every tile is included.
\end{proof}

\begin{figure}[ht]
  \centering
  {\scriptsize\input{generated/first_shells_table}}
  \caption{The $33$ first shells, their proposed parent types and central-completion outcomes.}
  \label{fig:shells}
\end{figure}

% ledger: R7, E4; Lean parent_atlas_eq_fine (T1n), parent_atlas_admissibility, parent_common_parity, halved_parent_baseline_tiling (spine); small_collar_realization_holds (T1n)
\begin{theorem}[admissibility-preserving coarsening; R\S8 Theorem 8.1; \tier{T1n},
unconditional]\label{thm:coarsening}
Let $D$ group a tiling $T$ into its unique parents, divide the parent origins and the scale by
two, and replace each parent's baseline by a fresh copy of $\solid$ in that frame. Then $D(T)$ is
again a tiling by $\solid$, and $D(T+v)=D(T)+v/2$ for every vector $v$.
\end{theorem}

In Lean the two clauses are \lean{coarsening\_is\_tiling} and
\lean{coarsening\_translation\_equivariant}; the physical realization of the coarsened tiling
uses the theorem \lean{small\_collar\_realization\_holds} (\tier{T1n}, \cref{sec:existence}) and
the tube construction.

\begin{proof}
The parent bodies partition space and their contact graph is connected, since the unit-cell
adjacency graph is. Every contact between two parents is one of the $44$ legal parent contacts
of \cref{sec:designtest}, whose translations are all even, so all parent origins lie in one coset
of $2\Z^3$ and halving puts them on a common lattice; by the identity of \cref{sec:designtest}
every halved contact belongs to $\atlas$. The halved baseline tiling is therefore a registered
chair tiling whose profiles match on every shared panel, and the small-collar realization lemma
(\cref{sec:existence}) turns it into a tiling by $\solid$. Translation covariance holds because
the parents of $T+v$ are the parents of $T$ translated by $v$ (\cref{thm:parent}), and the
origins are then divided by two; no phase is chosen, the actual parent origins are used.
\end{proof}

% ledger: R11 (E4)
The coset of $2\Z^3$ may be $(1,1,1)+2\Z^3$, so that $D(T)$ sits on a half-integer coset after
halving; periods are differences of poses within one coset and remain integer vectors, which
is all \cref{sec:aperiodicity} uses. $D$ is not a homothety of each cluster's corrugated surface
into $\solid$: it decodes the carrier and re-standardizes the surface by the same admissible
rule, which is why an exact self-similar corrugated macro-shape is not needed.

% ledger: T8, T9, N1, N2, R8; ruling R7; Lean carrier_hierarchy, carrier_hierarchy_unique, geometric_hierarchy_canonical, geometric_hierarchy_unique, nested_substitution_controls, hierarchy_cell_controls (T1n)
Iterating $D$ gives the hierarchy. Call a \emph{supertile of level $n$} a $2^n$-scaled chair in
a registered pose, and say that a family of partitions of the registered cells of $T$ into
level-$n$ supertiles, one partition for each $n$, is \emph{nested} if every level-$(n+1)$
supertile is exactly the union of its eight geometric child supertiles at level $n$.

\begin{theorem}[the intrinsic hierarchy; ruling R7; \tier{T1n}, unconditional:
\lean{carrier\_hierarchy}, \lean{carrier\_hierarchy\_unique},
\lean{geometric\_hierarchy\_canonical}, \lean{geometric\_hierarchy\_unique}]\label{thm:hierarchy}
Every tiling $T$ by $\solid$ carries a nested family of partitions into supertiles of every level,
with level $0$ the registered carrier tiling associated with $T$, and any two nested families
coincide level by level.
\end{theorem}

\begin{proof}
Existence: the level-$n$ partition is read off $D^n(T)$, each level-$n$ supertile being a chair
of $D^n(T)$ scaled by $2^n$ back to the coordinates of $T$; nesting is the definition of $D$.
Uniqueness: a nested family's level-$(n+1)$ partition is a partition of the level-$n$ supertiles
into parents, which \cref{thm:parent} applied to $D^n(T)$ determines; induct on $n$.
\end{proof}

% ledger: T8, M9 (closed); external review of the manuscript 2026-09-09, finding M2; pinned LogicalSpine.lean (LevelSupertilePose, carrier_hierarchy_unique, GeometricHierarchy, geometric_hierarchy_canonical, geometric_hierarchy_unique, carrierHierarchy_geometric); negative/NoncanonicalGeometricHierarchy.lean
The pinned Lean declarations prove this in two steps. \lean{carrier\_hierarchy\_unique} is stated
in the type whose level-$n$ poses belong to $D^n(T)$ (the structure \lean{LevelSupertilePose}
carries that membership). The structure \lean{GeometricHierarchy} carries no such membership: its
data are the raw poses, each level registered by one common rigid motion, with the nesting law of
the theorem; \lean{geometric\_hierarchy\_canonical} proves, by induction on the level from the
unique-parent theorem, that every pose of such a family lies in $D^n(T)$, and
\lean{geometric\_hierarchy\_unique} then gives uniqueness among arbitrary geometrically nested
registered families, \lean{carrierHierarchy\_geometric} supplying the existence half in the raw
type. The written induction above is the content of the first of these theorems. The raw
structure has no canonicity-membership field. The must-fail control
(\file{negative/NoncanonicalGeometricHierarchy.lean}) attempts to prove that a level-one pose is
noncanonical using \lean{geometric\_hierarchy\_canonical}; it checks that the positive membership
proof is rejected at the negated type.

% ledger: T9, N1, N2, L10; source SMKGS24 (fault lines); paper/scripts/seed_language_closure.py (T2); external review of the manuscript 2026-09-09, finding M1
Two things are deliberately not claimed. The theorem has no exhaustion clause: it does not say
that the supertiles containing one fixed tile grow to cover space, and they need not. The
substitution fixed point with seed $(6,0,4,5,2,3,0,1)$ (\cref{sec:limitperiodic}) has near the
origin two whole tiles in the frame $A=\operatorname{diag}(-1,-1,1)$, at the origin and at
$(1,1,-1)$; the first is its own child $000$ at every level, so its ancestors are the supertiles
$2^nA\carrier$, whose union is the closed octant $\{x\le0,\ y\le0,\ z\ge0\}$ and not all of
space, and the second is its own central child, its ancestors filling the closure of the
complement. This is a tiling with an infinite fault surface on supertile boundaries at all
levels, of the kind the hat paper contemplates for hats: ``it is conceivable that [tiles] could
tile infinite sectors of the plane, which could then be combined into tilings with infinite
`fault lines' that lie on the boundaries of supertiles at all levels'' \cite{SMKGS24}. It is
nevertheless substitution-legal, every finite patch of it occurring inside a supertile, so
non-exhaustion does not by itself take a tiling out of the substitution hull; whether every
tiling by $\solid$ lies in that hull is a separate question, left open (\cref{sec:remarks}).
Nor is the hierarchy claimed to
be locally derivable from the tiling by unmarked chairs: an unmarked chair has eight proper
orientations while a registered pose of $\solid$ is one of $24$, so the frame of a chair carries
a $\Z/3$ of information invisible to its geometry, and the parent is read from the features,
not from the chairs. What the theorem gives is exactly what \cref{sec:aperiodicity} needs: a
grouping at every scale that the tiling determines and that every symmetry of the tiling
therefore preserves. The same-frame grandchild at $(2,2,2)$ and the first two nested patches
($64$ and $4{,}096$ tiles, $448$ and $28{,}672$ cells) are checked in Lean as controls
(\cref{fig:nesting}).

\begin{figure}[ht]
  \centering
  \resizebox{0.8\linewidth}{!}{\input{generated/fig_nesting_slice}}
  \caption{The $0\le z<1$ unit-cell slice of the $64$-chair patch $4\carrier$; colours identify
  level-$1$ parents, and bold lines show level-$1$ and level-$2$ boundaries.}
  \label{fig:nesting}
\end{figure}

%% file: generated/first_shells_table.tex
% generated by paper/scripts/first_shells.py from certificates/candidate_certificate.json
\begin{tabular}{rlll}
\toprule
shell & atlas contacts & parent & completion \\
\midrule
1 & 3,9,10,12,20,29,38,40 & 7 & central \\
2 & 10,13,16,18,21,22,41 & 0 & completable \\
3 & 1,9,10,13,22,27,37,41 & 0 & completable \\
4 & 2,4,11,16,18,21,42 & 6 & completable \\
5 & 4,11,14,16,18,21,42 & 6 & impossible \\
6 & 11,14,15,16,18,21,42 & 2 & completable \\
7 & 4,11,16,18,21,22,42 & 6 & impossible \\
8 & 11,16,18,21,22,23,42 & 1 & completable \\
9 & 11,14,16,18,21,26,42 & 3 & impossible \\
10 & 11,16,18,21,22,26,42 & 3 & impossible \\
11 & 11,16,18,21,25,26,42 & 3 & completable \\
12 & 11,14,16,18,21,34,42 & 5 & impossible \\
13 & 11,16,18,21,22,34,42 & 5 & impossible \\
14 & 11,16,18,21,33,34,42 & 5 & completable \\
15 & 1,2,4,9,11,27,37,42 & 6 & completable \\
16 & 1,4,9,11,14,27,37,42 & 6 & impossible \\
17 & 1,9,11,14,15,27,37,42 & 2 & completable \\
\bottomrule
\end{tabular}
\par\smallskip
\begin{tabular}{rlll}
\toprule
shell & atlas contacts & parent & completion \\
\midrule
18 & 1,4,9,11,22,27,37,42 & 6 & impossible \\
19 & 1,9,11,22,23,27,37,42 & 1 & completable \\
20 & 1,9,11,14,26,27,37,42 & 3 & impossible \\
21 & 1,9,11,22,26,27,37,42 & 3 & impossible \\
22 & 1,9,11,25,26,27,37,42 & 3 & completable \\
23 & 1,9,11,14,27,34,37,42 & 5 & impossible \\
24 & 1,9,11,22,27,34,37,42 & 5 & impossible \\
25 & 1,9,11,27,33,34,37,42 & 5 & completable \\
26 & 4,11,16,18,21,42,43 & 6 & impossible \\
27 & 11,16,18,21,26,42,43 & 3 & impossible \\
28 & 11,16,18,21,34,42,43 & 5 & impossible \\
29 & 1,4,9,11,27,37,42,43 & 6 & impossible \\
30 & 1,9,11,26,27,37,42,43 & 3 & impossible \\
31 & 1,9,11,27,34,37,42,43 & 5 & impossible \\
32 & 11,16,18,21,42,43,44 & 4 & completable \\
33 & 1,9,11,27,37,42,43,44 & 4 & completable \\
\bottomrule
\end{tabular}

%% file: generated/fig_nesting_slice.tex
% generated by paper/scripts/nesting_slice.py from certificates/candidate_certificate.json (children)
\begin{tikzpicture}[x=1cm,y=1cm,font=\scriptsize]
  \fill[red!30] (0.0,0.0) rectangle (0.75,0.75);
  \draw[gray!60,thin] (0.0,0.0) rectangle (0.75,0.75);
  \fill[red!30] (0.0,0.75) rectangle (0.75,1.5);
  \draw[gray!60,thin] (0.0,0.75) rectangle (0.75,1.5);
  \fill[red!30] (0.0,1.5) rectangle (0.75,2.25);
  \draw[gray!60,thin] (0.0,1.5) rectangle (0.75,2.25);
  \fill[red!30] (0.0,2.25) rectangle (0.75,3.0);
  \draw[gray!60,thin] (0.0,2.25) rectangle (0.75,3.0);
  \fill[green!40] (0.0,3.0) rectangle (0.75,3.75);
  \draw[gray!60,thin] (0.0,3.0) rectangle (0.75,3.75);
  \fill[green!40] (0.0,3.75) rectangle (0.75,4.5);
  \draw[gray!60,thin] (0.0,3.75) rectangle (0.75,4.5);
  \fill[green!40] (0.0,4.5) rectangle (0.75,5.25);
  \draw[gray!60,thin] (0.0,4.5) rectangle (0.75,5.25);
  \fill[green!40] (0.0,5.25) rectangle (0.75,6.0);
  \draw[gray!60,thin] (0.0,5.25) rectangle (0.75,6.0);
  \fill[red!30] (0.75,0.0) rectangle (1.5,0.75);
  \draw[gray!60,thin] (0.75,0.0) rectangle (1.5,0.75);
  \fill[red!30] (0.75,0.75) rectangle (1.5,1.5);
  \draw[gray!60,thin] (0.75,0.75) rectangle (1.5,1.5);
  \fill[red!30] (0.75,1.5) rectangle (1.5,2.25);
  \draw[gray!60,thin] (0.75,1.5) rectangle (1.5,2.25);
  \fill[red!30] (0.75,2.25) rectangle (1.5,3.0);
  \draw[gray!60,thin] (0.75,2.25) rectangle (1.5,3.0);
  \fill[green!40] (0.75,3.0) rectangle (1.5,3.75);
  \draw[gray!60,thin] (0.75,3.0) rectangle (1.5,3.75);
  \fill[green!40] (0.75,3.75) rectangle (1.5,4.5);
  \draw[gray!60,thin] (0.75,3.75) rectangle (1.5,4.5);
  \fill[green!40] (0.75,4.5) rectangle (1.5,5.25);
  \draw[gray!60,thin] (0.75,4.5) rectangle (1.5,5.25);
  \fill[green!40] (0.75,5.25) rectangle (1.5,6.0);
  \draw[gray!60,thin] (0.75,5.25) rectangle (1.5,6.0);
  \fill[red!30] (1.5,0.0) rectangle (2.25,0.75);
  \draw[gray!60,thin] (1.5,0.0) rectangle (2.25,0.75);
  \fill[red!30] (1.5,0.75) rectangle (2.25,1.5);
  \draw[gray!60,thin] (1.5,0.75) rectangle (2.25,1.5);
  \fill[red!30] (1.5,1.5) rectangle (2.25,2.25);
  \draw[gray!60,thin] (1.5,1.5) rectangle (2.25,2.25);
  \fill[red!30] (1.5,2.25) rectangle (2.25,3.0);
  \draw[gray!60,thin] (1.5,2.25) rectangle (2.25,3.0);
  \fill[green!40] (1.5,3.0) rectangle (2.25,3.75);
  \draw[gray!60,thin] (1.5,3.0) rectangle (2.25,3.75);
  \fill[green!40] (1.5,3.75) rectangle (2.25,4.5);
  \draw[gray!60,thin] (1.5,3.75) rectangle (2.25,4.5);
  \fill[green!40] (1.5,4.5) rectangle (2.25,5.25);
  \draw[gray!60,thin] (1.5,4.5) rectangle (2.25,5.25);
  \fill[green!40] (1.5,5.25) rectangle (2.25,6.0);
  \draw[gray!60,thin] (1.5,5.25) rectangle (2.25,6.0);
  \fill[red!30] (2.25,0.0) rectangle (3.0,0.75);
  \draw[gray!60,thin] (2.25,0.0) rectangle (3.0,0.75);
  \fill[red!30] (2.25,0.75) rectangle (3.0,1.5);
  \draw[gray!60,thin] (2.25,0.75) rectangle (3.0,1.5);
  \fill[red!30] (2.25,1.5) rectangle (3.0,2.25);
  \draw[gray!60,thin] (2.25,1.5) rectangle (3.0,2.25);
  \fill[red!30] (2.25,2.25) rectangle (3.0,3.0);
  \draw[gray!60,thin] (2.25,2.25) rectangle (3.0,3.0);
  \fill[green!40] (2.25,3.0) rectangle (3.0,3.75);
  \draw[gray!60,thin] (2.25,3.0) rectangle (3.0,3.75);
  \fill[green!40] (2.25,3.75) rectangle (3.0,4.5);
  \draw[gray!60,thin] (2.25,3.75) rectangle (3.0,4.5);
  \fill[green!40] (2.25,4.5) rectangle (3.0,5.25);
  \draw[gray!60,thin] (2.25,4.5) rectangle (3.0,5.25);
  \fill[green!40] (2.25,5.25) rectangle (3.0,6.0);
  \draw[gray!60,thin] (2.25,5.25) rectangle (3.0,6.0);
  \fill[violet!30] (3.0,0.0) rectangle (3.75,0.75);
  \draw[gray!60,thin] (3.0,0.0) rectangle (3.75,0.75);
  \fill[violet!30] (3.0,0.75) rectangle (3.75,1.5);
  \draw[gray!60,thin] (3.0,0.75) rectangle (3.75,1.5);
  \fill[violet!30] (3.0,1.5) rectangle (3.75,2.25);
  \draw[gray!60,thin] (3.0,1.5) rectangle (3.75,2.25);
  \fill[violet!30] (3.0,2.25) rectangle (3.75,3.0);
  \draw[gray!60,thin] (3.0,2.25) rectangle (3.75,3.0);
  \fill[brown!40] (3.0,3.0) rectangle (3.75,3.75);
  \draw[gray!60,thin] (3.0,3.0) rectangle (3.75,3.75);
  \fill[brown!40] (3.0,3.75) rectangle (3.75,4.5);
  \draw[gray!60,thin] (3.0,3.75) rectangle (3.75,4.5);
  \fill[brown!40] (3.0,4.5) rectangle (3.75,5.25);
  \draw[gray!60,thin] (3.0,4.5) rectangle (3.75,5.25);
  \fill[brown!40] (3.0,5.25) rectangle (3.75,6.0);
  \draw[gray!60,thin] (3.0,5.25) rectangle (3.75,6.0);
  \fill[violet!30] (3.75,0.0) rectangle (4.5,0.75);
  \draw[gray!60,thin] (3.75,0.0) rectangle (4.5,0.75);
  \fill[violet!30] (3.75,0.75) rectangle (4.5,1.5);
  \draw[gray!60,thin] (3.75,0.75) rectangle (4.5,1.5);
  \fill[violet!30] (3.75,1.5) rectangle (4.5,2.25);
  \draw[gray!60,thin] (3.75,1.5) rectangle (4.5,2.25);
  \fill[violet!30] (3.75,2.25) rectangle (4.5,3.0);
  \draw[gray!60,thin] (3.75,2.25) rectangle (4.5,3.0);
  \fill[brown!40] (3.75,3.0) rectangle (4.5,3.75);
  \draw[gray!60,thin] (3.75,3.0) rectangle (4.5,3.75);
  \fill[brown!40] (3.75,3.75) rectangle (4.5,4.5);
  \draw[gray!60,thin] (3.75,3.75) rectangle (4.5,4.5);
  \fill[brown!40] (3.75,4.5) rectangle (4.5,5.25);
  \draw[gray!60,thin] (3.75,4.5) rectangle (4.5,5.25);
  \fill[brown!40] (3.75,5.25) rectangle (4.5,6.0);
  \draw[gray!60,thin] (3.75,5.25) rectangle (4.5,6.0);
  \fill[violet!30] (4.5,0.0) rectangle (5.25,0.75);
  \draw[gray!60,thin] (4.5,0.0) rectangle (5.25,0.75);
  \fill[violet!30] (4.5,0.75) rectangle (5.25,1.5);
  \draw[gray!60,thin] (4.5,0.75) rectangle (5.25,1.5);
  \fill[violet!30] (4.5,1.5) rectangle (5.25,2.25);
  \draw[gray!60,thin] (4.5,1.5) rectangle (5.25,2.25);
  \fill[violet!30] (4.5,2.25) rectangle (5.25,3.0);
  \draw[gray!60,thin] (4.5,2.25) rectangle (5.25,3.0);
  \fill[brown!40] (4.5,3.0) rectangle (5.25,3.75);
  \draw[gray!60,thin] (4.5,3.0) rectangle (5.25,3.75);
  \fill[brown!40] (4.5,3.75) rectangle (5.25,4.5);
  \draw[gray!60,thin] (4.5,3.75) rectangle (5.25,4.5);
  \fill[brown!40] (4.5,4.5) rectangle (5.25,5.25);
  \draw[gray!60,thin] (4.5,4.5) rectangle (5.25,5.25);
  \fill[brown!40] (4.5,5.25) rectangle (5.25,6.0);
  \draw[gray!60,thin] (4.5,5.25) rectangle (5.25,6.0);
  \fill[violet!30] (5.25,0.0) rectangle (6.0,0.75);
  \draw[gray!60,thin] (5.25,0.0) rectangle (6.0,0.75);
  \fill[violet!30] (5.25,0.75) rectangle (6.0,1.5);
  \draw[gray!60,thin] (5.25,0.75) rectangle (6.0,1.5);
  \fill[violet!30] (5.25,1.5) rectangle (6.0,2.25);
  \draw[gray!60,thin] (5.25,1.5) rectangle (6.0,2.25);
  \fill[violet!30] (5.25,2.25) rectangle (6.0,3.0);
  \draw[gray!60,thin] (5.25,2.25) rectangle (6.0,3.0);
  \fill[brown!40] (5.25,3.0) rectangle (6.0,3.75);
  \draw[gray!60,thin] (5.25,3.0) rectangle (6.0,3.75);
  \fill[brown!40] (5.25,3.75) rectangle (6.0,4.5);
  \draw[gray!60,thin] (5.25,3.75) rectangle (6.0,4.5);
  \fill[brown!40] (5.25,4.5) rectangle (6.0,5.25);
  \draw[gray!60,thin] (5.25,4.5) rectangle (6.0,5.25);
  \fill[brown!40] (5.25,5.25) rectangle (6.0,6.0);
  \draw[gray!60,thin] (5.25,5.25) rectangle (6.0,6.0);
  \draw[very thick] (0.0,0.0) -- (0.0,0.75);
  \draw[very thick] (0.0,0.0) -- (0.75,0.0);
  \draw[very thick] (0.0,0.75) -- (0.0,1.5);
  \draw[very thick] (0.75,0.0) -- (1.5,0.0);
  \draw[very thick] (0.0,2.25) -- (0.0,3.0);
  \draw[very thick] (0.0,3.0) -- (0.75,3.0);
  \draw[very thick] (0.0,1.5) -- (0.0,2.25);
  \draw[very thick] (0.75,3.0) -- (1.5,3.0);
  \draw[very thick] (3.0,0.0) -- (3.0,0.75);
  \draw[very thick] (2.25,0.0) -- (3.0,0.0);
  \draw[very thick] (1.5,0.0) -- (2.25,0.0);
  \draw[very thick] (3.0,0.75) -- (3.0,1.5);
  \draw[very thick] (3.0,2.25) -- (3.0,3.0);
  \draw[very thick] (2.25,3.0) -- (3.0,3.0);
  \draw[very thick] (3.0,1.5) -- (3.0,2.25);
  \draw[very thick] (1.5,3.0) -- (2.25,3.0);
  \draw[very thick] (0.0,5.25) -- (0.0,6.0);
  \draw[very thick] (0.0,6.0) -- (0.75,6.0);
  \draw[very thick] (0.0,4.5) -- (0.0,5.25);
  \draw[very thick] (0.75,6.0) -- (1.5,6.0);
  \draw[very thick] (0.0,3.0) -- (0.0,3.75);
  \draw[very thick] (0.0,3.0) -- (0.75,3.0);
  \draw[very thick] (0.0,3.75) -- (0.0,4.5);
  \draw[very thick] (0.75,3.0) -- (1.5,3.0);
  \draw[very thick] (3.0,5.25) -- (3.0,6.0);
  \draw[very thick] (2.25,6.0) -- (3.0,6.0);
  \draw[very thick] (1.5,6.0) -- (2.25,6.0);
  \draw[very thick] (3.0,4.5) -- (3.0,5.25);
  \draw[very thick] (3.0,3.0) -- (3.0,3.75);
  \draw[very thick] (2.25,3.0) -- (3.0,3.0);
  \draw[very thick] (1.5,3.0) -- (2.25,3.0);
  \draw[very thick] (3.0,3.75) -- (3.0,4.5);
  \draw[very thick] (6.0,0.0) -- (6.0,0.75);
  \draw[very thick] (5.25,0.0) -- (6.0,0.0);
  \draw[very thick] (4.5,0.0) -- (5.25,0.0);
  \draw[very thick] (6.0,0.75) -- (6.0,1.5);
  \draw[very thick] (3.0,0.0) -- (3.0,0.75);
  \draw[very thick] (3.0,0.0) -- (3.75,0.0);
  \draw[very thick] (3.75,0.0) -- (4.5,0.0);
  \draw[very thick] (3.0,0.75) -- (3.0,1.5);
  \draw[very thick] (6.0,2.25) -- (6.0,3.0);
  \draw[very thick] (5.25,3.0) -- (6.0,3.0);
  \draw[very thick] (6.0,1.5) -- (6.0,2.25);
  \draw[very thick] (4.5,3.0) -- (5.25,3.0);
  \draw[very thick] (3.0,2.25) -- (3.0,3.0);
  \draw[very thick] (3.0,3.0) -- (3.75,3.0);
  \draw[very thick] (3.75,3.0) -- (4.5,3.0);
  \draw[very thick] (3.0,1.5) -- (3.0,2.25);
  \draw[very thick] (6.0,5.25) -- (6.0,6.0);
  \draw[very thick] (5.25,6.0) -- (6.0,6.0);
  \draw[very thick] (6.0,4.5) -- (6.0,5.25);
  \draw[very thick] (4.5,6.0) -- (5.25,6.0);
  \draw[very thick] (6.0,3.0) -- (6.0,3.75);
  \draw[very thick] (5.25,3.0) -- (6.0,3.0);
  \draw[very thick] (6.0,3.75) -- (6.0,4.5);
  \draw[very thick] (4.5,3.0) -- (5.25,3.0);
  \draw[very thick] (3.0,5.25) -- (3.0,6.0);
  \draw[very thick] (3.0,6.0) -- (3.75,6.0);
  \draw[very thick] (3.75,6.0) -- (4.5,6.0);
  \draw[very thick] (3.0,4.5) -- (3.0,5.25);
  \draw[very thick] (3.0,3.0) -- (3.0,3.75);
  \draw[very thick] (3.0,3.0) -- (3.75,3.0);
  \draw[very thick] (3.0,3.75) -- (3.0,4.5);
  \draw[very thick] (3.75,3.0) -- (4.5,3.0);
  \draw[line width=1.6pt,black] (0,0) rectangle (6.0,6.0);
  \node[anchor=north,font=\small,align=center] at (3.0,-0.15) {slice $0\le z<1$ of the $64$-chair patch $4\carrier$ (side $8$):\\ cells coloured by level-1 supertile; thick lines: level-1 outlines, level-2 outline};
\end{tikzpicture}

%% file: sec7_aperiodicity.tex
\section{Aperiodicity and homochirality}\label{sec:aperiodicity}

% ledger: T13, T2, T3, T4, T5, N3; source PROOFS_registered.md R§10; Lean period_halving, all_periods_grid, no_period, sym_card_le_24, r44_einstein (spine); 1711.03401 L8 register
\begin{theorem}[no periods; R\S10; \tier{T1n}, unconditional:
\lean{r44\_einstein}, using \lean{period\_halving} and
\lean{all\_periods\_grid}; terminal arithmetic lemma
\lean{no\_period}: \tier{T1}]\label{thm:noperiod}
Every tiling $T$ by $\solid$ has $\Per(T)=\{0\}$.
\end{theorem}

\begin{proof}
Let $p$ be a translational period of $T$. Registration (\cref{thm:registration}) and the absence
of nonidentity self-isometries (\cref{app:asymmetry}) make $p$ an integer vector: the translation by $p$ carries
a registered tile to a registered tile in the same frame, and the difference of two integer
shifts is an integer vector. The parent partition is translation-covariant (\cref{thm:parent}),
so $D(T+p)=D(T)+p/2$ by \cref{thm:coarsening}; since $T+p=T$, $p/2$ is a period of $D(T)$, which
is again a registered tiling by $\solid$. Iterating, $p/2^n$ is a period of $D^n(T)$ and hence an
integer vector for every $n$, so $p\in\bigcap_n 2^n\Z^3=\{0\}$.
\end{proof}

This is the classical argument that a unique hierarchy forbids periods, in the form the Spectre
paper states \cite{SMKGS24b} and in the form Robinson's tilings first exhibited: a period would
have to be a period of every level, and the levels grow without bound.

\begin{theorem}[finite symmetry; R\S10; \tier{T1n}, unconditional:
\lean{r44\_einstein}; conditional bound
\lean{sym\_card\_le\_24}: \tier{T1}]\label{thm:sym}
Every tiling $T$ by $\solid$ has $|\Sym(T)|\le24$.
\end{theorem}

\begin{proof}
After registration every tile's frame lies in one coset of the $24$-element proper cubic group,
and $\solid$ has no self-isometry, so a symmetry of $T$ is determined by where it sends one
chosen tile: its linear part is one of at most $24$ frames, and, that being fixed, its
translation part is determined as well. The map from $\Sym(T)$ to linear parts has kernel the
translational symmetries, which are trivial by \cref{thm:noperiod}; hence it is injective into a
set of at most $24$ elements.
\end{proof}

\begin{corollary}[strong aperiodicity; a deduction from the two theorems above]\label[corollary]{cor:strong}
No tiling by $\solid$ has a symmetry of infinite order. Hence $\solid$ is a strongly aperiodic
monotile of $\R^3$ in the sense of Mozes and of the hat paper, and \cref{thm:main} holds.
\end{corollary}

\begin{proof}
A finite group has no element of infinite order; existence is \cref{thm:existence}.
\end{proof}

In the taxonomy of Coulbois, Gajardo, Guillon and Lutfalla, \cref{thm:sym} says that $\solid$
is mildly aperiodic. Strong aperiodicity in their sense would require every tiling to have
trivial symmetry group, and $\solid$ does not have that property: the substitution fixed points
of \cref{prop:seedsym} include tilings whose symmetry group has order $8$, as well as tilings
with trivial symmetry group. In particular no screw motion, rational or
irrational, is a symmetry of any tiling by $\solid$, which is the property the biprism lacks.

% ledger: T6, T7, D5; Lean tiling_homochiral, tiling_chirality_corollary (spine); source 2305.17743 L22
\begin{theorem}[homochirality; \lean{tiling\_homochiral}, \lean{tiling\_chirality\_corollary}:
\tier{T1n}, unconditional]\label{thm:homochiral}
In every tiling $T$ by $\solid$ all placements have the same handedness: the determinants of
their linear parts are all $+1$ or all $-1$. Consequently a reflected and an unreflected copy of
$\solid$ never occur in the same tiling, and every tiling is homochiral.
\end{theorem}

\begin{proof}
After the ambient isometry of \cref{thm:registration}, every placement is a proper cubic frame
followed by a translation, so all placements have determinant $+1$ relative to that isometry;
undoing it multiplies all determinants by the same sign. For two tiles $g\solid$ and $h\solid$
the relative motion $hg^{-1}$ then has determinant $+1$, which is the Spectre paper's definition
of a homochiral tiling.
\end{proof}

\begin{corollary}[strictly chiral aperiodic monotile; a deduction at the same verification
boundary]\label[corollary]{cor:chiral}
$\solid$ is a strictly chiral aperiodic monotile of $\R^3$ in the sense of the Spectre paper: it
admits tilings, and every tiling it admits is homochiral and non-periodic.
\end{corollary}

Reflections were never excluded by convention: all $48$ signed frames were admitted as candidate
poses in \cref{sec:companions,sec:registration}, and the $44$ legal contacts turned out proper.
A tiling may be reflected as a whole; what the geometry forbids is a mixture. The
Schmitt--Conway--Danzer biprism, by contrast, needs reflections forbidden for its aperiodicity
\cite{ST12}. The four clauses of \cref{thm:full} are now proved: (1) is \cref{thm:existence},
(2) is \cref{thm:noperiod,thm:sym}, (3) is \cref{thm:homochiral}, and (4) is
\cref{thm:hierarchy}.

%% file: sec8_mechanization.tex
\section{Mechanization and the limit-periodic corollary}\label{sec:mechanization}

\subsection{What is checked by machine}\label{sec:lean}

% ledger: M1, M2, T12, D4, N8; source lean/R44/AXIOMS.md, HYPOTHESES.md, THEOREM.md at the pinned commit; profile §7
The finite parts of the proof are theorems of a Lean~4 development written for this paper, with
Mathlib as its library, and the logical assembly of the whole argument is a Lean theorem as
well. \Cref{tab:finite} lists the $23$ finite theorems with their statements and the method by
which each is decided: by the kernel's evaluator (\lean{decide}) where that is feasible, and
otherwise by \lean{native\_decide}, which compiles the decision procedure and trusts its result
through one named axiom per theorem. Those axioms are the only non-standard ones in the
development: \cref{tab:axioms} gives, for every finite theorem and for the main theorems, the
output of Lean's \lean{\#print axioms}, which for the main theorem consists of the three standard
axioms of classical Mathlib and the $21$ named hooks of the finite checks it uses. No theorem in
the development depends on \lean{sorryAx}, and no admission remains.

% ledger: T12, M2b (closed), D4; the honest sentence; the Hypotheses record (empty)
The assembly is the theorem \lean{r44\_einstein}, which takes no hypothesis: it states that a
tiling exists, that every tiling has $\Per(T)=\{0\}$, and that $|\Sym(T)|\le24$; the theorems
\lean{tiling\_homochiral} and \lean{carrier\_hierarchy}, with
\lean{geometric\_hierarchy\_unique}, state clauses (3) and (4) of \cref{thm:full}. The
development was written with a record \lean{Hypotheses} of named written lemmas as its
verification boundary, and the written lemmas of this paper are the fields that record had: the
tube construction and carrier recovery in \cref{app:tubes,app:asymmetry}, realization in \cref{sec:existence},
and the companion and registration arguments in \cref{sec:companions,sec:registration}. Every one
of them is now a Lean theorem of the pinned development (\cref{tab:hypotheses} lists the $19$
theorems, with their tiers by their own axiom lines), the record is empty, and the conditional
form survives under the name \lean{r44\_einstein\_of\_hypotheses} over the empty record
(\cref{app:lean}). Everything else, including the period-halving tower, the symmetry bound, local
finiteness, the unique parent, the coarsening and the hierarchy with its uniqueness among
arbitrary geometrically nested registered families (\cref{thm:hierarchy}), is proved inside
Lean. The honest summary is therefore this: the theorem is kernel-checked modulo the named
compiler hooks; the written proofs of \cref{app:solidchecks,sec:existence} and
\cref{sec:companions,sec:registration} remain as exposition, independently reviewed, and four of
the formal proofs argue differently from the written ones, as noted where those lemmas appear
(carrier recovery, containment rigidity, the covering by a component's solids, and small-collar
realization). The same pipeline was exercised on the hat as a calibration of the tooling
(\cref{app:calibration}). The hat and Spectre papers rely on no proof assistant; the nearest comparator is Myers's
in-progress formalization of those papers \cite{Mye24lean}, and we claim no coverage by Mathlib
itself \cite{Lean4}.

% ledger: M3, M4, M5, M8; source APPENDIX_DATA.md; verify/replay.py; controls.sh
Independently of Lean, every enumeration is replayed from the Python standard library by one
command in about a minute (\cref{app:data}), and the companion census was checked by two
independent Python implementations; the second extracts the features from the actual
triangulated mesh rather than from their annotation table. No floating-point value enters any
checked statement: certificates are integers and exact rationals, the checkers use exact
arithmetic, and the Lean theorems decide integer literals. Negative controls accompany the
positive build. The diagonal audit distinguishes assertions about distinct companion tiles from
relations with legitimate reflexive cases; the controls exercise the repaired diagonal
application, corrupted finite inputs, mixed-handedness claims, a second-hierarchy claim and a
claim that a raw geometric hierarchy has a noncanonical level-one pose, each
checked to fail with its expected diagnostic. Six corrupted inputs to the replay are rejected. The
development was reviewed adversarially twice before this paper was written; the
first round found the hypothesis record inconsistent (one field admitted comparing a tile with
itself), which was repaired and controlled, and the second round found no blocking or major
issue, while noting that the reviewer could not run Lean. A third round, on the closed record with
the empty hypothesis record, was requested on 2026-09-10; as of that date its verdict is
pending. We make no claim of simplicity, and we
have one route to the theorem: the continuous argument of \cref{sec:companions} is written for
checking by hand with its arithmetic in Lean, and the finite part is where the computer does
the work.

\emph{Code.} The repository containing the solid, the certificates, the replay, the Lean
development and the viewer accompanies this paper (\cref{app:data}); its mathematical baseline is commit
\texttt{d90313a717}, the intended release tag is \texttt{paper-v1}, and tag creation and licensing
remain pending owner actions.

% ledger: D4, T12, M2, M6, M7; the freeze (step 18) and the re-pin (step 19)
\emph{Version freeze.} This review version uses commit \texttt{d90313a717} (2026-09-10) as its
mathematical and Lean baseline. The $19$ theorems of \cref{tab:hypotheses}, the empty record, and
the recorded axiom dependencies of \cref{tab:axioms} refer to that baseline; the geometric
certificate data and Lean statements refer to that baseline, and the bibliography and replay
measurements have the separately identified sources recorded in the paper. The intended
paper-release tag is \texttt{paper-v1}, to be placed on the paper-release commit; tag creation
and licensing remain pending owner actions. The theorem, the written proof and the finite
certificates are unchanged since the earlier baseline \texttt{dd2735d9bd} (2026-09-08), at which
$16$ of the $19$ written lemmas were still hypotheses of the Lean assembly.

% ledger: N9; simulations/periodicity_search.py; simulations/periodicity_search_report.json (commit d90313a717); simulations/review/2026-09-09_periodicity_search_r1_eddy.md (controls run)
\emph{Evidence, not proof.} Independently of the proof, and used nowhere in it, a theorem-free
search was run over grid-registered copies of $\solid$ (all $48$ signed frames, reflections
included; copies compatible when their carrier cells are disjoint and the features on every
shared unit panel are exact opposites, in rational arithmetic;
\file{simulations/periodicity\_search.py}). A SAT solver was asked, for every full-rank sublattice
$L\subset\Z^3$ of index at most $40$, whether the torus $\R^3/L$ can be tiled by such copies;
all $43{,}981$ instances are unsatisfiable, so no tiling by registered copies has a full-rank
period lattice of index at most $40$. Heesch-style coronas were grown around one copy to radius
$5$ of $5$ tried, with $251$ copies at radius $5$. The positive controls of the same script find
the unit cube periodic at index $1$ and the featureless chair carrier periodic at index $7$,
as recorded in \file{simulations/review/2026-09-09_periodicity_search_r1_eddy.md}.
The torus and corona receipt is \file{simulations/periodicity\_search\_report.json}
(wall time $108{,}682$ s). This is
evidence in the sense of a sanity check on the finite data: it concerns registered copies and
full-rank lattices only, whereas \cref{thm:noperiod} excludes every period of every tiling.

\input{generated/finite_theorems_table}

\input{generated/hypotheses_table}

\input{generated/axioms_table}

\subsection{The limit-periodic corollary}\label{sec:limitperiodic}

% ledger: L1, L2, L3, L4, L5, L6, L7, Q15, Q16, Q17, Q18, Q19, N1; DECISIONS Q-A (fixed-point wording), Q-B (seed); Lean substitution_modular_coincidence (T1n); verify/substitution_modular_coincidence.py (T2); paper/scripts/seed_language_closure.py (T2)
The substitution of \cref{sec:substitution} is a lattice substitution in the sense of Lee and
Moody \cite{LM01}. Label each unit cell of a registered tiling by the frame of the chair that
covers it and by which of the seven cells of that chair it is: $24\times7=168$ labels. The
refinement replaces each labelled cell by its eight labelled children at twice the scale, so the
labelling is a substitution on $\Z^3$ with inflation $2$ and $168$ symbols, and its $168\times168$
substitution matrix has constant column sum $8=|\det 2I|$.

\begin{theorem}[modular coincidence; \tier{T1n}, replayed in Python]\label{thm:coincidence}
The $168$-label substitution is total and primitive, with least primitive exponent $N=3$, and it
has a modular coincidence of least depth $M=3$: at the address $a=(0,0,2)$ modulo $2^3$, the
level-$3$ descendant of every one of the $168$ labels has label $i=78$.
\end{theorem}

The Lean theorem is \lean{substitution\_modular\_coincidence}. The check tracks the full set of
$168$ labels from the root, so the coincidence is over all labels, not along a single orbit; the integers $N$ and $M$ are independent and happen to
coincide.

% ledger: L3, L4, L9, L10, N1, T9; DECISIONS Q-B; Lean seed_language_closure (T1n); verify/replay.py step seed_language_closure (T2); paper/scripts/seed_language_closure.py (T2); external review of the manuscript 2026-09-09, findings B1 and M1 (paper/review/EXTERNAL_RULING_2026-09-09.md)
A legal fixed point on the whole of $\Z^3$ exists. A substitution-legal fixed labelling of
$\Z^3$ is determined by its seed on $\{-1,0\}^3$: the cell $-\delta$, $\delta\in\{0,1\}^3$, has
parent $-\delta$ and residue $\delta$, so its eight labels satisfy their fixed-residue equations;
conversely, a seed in the substitution language satisfying those equations yields such a
labelling, whereas a block satisfying the equations but outside the language need not (one such
block encodes inconsistent whole-chair labels). The substitution language
contains $27$ fixed seed blocks on $\{-1,0\}^3$, in proper-rotation orbits of sizes $24$ and $3$;
iterating any one gives a legal substitution-fixed labelling of $\Z^3$ (\cref{fig:coincidence}).
The count is by exact closure of the language of legal $2\times2\times2$ blocks, which grows
through $168$, $600$, $1{,}278$, $1{,}398$ and $1{,}410$ blocks under successive substitution and
then is stable, and every seed is checked to reproduce itself in place under iteration (the
theorem \lean{seed\_language\_closure}, \tier{T1n}, whose statement lists the five closure sizes,
the count $27$, the orbit sizes $24$ and $3$, the stabilizer histogram, in-place reproduction and
commutation with the $24$ frames; the replay step of the same name in \file{verify/replay.py};
and, independently, \file{paper/scripts/seed\_language\_closure.py}, standard library,
\cref{app:data}); a bounded
search inside level-$3$ patches, which an earlier draft reported as the count $24$, misses the
three seeds of the small orbit, and the external review of the manuscript found them. Every
finite patch of a fixed labelling occurs in some $\sigma^n(j)$ and hence, by primitivity, inside
a refinement of the native chair, that is, inside the tiling of \cref{thm:existence}; so the
chair poses the labelling encodes form a registered chair tiling of $\R^3$ all of whose contacts
lie in $\atlas$, and small-collar realization (\cref{sec:existence}) turns it into a tiling by
$\solid$ at those poses. We write $T_w$ for the tiling determined by the seed $w$.

\begin{proposition}[symmetries of the fixed points; finite data \tier{T1n}
(\lean{seed\_language\_closure}) and \tier{T2}, the deduction
written]\label[proposition]{prop:seedsym}
Let $G$ be the $24$-element proper cubic group, acting on the seed cube by moving the cells and
composing the frames in the labels. For each of the $27$ fixed seeds $w$, $\Sym(T_w)$ is the
stabilizer of $w$ in $G$. The stabilizers have order $1$ for the $24$ seeds of the large orbit
and order $8$ for the three seeds of the small orbit; for the seed $(28,84,91,35,98,42,49,105)$,
in the repository's label order, the stabilizer is generated by $(x,y,z)\mapsto(z,y,-x)$ and
$(x,y,z)\mapsto(-x,-y,z)$. Hence some tilings by $\solid$ have trivial symmetry group and some
have symmetry group of order $8$.
\end{proposition}

\begin{proof}
Let $g=(R,t)$, $x\mapsto Rx+t$, be a symmetry of $T_w$. As in the proof of \cref{thm:sym},
registration and the absence of self-isometries give $R\in G$ and $t\in\Z^3$. The parent
partition is read from the tiling alone, and the children of a rotated pose are the rotated
children (a finite check over $G$ in the same script), so by the uniqueness of parents
(\cref{thm:parent}) the parents of $RT$ are the rotated parents of $T$; with the translation
clause of \cref{thm:coarsening}, $D(gT)=(R,t/2)\,D(T)$. The seed being fixed in place,
$D(T_w)=T_w$, so $(R,t/2)$ is again a symmetry of $T_w$ and $t/2\in\Z^3$; iterating,
$t\in\bigcap_n2^n\Z^3=\{0\}$. Thus $\Sym(T_w)=\{R\in G: RT_w=T_w\}$. The substitution commutes
with the action of $G$ on labelled cells (the same finite check), so $RT_w$ is the fixed point
with seed $Rw$; a fixed point is determined by its seed, so $RT_w=T_w$ exactly when $Rw=w$. The
orders and the generators are read off the $27$ seeds.
\end{proof}

The seed $(6,0,4,5,2,3,0,1)$ of the large orbit is the one whose ancestors fill only an octant
(\cref{sec:hierarchy}); its tiling has trivial symmetry group. The three seeds of the small orbit
each consist of eight whole tiles at the origin, one in each octant, each its own child $000$ at
every level.

% ledger: D9, Q16; source math0002019 L102-L107 (Lee-Moody definitions)
We recall the definitions the corollary uses, from Lee and Moody \cite{LM01}. A cut-and-project
scheme consists of a lattice $\widetilde L\subset\R^3\times G$, where $G$ is a locally compact
abelian group, such that the physical projection $\pi_1$ is injective on $\widetilde L$ and the
internal projection $\pi_2$ has dense image in $G$. A model set is, up to translation,
$\{\pi_1(\ell):\ell\in\widetilde L,\ \pi_2(\ell)\in W\}$, where the window $W\subset G$ has
nonempty interior and compact closure; it is regular when $\partial W$ has Haar measure zero.
Here the internal group is $\varprojlim_n\Z^3/2^n\Z^3$, the $2$-adic completion. A modular
coincidence occurs when, at some substitution depth and residue address, all substitution maps
at that address have the same output label.

% ledger: L1, L2, L3, L4, L5, L6, L7, Q15, Q16, Q17, Q18, Q19, N1; DECISIONS Q-A (fixed-point wording); the corollary and the paragraph after it
\begin{corollary}[limit-periodic structure of the fixed point; the cited imports are \tier{T3}]\label[corollary]{cor:limitperiodic}
By Lee and Moody's Theorem 3 \cite{LM01}, applied in the direction (iv)~$\Rightarrow$~(ii) with
its hypotheses verified as above (\cref{thm:coincidence}: primitivity; Perron--Frobenius eigenvalue $8=|\det 2I|$; the
label classes of the fixed point partition $\Z^3$), the $168$ label classes of the substitution
fixed point are regular model sets for the $2$-adic cut-and-project scheme of \cite{LM01}; by
Schlottmann's theorem, in the form stated by Lee, Moody and Solomyak
\cite[Theorem~5.11]{LMS03}, each is pure point diffractive; in the terminology of Baake, Moody and
Schlottmann and of Baake and Grimm \cite{BMS98,BG10,BG13} the structure is limit-periodic.
\end{corollary}

The corollary is a statement about the fixed point, not about every tiling by $\solid$: whether
every tiling lies in the substitution hull is the fault-line question left open in
\cref{sec:hierarchy}, and Lee and Moody's own passage from the substitution to every tiling
obeying the Taylor--Socolar matching rules is a separate three-part argument \cite{LM13}, not
an application of their Theorem 3. In Socolar's words about the hat, whose structure ``is
qualitatively different from the limit-periodic structure of the hexagonal Taylor--Socolar
monotile tiling, which has Bragg diffraction peaks at wavelengths of the form $2^{-n}k_0$ for
arbitrarily large $n$'' \cite{Soc23}, the fixed-point label-class structure established here is
limit-periodic with a $2$-adic internal space, rather than quasicrystalline in the sense
contrasted by Socolar.
Akiyama and Lee verified pure point spectrum for a $168$-prototile Taylor--Socolar substitution by
an overlap coincidence \cite{AL14}, and Lee and Moody checked the sphinx's modular coincidence by
computer at depth $8$ \cite{LM01}; ours is checked at depth $3$ in Lean and replayed in Python.

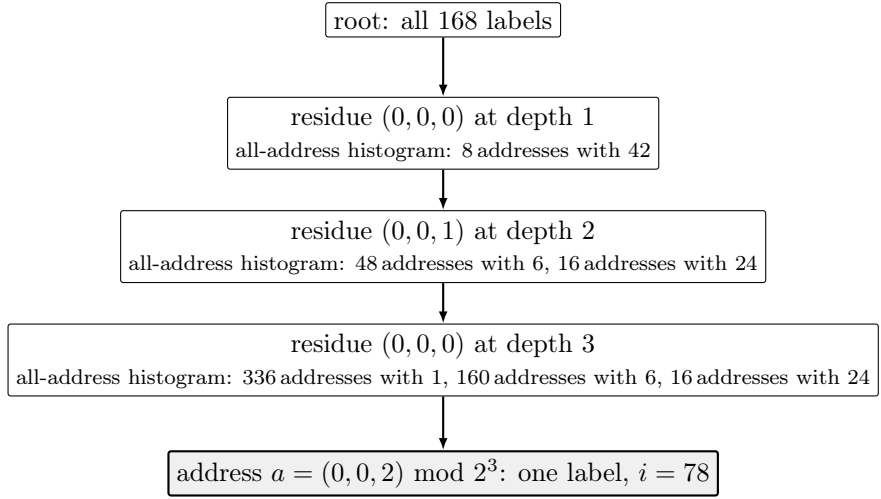
\begin{figure}[ht]
  \centering
  \input{generated/fig_coincidence}
  \caption{The modular-coincidence address and the all-address distributions of label-set sizes
  at depths $1$--$3$.}
  \label{fig:coincidence}
\end{figure}

%% file: generated/finite_theorems_table.tex
% generated by paper/scripts/finite_theorems_table.py; from lean/R44/R44/Theorems.lean docstrings and the AXIOMS.md ledger at d90313a717
{\small
\begin{longtable}{L{0.22\linewidth} L{0.48\linewidth} L{0.17\linewidth}}
\caption{The finite theorems of the Lean development at the state described in this paper: statement in words (the theorem's docstring) and decision method (FIGURES.md, Table 1).}\label{tab:finite}\\
\toprule
theorem & statement (docstring) & method \\
\midrule
\endfirsthead
\toprule
theorem & statement (docstring) & method \\
\midrule
\endhead
\bottomrule
\endfoot
\lean{children\_partition\_2P} & The eight literal child poses are proper and their 56 integer cells are a disjoint, exact partition of the doubled seven-cube chair. & kernel \lean{decide} \\
\lean{contact\_closure\_30} & Refining the 21 internal face contacts once gives 30 states, and one more refinement adds no state. The result is the literal 30-state certificate. & \lean{native\_decide} \\
\lean{profile\_canonical} & The 30 states induce 372 distinct sign equations. Their listed twelve connected components are balanced, rooted positively at their least roles, and ordered by those roots; canonical signs and component magnitudes give both independent 192-entry profile literals. & \lean{native\_decide} \\
\lean{atlas\_44} & Among all 2,388 disjoint touching shell poses in all 48 signed frames, profile compatibility admits exactly the literal 44-pose atlas. & \lean{native\_decide} \\
\lean{first\_shells\_33} & Exhaustive exact-cover recursion gives precisely the 33 literal covers of the 22-cell shell, and every cover has one candidate parent signature. & \lean{native\_decide} \\
\lean{central\_completion} & Of the 32 outer-root shells, exactly 14 have a compatible completion and 18 have none; every viable completing shell types the centre as central. & \lean{native\_decide} \\
\lean{parent\_conflicts\_28} & Every one of the 28 unordered pairs among the eight parents through a root contains an incompatible pair of child poses. & \lean{native\_decide} \\
\lean{parent\_atlas\_eq\_fine} & The induced parent census is 697 candidates, 116 disjoint macro-bodies and 44 admitted contacts; all admitted shifts are even and halving gives the fine 44-contact atlas extensionally. & \lean{native\_decide} \\
\lean{mates\_census} & The literal complete-feature census has 6,862 poses, of which 5,317 avoid the root baseline. The 5,273 rejection records quantify over every isolated companion option (299,975 in all); every selected overlap retains side at least 42/\allowbreak{}400 = 21/\allowbreak{}200. The 44 survivors are exactly the registered atlas. & \lean{native\_decide} \\
\lean{deviations\_distinct} & For 1 <= i,j <= 12 the base/\allowbreak{}ridge deviation equation has no solution. For L=3/\allowbreak{}25, both requested rational inequalities hold after clearing 25\textasciicircum{}2. & kernel \lean{decide} \\
\lean{no\_native\_symmetry} & Among all 48 signed coordinate permutations, only the identity preserves the centres, outward normals and coefficients of the 192-feature table. & kernel \lean{decide} \\
\lean{native\_features\_length} & The literal table has exactly the 192 roles used by the geometric model. & kernel \lean{decide} \\
\lean{native\_feature\_centers\_nodup} & Distinct literal roles have distinct eighth-grid feature centres. & kernel \lean{decide} \\
\lean{native\_feature\_normal\_axis} & Every literal panel normal has unit coefficient on its selected axis. & kernel \lean{decide} \\
\lean{solid\_mesh\_exact} & The 2,138 literal vertices and 4,272 indexed triangles are a connected, closed oriented 6,408-edge surface; it is exactly the panel/\allowbreak{}feature construction triangle-by-triangle and has signed volume seven. & \lean{native\_decide} \\
\lean{boundary\_sphere} & Every indexed boundary edge has exactly two incident triangles, the edge-sharing triangle graph is connected, every vertex link is one cycle, and the computed counts satisfy V - E + F = 2138 - 6408 + 4272 = 2. The classification theorem for connected closed surfaces therefore makes this PL surface a 2-sphere; the PL Schoenflies (Alexander) theorem then makes the bounded solid a closed 3-ball. These classical bridge theorems are imported at T3; see E. E. Moise, *Geometric Topology in Dimensions 2 and 3*, and C. P. Rourke--B. J. Sanderson, *Introduction to Piecewise-Linear Topology*. & \lean{native\_decide} \\
\lean{nested\_substitution\_controls} & The same-frame grandchild occurs at (2,2,2), and the first two nested substitution controls have 64/\allowbreak{}4,096 poses and 448/\allowbreak{}28,672 disjoint cells. & \lean{native\_decide} \\
\lean{hierarchy\_cell\_controls} & The root double refinement contains all 448 cells \texttt{4a+b} of the seven four-by-four-by-four blocks, and every signed-permutation lower-corner action permutes the 64 offsets in the form required by equivariance. & \lean{native\_decide} \\
\lean{registered\_shell\_cell\_controls} & The 192 panel features map onto the 22 shell cells; every legal-contact frame is one of the 48 signed permutations; and an exact opposite center/\allowbreak{}normal/\allowbreak{}coefficient mate owns the cell beyond the root panel. & \lean{native\_decide} \\
\lean{orientation\_group\_24} & The 19 contact frames generate exactly the literal 24-element proper cubic rotation group. & \lean{native\_decide} \\
\lean{substitution\_modular\_coincidence} & The cell-labelled R44 chair substitution is a total constant-length substitution on 168 labels. Its 168$\times$168 matrix has column sum eight and least positive power three; its least modular-coincidence depth is three. The first lexicographic witness is address \texttt{(0,0,2)} and label 78, namely generated frame 11 and local chair cell 1. Lee--Moody, *Lattice substitution systems and model sets*, Theorem 3, supplies the external implication from primitive modular coincidence to model sets; that classification is a cited T3 import, not part of this theorem. & \lean{native\_decide} \\
\lean{seed\_language\_closure} & Closing the legal \texttt{2$\times$2$\times$2} block language under all 27 windows of a substituted block has successive sizes 168, 600, 1,278, 1,398 and 1,410, after which the language is stable. Exactly 27 blocks on \texttt{\{-1,0\}$^3$} are fixed by their eight residue maps. They form two proper-frame orbits of sizes 24 and 3; 24 seeds have trivial origin stabilizer and three have stabilizer order 8. Every seed reproduces itself in place, and the cell substitution commutes with all 24 generated proper frames. & \lean{native\_decide} \\
\lean{mesh\_angle\_audit} & Every actual mesh edge has one of the declared flat/\allowbreak{}carrier/\allowbreak{}feature dihedrals. The 1,536 feature edges contain 32 copies of each (base/\allowbreak{}ridge, magnitude 1..12, positive/\allowbreak{}negative side) class with the exact squared-angle and edge-length formula. & \lean{native\_decide} \\
\end{longtable}}

%% file: generated/hypotheses_table.tex
% generated by paper/scripts/hypotheses_table.py; from lean/R44/HYPOTHESES.md 'Discharged', PAPER_UPDATES.md section 2 and build_axioms.log at d90313a717; 19 former fields; 0 remaining
{\small
\begin{longtable}{L{0.21\linewidth} L{0.09\linewidth} L{0.22\linewidth} L{0.05\linewidth} L{0.31\linewidth}}
\caption{The written lemmas of Sections 2--5 that formed the \lean{Hypotheses} record, each now a Lean theorem of the pinned development (the record is empty, Appendix C): the former field, the written lemma it quoted, the theorem, its tier by its own axiom line, and the argument of the formal proof where it differs from, or refines, the written one (FIGURES.md, Table 3).}\label{tab:hypotheses}\\
\toprule
former field & written lemma & Lean theorem & tier & how the formal proof argues \\
\midrule
\endfirsthead
\toprule
former field & written lemma & Lean theorem & tier & how the formal proof argues \\
\midrule
\endhead
\bottomrule
\endfoot
\lean{retained\_core\_overlap} & A-L5.1 & \lean{retained\_core\_overlap\_holds} & T1 & density choice of overlapping cube interiors; eighth-grid midpoint \\
\lean{carrier\_feature\_frame\_reduction} & R\S{}6, E5 & \lean{carrier\_feature\_frame\_reduction\_holds} & T1n & six carrier symmetries from diameter pairs and the missing corner; signed heights from membership thresholds along normals \\
\lean{circular\_cone\_solid\_angle} & A-L4.1 & \lean{circular\_cone\_solid\_angle\_holds} & T1 & horizontal disk slices; volume (2$\pi$/\allowbreak{}3)(1 - L/\allowbreak{}$\surd$(1+L$^2$)), L = 3/\allowbreak{}25 \\
\lean{per\_tube\_homeomorphisms} & A\S{}1 & \lean{per\_tube\_homeomorphisms\_holds} & T1n & inverse written with the continuous clamp $\psi$; seams and joint continuity in the height proved \\
\lean{feature\_tube\_maps\_glue} & A\S{}1 & \lean{feature\_tube\_maps\_glue\_holds} & T1n & gluing by finite composition of supported homeomorphisms \\
\lean{feature\_tube\_map\_carries\_carrier} & A\S{}1 & \lean{feature\_tube\_map\_carries\_carrier\_holds} & T1n & F(x) $\in$ Q $\Leftrightarrow$ x $\in$ P proved pointwise, not inferred from retained cores \\
\lean{feature\_circular\_cone\_containment} & A-L4.1 & \lean{feature\_circular\_cone\_containment\_holds} & T1n & agreement with the signed-tent hypograph on an open ambient neighbourhood of every feature-graph point (base edges included); Lipschitz cone inclusion transported by an affine isometry \\
\lean{planar\_area\_carrier\_recovery} & R\S{}6, E5 & \lean{planar\_area\_carrier\_recovery\_holds} & T1 & alternative proof of the unchanged proposition: diameter-endpoint argument (squared distance bound 1089/\allowbreak{}100 < 12; the six antipodal carrier corners recover centre and axes; the missing corner excludes reversal). Not the written planar-area calculation; say so where the paper cites R\S{}6 \\
\lean{cone\_sector\_budgets} & A-L2.2 & \lean{cone\_sector\_budgets\_holds} & T1n & Q as a finite Boolean combination of affine halfspaces; uniform conical germs; equality with the radial tangent cone; coverage includes cone-boundary directions, disjointness is of interiors \\
\lean{connected\_feature\_companion} & A-T4.2 & \lean{connected\_feature\_companion\_holds} & T1n & finite disjoint closed cover on connected sets; triple contact excluded via the discharged open-cone inclusion and volume formula (no assumption that a cone and its interior have equal volume) \\
\lean{feature\_containment\_rigidity} & A-L4.3, E1 & \lean{feature\_containment\_rigidity\_holds} & T1n & compactness argument (K compact, f an isometry, f(K) $\subseteq$ K $\Rightarrow$ f(K) = K) in place of the written graph-length argument; uses the distinct-tile premise \\
\lean{companion\_pose\_discrete} & A-C4.4 & \lean{companion\_pose\_discrete\_holds} & T1n & recovery of centre, normal, cubic frame, eighth-grid translation; census equality via list semantics; one side condition on fixed census metadata is the named T1n theorem \lean{mate\_records\_roles\_nonempty} (ruling R8) \\
\lean{only\_registered\_mates} & A-L5.3 & \lean{only\_registered\_mates\_holds} & T1 & endpoint keeps its frozen companion premises, including distinct placements; accepted companion/\allowbreak{}rigidity/\allowbreak{}census/\allowbreak{}overlap helpers assemble the literal 44-mate conclusion \\
\lean{component\_solids\_cover} & A-L6.2 & \lean{component\_solids\_cover\_holds} & T1 & clamp to an interior point with dist$^2$ $\le$ 3/\allowbreak{}2500 < 1/\allowbreak{}100, use the carried 1/\allowbreak{}4-ball and packing disjointness; alternative to assembly-wide tube gluing and conditional on an existing tiling, not an existence proof \\
\lean{small\_collar\_realization} & R\S{}8--9 & \lean{small\_collar\_realization\_holds} & T1n & duplicate tube sites are identified by centre; conjugated maps agree; membership equivalence holds for every tile including unrelated ones; separation is 21/\allowbreak{}200 > 1/\allowbreak{}16 \\
\lean{complete\_dihedral\_list} & A-L3.1 & \lean{complete\_dihedral\_list\_holds} & T1n & native wedge charts and all 6,408 mesh-edge incidences; the finite incidence trace is the named R8 theorem \lean{native\_mesh\_incidence\_trace} \\
\lean{generic\_feature\_partner} & A-L3.2 & \lean{generic\_feature\_partner\_holds} & T1n & native boundary strata and ordinary-meridian classification, with the finite exceptional-vertex tables certified by named native hooks \\
\lean{baseline\_component\_covers\_grid} & A-L6.1 & \lean{baseline\_component\_covers\_grid\_holds} & T1n & concrete companion assembly from the discharged local geometry and exact atlas \\
\lean{unrestricted\_alignment} & A-T6.3 & \lean{unrestricted\_alignment\_holds} & T1n & component coverage and the certified proper-frame group assemble global registration \\
\end{longtable}}

%% file: generated/axioms_table.tex
% generated by paper/scripts/axioms_table.py; from lean/R44/build_axioms.log at d90313a717; 169 declarations with an axiom line
{\small
\begin{longtable}{L{0.3\linewidth} L{0.58\linewidth}}
\caption{Axiom lines: the output of \lean{\#print axioms} for the finite theorems and the main theorems; ``standard'' is \lean{propext}, \lean{Classical.choice}, \lean{Quot.sound}, and hooks are the per-theorem \lean{native\_decide} axioms (FIGURES.md, Table 4).}\label{tab:axioms}\\
\toprule
declaration & axioms (\texttt{\#print axioms}) \\
\midrule
\endfirsthead
\toprule
declaration & axioms (\texttt{\#print axioms}) \\
\midrule
\endhead
\bottomrule
\endfoot
\lean{children\_partition\_2P} & none \\
\lean{contact\_closure\_30} & hooks: contact\_closure\_30 \\
\lean{profile\_canonical} & propext; hooks: profile\_canonical \\
\lean{atlas\_44} & standard; hooks: atlas\_44 \\
\lean{first\_shells\_33} & propext; hooks: first\_shells\_33 \\
\lean{central\_completion} & propext; hooks: central\_completion \\
\lean{parent\_conflicts\_28} & hooks: parent\_conflicts\_28 \\
\lean{parent\_atlas\_eq\_fine} & standard; hooks: parent\_atlas\_eq\_fine \\
\lean{mates\_census} & standard; hooks: mates\_census \\
\lean{deviations\_distinct} & none \\
\lean{no\_native\_symmetry} & propext \\
\lean{native\_features\_length} & propext \\
\lean{native\_feature\_centers\_nodup} & propext \\
\lean{native\_feature\_normal\_axis} & propext \\
\lean{solid\_mesh\_exact} & standard; hooks: solid\_mesh\_exact \\
\lean{boundary\_sphere} & standard; hooks: boundary\_sphere \\
\lean{nested\_substitution\_controls} & Quot.sound, propext; hooks: nested\_substitution\_controls \\
\lean{hierarchy\_cell\_controls} & hooks: hierarchy\_cell\_controls \\
\lean{registered\_shell\_cell\_controls} & propext; hooks: registered\_shell\_cell\_controls \\
\lean{orientation\_group\_24} & Quot.sound, propext; hooks: orientation\_group\_24 \\
\lean{substitution\_modular\_coincidence} & propext; hooks: substitution\_modular\_coincidence \\
\lean{seed\_language\_closure} & Quot.sound, propext; hooks: seed\_language\_closure \\
\lean{mesh\_angle\_audit} & standard; hooks: mesh\_angle\_audit \\
\lean{r44\_einstein} & standard; 21 hooks (DischargeCarrierStrata.carrier\_coordinate\_states\_table, DischargeMeshTrace.native\_mesh\_incidence\_trace, \dots) \\
\lean{tiling\_homochiral} & standard; 13 hooks (DischargeCarrierStrata.carrier\_coordinate\_states\_table, DischargeMeshTrace.native\_mesh\_incidence\_trace, \dots) \\
\lean{tiling\_chirality\_corollary} & standard; 13 hooks (DischargeCarrierStrata.carrier\_coordinate\_states\_table, DischargeMeshTrace.native\_mesh\_incidence\_trace, \dots) \\
\lean{carrier\_hierarchy} & standard; 17 hooks (DischargeCarrierStrata.carrier\_coordinate\_states\_table, DischargeMeshTrace.native\_mesh\_incidence\_trace, \dots) \\
\lean{carrier\_hierarchy\_unique} & standard; 17 hooks (DischargeCarrierStrata.carrier\_coordinate\_states\_table, DischargeMeshTrace.native\_mesh\_incidence\_trace, \dots) \\
\lean{substitution\_modular\_coincidence} & propext; hooks: substitution\_modular\_coincidence \\
\end{longtable}}

%% file: generated/fig_coincidence.tex
% generated by paper/scripts/coincidence_tree.py from verify/substitution_modular_coincidence.py
\begin{tikzpicture}[x=1cm,y=1cm,font=\small,
  lvl/.style={draw,rounded corners=1pt,inner sep=3pt,align=center,minimum width=2.6cm},
  arr/.style={-{Latex[length=1.6mm]},thick}]
  \node[lvl] (d0) at (0,0) {root: all $168$ labels};
  \node[lvl] (d1) at (0,-1.5) {residue $(0,0,0)$ at depth 1\\ \scriptsize all-address histogram: 8\,addresses with 42};
  \draw[arr] (d0) -- (d1);
  \node[lvl] (d2) at (0,-3.0) {residue $(0,0,1)$ at depth 2\\ \scriptsize all-address histogram: 48\,addresses with 6, 16\,addresses with 24};
  \draw[arr] (d1) -- (d2);
  \node[lvl] (d3) at (0,-4.5) {residue $(0,0,0)$ at depth 3\\ \scriptsize all-address histogram: 336\,addresses with 1, 160\,addresses with 6, 16\,addresses with 24};
  \draw[arr] (d2) -- (d3);
  \node[lvl,thick,fill=black!6] (fin) at (0,-6.0) {address $a=(0,0,2)$ mod $2^{3}$: one label, $i=78$};
  \draw[arr] (d3) -- (fin);
\end{tikzpicture}

%% file: sec9_context.tex
\section{Finite admissibility and contextual distinctions}\label{sec:context}

% ledger: C1--C6 (2026-09-13 supplement); these are written deductions, not new Lean claims.
The results above give a concrete instance of a distinction in the Six Birds (SBT)
calculus \cite{Tsi26F4}: a rule can survive composition while its realizations require
successively finer contextual distinctions. This complementary reading proceeds from
a local obstruction to the intrinsic hierarchy, then asks what the construction
suggests about descriptions of collective systems. The deductions below are written
proofs; no new Lean verification is claimed.

\emph{A distinction that cannot be erased.}
Let $C(x,y,z)=(y,z,x)$. Since $C\carrier=\carrier$, the nine pairs
$(C^a\solid,C^b\solid+(1,1,1))$, for $a,b\in\{0,1,2\}$, have the same ordered
pair of bare carriers. Their feature profiles match precisely when $a=b$.
The exact checker \file{paper/scripts/contextual_frames.py} reconstructs the $24$
opposed feature centres in each pair from the solid's rational feature records:
every off-diagonal pair mismatches all $24$ coefficients and has an interior overlap.
For example, $\solid$ and $C\solid+(1,1,1)$ both contain
$(5/4,11/8,1)$ in their interiors: there the root's role $188$ protrudes in $+z$
by $12/10000$, and the neighbour's role $4$ protrudes in $-z$ by $5/10000$.
The diagonal pairs are rotated copies of the $000$/central subdivision pair.
In SBT terminology this is a \emph{descent obstruction}: admissibility is not a
function of the erased carrier description. Testing the neighbour against each
of the three root frames requires retaining all three neighbour-frame distinctions.

\emph{Two different closure questions.}
SBT's \emph{sufficiency closure} (F7 of \cite{Tsi26F4}) asks whether each declared
readout factors through an observation. It does not assert that an operation
preserves legal configurations. Here the latter is the separate geometric statement
$D(X_Q)\subseteq X_Q$, where $X_Q$ is the space of tilings by $\solid$
(\cref{thm:coarsening}). This permits repeated recognition and restandardization
under the same rule. The observations of position in that hierarchy behave differently.

\begin{proposition}[phase distinctions]\label{prop:phases}
Fix a registered tiling $T$ in integer coordinates. Its level-$n$ parent anchors
lie in a unique coset $\alpha_n(T)+2^n\Z^3$, with
$\alpha_{n+1}(T)\equiv\alpha_n(T)\pmod{2^n}$. The observations
\[
 q_n^T(v)=v-\alpha_n(T)\pmod{2^n},\qquad v\in\Z^3,
\]
have $8^n$ values. Writing $\ker q=\{(v,w):q(v)=q(w)\}$,
\[
 \ker q_n^T=\{(v,w):w-v\in2^n\Z^3\},\qquad
 \ker q_{n+1}^T\subsetneq\ker q_n^T,\qquad
 \bigcap_{n\ge0}\ker q_n^T=\Delta,
\]
where $\Delta=\{(v,v):v\in\Z^3\}$.
\end{proposition}

\begin{proof}
The common-parent parity argument of \cref{thm:coarsening}, applied at every
level and expressed in the original coordinates, gives the cosets. Each parent
has child $000$ at its own anchor, so the anchor sets are nested and the cosets
are coherent. Reduction modulo $2^n$ is onto and has the displayed kernel.
The pair $v,v+2^ne_1$ lies in that kernel but not the next; their intersection
is $\Delta$ because $\bigcap_n2^n\Z^3=\{0\}$.
\end{proof}

The phases are read from the intrinsic hierarchy, locally at each fixed level
with a neighbourhood growing with the level. In particular, for $p\in\Z^3$,
$\alpha_n(T+p)=\alpha_n(T)+p\pmod{2^n}$. A period must therefore vanish modulo
every $2^n$, recovering \cref{thm:noperiod}. This uses the global partitions,
not exhaustion by one tile's ancestors. Equal phases need not give equal patches:
$8^n$ counts phase states, not patch types or entropy. Each next level requires
eight times as many phase states, while each fixed phase observation has the
lawful translation update $q_n^T(v+u)=q_n^T(v)+u$.

\emph{Exact contextual sufficiency.}
Let $\ell:\Z^3\to A$ be the $168$-label cell encoding of $T$
(\cref{sec:limitperiodic}). It is faithful: a cell's position, native frame
and native cell identify its owner's whole pose. Declare the readouts to be
all translated labels, $v\mapsto\ell(v+u)$ for $u\in\Z^3$. Their common
equivalence, the \emph{predictive quotient} for this spatial probe, satisfies
\[
 v\equiv w\ \Longleftrightarrow\
 \forall u\in\Z^3,\ \ell(v+u)=\ell(w+u)
 \ \Longleftrightarrow\ w-v\in\Per(\ell).
\]
Indeed, put $u=z-v$; the reverse implication is the definition of a period.
Faithfulness gives $\Per(\ell)=\Per(T)=\{0\}$. Thus any observation sufficient
for all these readouts must distinguish every probe origin. This consequence
holds for any faithful aperiodic lattice labeling; Chair44 additionally supplies
the intrinsic phase hierarchy organizing such distinctions.

This sufficiency requirement concerns the declared probes. It excludes neither
finite descriptions nor useful coarse models. The coherent phases give a $2$-adic
coordinate, but need not determine the whole tiling; the hull and diffraction
questions retain the scope stated in \cref{sec:limitperiodic}.

\emph{Collective descriptions and their limits.}
These results suggest a broader question: which relationships must a description
retain for a collective to support its declared operations? An inventory of parts
can identify arrangements that behave differently when placed in a further context.
The appropriate criterion is therefore relative to a specified family of contexts
$\mathcal C$ and their observable outcomes:
\[
 q(x)=q(y)\quad\Longrightarrow\quad
 \operatorname{Outcome}(C[x])=\operatorname{Outcome}(C[y])
 \quad\text{for every }C\in\mathcal C.
\]
Here the outcome includes whether the composition is admissible at all. The
coarsest equivalence satisfying this criterion identifies exactly those states
that all declared contexts fail to distinguish. It specifies the relational
information required by those tests, rather than a complete microscopic inventory.

In SBT terms, a separating context witnesses a descent obstruction. Enlarging the
context family can only refine its observational equivalence, but strict refinement
at arbitrarily large depths requires a separate argument. Recursive composition
alone does not supply one. For Chair44, \cref{prop:phases} supplies that argument:
each phase level identifies positions separated by some nonzero translations,
while no such identification survives every level. Independently, coarsening
preserves the geometric admissibility law that makes the hierarchy repeatable.

The philosophical consequence is precise and limited: a finite, recursively stable
law can require a realization to preserve relational distinctions at arbitrarily
large scales. The infinite arrangement carries those distinctions; the law
constrains their organization. This suggests a question for other collective
systems: which permitted interaction reveals information erased by a proposed
description, and does retaining that information suffice after further composition?
An answer requires identifying that system's own states, operations and separating
contexts. Chair44 provides a geometric instance in which both lawful recursion and
unbounded contextual differentiation can be established.

%% file: sec9_remarks.tex
\section{Remarks and open problems}\label{sec:remarks}

% ledger: P1-P8, H12, N3, N4, N5, N6, Q5; source COMMUNITY_PAPER_PROFILE.md §10; 2303.10798 L253, L262; 2409.15880 L49, L236; 2305.17743 L115, L118
We close with questions, in the hat paper's manner, and we phrase each as a question.

\begin{enumerate}[label=\textup{(\arabic*)}, leftmargin=*]
  \item \emph{Symmetry.} The fixed-point construction (\cref{prop:seedsym}) yields tilings with
    symmetry groups of orders $1$ and $8$. Is the bound $|\Sym(T)|\le24$ attained, and which
    other finite symmetry groups occur? Question~32 of Coulbois, Gajardo, Guillon and Lutfalla
    asks for a geometric monotile whose tilings all have trivial stabilizer \cite{CGGL24};
    $\solid$ is not one (\cref{sec:aperiodicity}).
  \item \emph{Fault surfaces and the tiling space.} The ancestors of one fixed tile can fail to
    cover space (\cref{sec:hierarchy}). Which non-exhausting hierarchies occur, and does every
    tiling lie in the substitution hull? The hat paper asks the latter of hats \cite{SMKGS24}; for
    $\solid$ the question is also where a hull-level version of \cref{cor:limitperiodic}
    (\cref{sec:limitperiodic}) would
    have to be settled.
  \item \emph{Convexity.} Is there a convex strongly aperiodic monotile in $\R^3$? With
    reflections forbidden, the biprism is convex and weakly but not strongly aperiodic; with
    reflections allowed it admits periodic tilings; $\solid$ is strongly aperiodic and not convex;
    and for translations alone convex tiles cannot be aperiodic \cite{CGGL24}.
  \item \emph{Achirality.} Is there a strongly aperiodic monotile of $\R^3$ with a mirror symmetry,
    or one whose tilings mix both handednesses? The Spectre paper poses the einstein problem for
    each isometry group $G$ \cite{SMKGS24b}; $\solid$ answers it for the full isometry group of
    $\R^3$ and, being strictly chiral, for the group of proper motions at the same time.
  \item \emph{Simplicity.} What is the smallest number of faces, or of vertices, of a strongly
    aperiodic monotile in $\R^3$? $\solid$ has $4{,}272$ triangles because each of its $192$ features
    is a pyramid; within the twelve-parameter family of \cref{app:family} the proof does not
    change, and we do not know how far the architecture can be simplified.
  \item \emph{Heesch and isohedral numbers; decidability.} What Heesch numbers and isohedral
    numbers occur for polyhedra in $\R^3$, and is it decidable whether a given polyhedron tiles
    $\R^3$? We prove nothing in this direction.
  \item \emph{A shorter proof.} Simplified proofs of the hat's aperiodicity appeared within
    months \cite{AA25}. The universality half of our proof, \cref{sec:companions,sec:registration},
    is the part we would most like to see shortened; the finite half is a census that a different
    argument might avoid.
  \item \emph{The mechanism elsewhere.} The construction puts small features on a periodic
    carrier so that the features force a registered lattice substitution whose parent language is
    its own fine language. Does the same design yield einsteins in other dimensions, on other
    lattices, or in other crystallographic settings, and is there a version in the plane with a
    different carrier?
\end{enumerate}

% ledger: O1, N4; source 2303.10798 L262 (deflationary close)
The features of $\solid$ have height unit $1/10000$ and maximum height $12/10000$, measured in
carrier-cell units. Everything the theorem says about tilings by $\solid$, that they exist, that
they are hierarchical, that they have no period and no symmetry of infinite order, and that they
never mix a tile with its mirror image, is decided by those pyramids. We do not know whether
fewer or larger features, or a fundamentally simpler solid, can enforce the same conclusions;
$\solid$ may be far from the simplest example.

%% file: appA_data.tex
\section{Data and replay}\label{app:data}

% ledger: M3, M4, M6, M7, Q22; source APPENDIX_DATA.md §1-§4, §6; verify/replay.py; viewer/README.md
\emph{Inputs.} The theorem is about one solid, \file{solid/r44_solid.json}: its $2{,}138$
vertices as exact rationals, its $4{,}272$ triangles, the $192$ features with their signed
coefficients, and the eight child poses. Two certificate files carry the finite data the proof
uses: \file{certificates/candidate_certificate.json} (the contact closure and profile, the
$44$-contact atlas, the $33$ first shells, the completions and conflicts, the parent census) and
\file{certificates/companion_collision_certificate.json} (the $5{,}273$ rejection records with
their $299{,}975$ collision witnesses, the boxes themselves in a companion file). Their sha256
digests are in \cref{tab:hashes}. The replay entrypoint checks that the four canonical inputs are
byte-identical to the packet copies and checks the fixed solid digest before invoking the
checkers.

\begin{table}[ht]
  \centering
  \caption{Input hashes and tool versions, computed at the state of the repository described in
  this paper (FIGURES.md, Table 6).}
  \label{tab:hashes}
  \small
  \input{generated/appendix_data_table}
\end{table}

% ledger: M3, M4, M5, M6; source APPENDIX_DATA.md §2-§4; paper/review/REPLAY_FROM_PAPER.md
\emph{Replay in one minute.} From the repository root,
\begin{verbatim}
python3 verify/replay.py
\end{verbatim}
runs, with the Python standard library only and no network, the checker of the registered
theorem (the first implementation), the alignment checker and the integer-only replay of the
collision boxes (the second implementation), the mesh check \file{verify/boundary_sphere.py},
the tube-formula controls, \file{verify/substitution_modular_coincidence.py} and the
block-language closure \file{verify/seed_language_closure.py}; it asserts the
input digests first and writes \file{verify/replay_report.json}. In the recorded clean-clone run
at the pinned commit (\file{paper/review/REPLAY_FROM_PAPER.md}, 2026-09-10) the replay
completed in about half a minute ($28.316$
seconds wall time) on the machine that ran the Lean builds; its hardware is not otherwise
specified; hardware specifications and Lean build time and peak memory are not reported for this
version. Three
further commands complete the README's one-minute list:
\begin{verbatim}
(cd lean/R44 && python3 gen/extract.py --check)
python3 -m unittest discover simulations/tests
cd verify/packets/r44_unrestricted_alignment
python3 src/test_mutations.py
\end{verbatim}
The first checks that the Lean literals generated from the JSON inputs match those inputs. The
unit-test command reported $10$ tests with six skipped by optional-dependency guards; it does not
establish that those six tests passed. The last command rejects six corrupted inputs (a missing triangle, a reversed triangle, a moved apex, a
missing rejection witness, a wrong companion quantifier, a missing registered contact) and states
its own scope: it checks that malformed inputs are rejected and does not validate the continuous
proof.

\emph{Paper-lane scripts.} The standard-library scripts cited in the text as \tier{T2} live in
\file{paper/scripts/}: the exact closure of the block language and the fixed seeds
(\file{seed_language_closure.py}, \cref{sec:limitperiodic}), the planar areas of \cref{app:asymmetry}
(\file{planar_areas.py}) and the non-convexity witness; each reads the repository data and
prints \texttt{STATUS: PASS}.

\emph{Figures.} \Cref{fig:tile,fig:panel} are drawn by \file{figures/print-preview/build\_preview.py}
from \file{solid/r44\_solid.json} with the exaggerations stated in their captions: the carrier and
the $192$ feature centres keep their coordinates, every feature base is widened $\times 5$ about
its centre and every signed height is multiplied by $80$, so polarities and relative heights
are preserved and nothing is added, omitted or moved. The script checks all $192$ feature records
against the mesh and the native panel table in rational arithmetic, and a second script reads the
display mesh back and reconstructs its $768$ feature triangles independently; the outputs and their
digests are in \file{figures/print-preview/}. \Cref{fig:cluster,fig:patch} are the renders of
\file{figures/render\_r44.py}. No figure is verification.

\emph{Exact arithmetic.} Every certificate is integer or rational data; the only decimal literals
in the certificates are two timing fields. The checkers use Python integers and
\texttt{fractions.Fraction}; the only floating-point conversion in the repository's verification
code is the export of the mesh to a rendering format. The verification scripts named in the ``Paper-lane scripts'' paragraph
also use exact rationals only; the planar-area checker bounds each non-axis facet's
area $\sqrt q/2$ by $(q+m^2)/(4m)$ with an integer-square-root witness $m$, exactly. Collision
boxes are stored in integer units of $1/400$. The Lean theorems decide integer and rational
literals.

\emph{Lean.} The development is \file{lean/R44}, built with \texttt{lake build} under the
toolchain of \cref{tab:hashes} (the build is capped at four parallel jobs and the Mathlib cache is
pre-seeded); \file{lean/R44/scripts/controls.sh} runs the positive build, then every must-fail
file against its expected diagnostic, the scope regressions and the admission count, in one
command. The axiom lines of \cref{tab:axioms} are produced by the build itself
(\file{build_axioms.log}).

% ledger: Q22, M7, O12; source viewer/README.md; figures/print-preview/README.md
\emph{Viewer.} \file{viewer/index.html} shows the solid, the eight-child cluster and the
$64$-, $512$- and $4{,}096$-copy patches in the browser without a server. It
is generated from the canonical rational data with the source hashes recorded; its features are
drawn exaggerated by default, with a switch to true heights; colours are annotation. It also
offers a period test on finite patches, which, as its own caption says, proves nothing.

\emph{Code.} Repository: \url{https://github.com/ioannist/six-birds-tiles}. The Lean statements,
hypothesis record and axiom lists reproduced in this paper refer to commit
\texttt{d90313a717}, and the recorded clean-clone replay was run at that commit (an earlier
recorded run at \texttt{bdc6d3b957}, before the block-language closure joined the replay, gave
the same digests and outcomes). The intended paper-release tag is \texttt{paper-v1};
tag creation and licensing remain pending owner actions.

%% file: generated/appendix_data_table.tex
% generated by paper/scripts/appendix_data_table.py; sha256 computed live; Lean versions at d90313a717
\begin{tabular}{L{0.42\linewidth} L{0.5\linewidth}}
\toprule
artifact & sha256 / version \\
\midrule
\texttt{solid/\allowbreak{}\allowbreak{}r44\_\allowbreak{}solid.json} & \texttt{f320d7a0c2d784a3\dots{}42f0ed55} \\
\texttt{solid/\allowbreak{}\allowbreak{}native\_\allowbreak{}panels.csv} & \texttt{b9eec9fd834abba4\dots{}d3a9ed2d} \\
\texttt{certificates/\allowbreak{}\allowbreak{}candidate\_\allowbreak{}certificate.json} & \texttt{44e9b3f6048497de\dots{}c5de5ff2} \\
\texttt{certificates/\allowbreak{}\allowbreak{}companion\_\allowbreak{}collision\_\allowbreak{}certificate.json} & \texttt{36e89e2fe81f4788\dots{}1b068b08} \\
\texttt{certificates/\allowbreak{}\allowbreak{}collision\_\allowbreak{}core\_\allowbreak{}witnesses.jsonl.gz} & \texttt{ae64a5283400c61e\dots{}161f8a45} \\
Lean toolchain & \texttt{leanprover/\allowbreak{}lean4:v4.31.0} \\
Mathlib commit & \texttt{fabf563a7c95} \\
batteries / aesop / Qq / proofwidgets & \texttt{fa08db58b30e, e3cb2f741431, f46324995fca, 24b0d9dc081c} \\
\texttt{verify/\allowbreak{}\allowbreak{}replay.py} (Python standard library) & PASS, 27.0 s \\
Python (replay and tables) & 3.12.3 \\
\bottomrule
\end{tabular}

%% file: appB_census.tex
\section{Census tables}\label{app:census}

% ledger: R1-R9, A11, L1, M3; generated census_table; Lean names per row
\Cref{tab:census} collects the finite censuses used in the proof, read by a script from the
certificates and the packet results that the replay checks. For each row the Lean theorem that
certifies it is: the contact closure and profile, \lean{contact\_closure\_30} and
\lean{profile\_canonical}; the atlas and the group it generates, \lean{atlas\_44} and
\lean{orientation\_group\_24}; the first shells, completions and conflicts,
\lean{first\_shells\_33}, \lean{central\_completion} and \lean{parent\_conflicts\_28}; the parent
census, \lean{parent\_atlas\_eq\_fine}; the companion census and its collisions,
\lean{mates\_census}; the substitution, \lean{substitution\_modular\_coincidence}. The companion
census is computed by the two independent implementations of \cref{sec:computer}; the second
implementation records, for its own audit, that it imports no upstream program, derives the
features from the actual triangles, uses no flat-panel filter to reject placements, and uses no
existing rejection to prune the companion options.

\begin{table}[ht]
  \centering
  \caption{The censuses (FIGURES.md, Table 2).}
  \label{tab:census}
  \small
  \input{generated/census_table}
\end{table}

% ledger: R2, A11; source PROOFS_registered.md R§6; alignment_verification.json (complete_feature_mates)
The $2{,}388$ candidate contact poses arise as follows. A root chair in the identity pose has $22$
grid cells adjacent to it across its $24$ panels (the $21$ outer panels face distinct exterior
cells, while all three notch panels face the missing cell $(1,1,1)$, giving $22$ distinct shell
cells); a neighbouring chair in one of
the $48$ signed frames covering a shell cell is positioned by subtracting each of its seven cells
from that shell cell; deduplicating and rejecting overlaps with the root leaves $2{,}388$. The
$6{,}862$ companion poses are the formula poses of \cref{cor:discrete} over all ordered pairs of
opposite features and all signed permutations with matching normals; of them $1{,}545$ collide with
the root directly, $5{,}317$ are isolated ($5{,}234$ with a fractional translation, $83$ integral),
$5{,}273$ are rejected by the collision witnesses, and the $44$ survivors equal the atlas as a
set.

%% file: generated/census_table.tex
% generated by paper/scripts/census_table.py; from certificates/*.json, the alignment packet results and substitution_modular_coincidence.py
\begin{tabular}{L{0.6\linewidth} l}
\toprule
census & count \\
\midrule
contact closure: internal face contacts $\to$ closed states; sign equations; balanced components & $21\to30$; $372$; $12$ \\
shell poses in $48$ frames $\to$ legal contacts (the atlas); distinct rotations; group generated & $2{,}388\to44$; $19$; $24$ \\
first shells; outer-root completions found / impossible; parent conflicts & $33$; $14/18$; $28$ \\
parent contacts: candidates $\to$ disjoint $\to$ legal; halved $=$ fine atlas & $697\to116\to44$; yes \\
complete-feature mate poses $\to$ isolated $\to$ survivors & $6{,}862\to5{,}317\to44$ \\
rejected isolated poses; companion-option collisions checked & $5{,}273$; $299{,}975$ \\
substitution labels; column sum; least primitive exponent $N$; coincidence depth $M$, address $a$, label $i$ & $168$; $8$; $3$; $3$, $(0,0,2)$, $78$ \\
\bottomrule
\end{tabular}

%% file: appC_lean.tex
\section{Lean statement listing and axioms}\label{app:lean}

% ledger: T1-T3, T12, M1, M2, M2b (closed), M5; source lean/R44/R44/LogicalSpine.lean and AXIOMS.md at the pinned commit
\emph{The main theorem.} In \file{lean/R44/R44/LogicalSpine.lean}, with $Q$ the solid defined
from the canonical panel data, \lean{Tiling Q} the tilings of $\R^3$ by affine-isometric copies of
$Q$ (reflections allowed, local finiteness not assumed), and \lean{Per}, \lean{Sym} as in
\cref{sec:terminology}, the main theorem reads
\begin{footnotesize}
\begin{verbatim}
theorem r44_einstein : Nonempty (Tiling Q) ∧ ∀ T : Tiling Q,
    Per T = {0} ∧ (Set.univ : Set (Sym T)).encard ≤ 24
\end{verbatim}
\end{footnotesize}
and takes no hypothesis argument. The companion declarations of \cref{sec:lean} (homochirality,
the chirality corollary, the hierarchy and its uniqueness) retain an argument \lean{H} of type
\lean{Hypotheses}, but that record is empty and is instantiated by the empty anonymous
constructor; it imposes no mathematical assumption. The finite theorems of \cref{tab:finite} take no hypothesis
argument.

% ledger: T12, M5; source lean/R44/HYPOTHESES.md, LogicalSpine.lean (structure Hypotheses) at the pin
\emph{The record.} The development was written around a record \lean{Hypotheses} whose fields
quoted the written lemmas of this paper, one field per lemma, with \lean{r44\_einstein} proved
over that record. Each field has since been proved as a theorem of the development
(\cref{tab:hypotheses}; the theorems live in \file{lean/R44/R44/Proved/}), and at the pinned
commit the record has no fields: its declaration is, verbatim,
\begin{footnotesize}
\input{generated/hypotheses_record}
\end{footnotesize}
The conditional theorem survives, over the empty record, as
\lean{r44\_einstein\_of\_hypotheses}; the main theorem is its instance at the empty record.
\file{lean/R44/HYPOTHESES.md} gives, for each former field, the written lemma it quoted, the
theorem that discharged it and the delivery it came in. The diagonal audit, made after the first
external review, distinguishes assertions about distinct companion tiles from relations with
legitimate reflexive cases.

% ledger: M2, M2b (closed), M5; source lean/R44/build_axioms.log and AXIOMS.md at the pin
\emph{Axioms.} \Cref{tab:axioms} is the output of \lean{\#print axioms}; the main theorem's line
is, verbatim,
\begin{scriptsize}
\begin{verbatim}
'R44.r44_einstein' depends on axioms: [propext, Classical.choice, Quot.sound,
  R44.atlas_44._native.native_decide.ax_1_1,
  R44.central_completion._native.native_decide.ax_1_1,
  R44.contact_closure_30._native.native_decide.ax_1_1,
  R44.first_shells_33._native.native_decide.ax_1_1,
  R44.hierarchy_cell_controls._native.native_decide.ax_1_1,
  R44.mate_records_roles_nonempty._native.native_decide.ax_1_1,
  R44.mates_census._native.native_decide.ax_1_1,
  R44.mesh_angle_audit._native.native_decide.ax_1_1,
  R44.nested_substitution_controls._native.native_decide.ax_1_1,
  R44.orientation_group_24._native.native_decide.ax_1_1,
  R44.parent_atlas_eq_fine._native.native_decide.ax_1_1,
  R44.parent_conflicts_28._native.native_decide.ax_1_1,
  R44.profile_canonical._native.native_decide.ax_1_1,
  R44.registered_shell_cell_controls._native.native_decide.ax_1_1,
  R44.solid_mesh_exact._native.native_decide.ax_1_1,
  R44.transported_role_geometry._native.native_decide.ax_1_1,
  R44.DischargeCarrierStrata.carrier_coordinate_states_table
      ._native.native_decide.ax_1_1,
  R44.DischargeMeshTrace.native_mesh_incidence_trace
      ._native.native_decide.ax_1_1,
  R44.DischargeVertexData.carrier_vertex_lookup._native.native_decide.ax_1_1,
  R44.DischargeVertexData.feature_index_bounds._native.native_decide.ax_1_1,
  R44.DischargeVertexData.feature_vertex_lookup._native.native_decide.ax_1_1]
\end{verbatim}
\end{scriptsize}
(one line in the log, broken here, the two longest names at their namespace dot). Each
\lean{native\_decide} hook trusts the compiled evaluation
of one closed decidable proposition. Of the $21$ hooks, $15$ belong to the finite theorems of
\cref{tab:finite} and to the auxiliary \lean{transported\_role\_geometry} check; the other six
are closed checks over the certified finite data made by the discharge proofs: a side condition
on the census metadata (\lean{mate\_records\_roles\_nonempty}), the $6{,}408$-edge mesh
incidence trace (\lean{native\_mesh\_incidence\_trace}), a $125$-state classifier of carrier
coordinates (\lean{carrier\_coordinate\_states\_table}), and three lookups in the tables of
exceptional feature and carrier vertices. None of them changes the proposition it certifies; the
repository's rulings R8 and R10 in \file{lean/R44/MECHANIZATION\_PLAN.md} record why each is a
named hook rather than a kernel \texttt{rfl}. The independent Python replays check the certificates
and geometric data described in \cref{app:data}. No \lean{sorryAx} occurs anywhere in the
development, and the directory \file{lean/R44/R44/Discharge/}, which held the separately built
discharge library with its marked admissions at the earlier baseline, holds no module at the
pinned commit.

\emph{Controls.} \file{lean/R44/scripts/controls.sh} runs the positive build and then compiles
every must-fail file, checking that each fails with its expected diagnostic. The negative
controls exercise the repaired diagonal application, corrupted finite inputs, mixed-handedness
claims, a second-hierarchy claim and a claim that a raw geometric hierarchy has a noncanonical
level-one pose. The
runner also checks the scope regressions (which include legitimate reflexive uses) and reports
the admission count, zero at the pinned commit. A transcript of the controls run, with build time
and peak memory, is not included at the state described here.

%% file: generated/hypotheses_record.tex
% generated by paper/scripts/hypotheses_record.py from lean/R44/R44/LogicalSpine.lean at d90313a717; docstrings omitted
\begin{verbatim}
structure Hypotheses : Prop where
\end{verbatim}

%% file: appD_calibration.tex
\section{The planar calibration}\label{app:calibration}

% ledger: D8; source calibration_2d/LANDING.md, rounds/r004_varying_geometry_tower/formal/Tower.lean; not evidence for the theorem
The same pipeline was run on the hat as a calibration of the tooling, not as a result: the hat
polykite with its thirteen vertices was reconstructed as an exact polygon, its published
substitution and recognition grammar were replayed with explicit proof dependencies, the growth
of patches was derived, and controls were run. The calibration checked an abstract theorem that
transports a period unchanged through coarsening levels and excludes nonzero periods using
unbounded lower bounds on their size; its geometric and recognition inputs remain explicit
imports from the hat paper \cite{SMKGS24}, to whose authors the hat and its aperiodicity are due.
The R44 proof uses the related obstruction in a different normalization: periods halve while
nonzero registered periods have a uniform positive lower bound. The calibration is recorded in
the repository under \texttt{calibration\_2d/} with its own charter, receipts and lessons, and it
served as the template for the Lean skeleton of \cref{sec:mechanization} and as an exercise of the
replay, control and review discipline of \cref{app:data} on a known einstein.

%% file: appE_geometry.tex
\section{Geometric details of the construction}\label{app:solidchecks}

% ledger: O1, O2, O3, O4, O5, O6, O7, R1, A16, T14; exposition of the existing construction checks relocated from Section 2
The recipe in \cref{sec:solid} specifies the solid. Here we record how its signed
height table is obtained and why the resulting boundary, asymmetry and feature
angles have the properties used by the proof. The numerical choices in the recipe
are one explicit member of the family described at the end of this appendix.

\subsection{How the height table is determined}\label{app:profile}

% ledger: R1; Lean contact_closure_30, profile_canonical
The table of heights is not arbitrary; it is determined by the substitution of
\cref{sec:substitution} and can be reconstructed from it. Start with the $21$ face contacts that
occur inside the eight-child dissection of the doubled chair and propagate them through
refinement until the set of contacts is closed; it closes at $30$. At every pair of marks that
these contacts bring into coincidence impose that the two heights be opposite, $a_u=-a_v$. This
gives $372$ distinct equations whose signed graph on the $192$ marks has exactly twelve balanced
connected components, each with eight positive and eight negative vertices. Assign the magnitude
$j$ to the $j$-th component in order of its smallest mark, take that mark positive, and transport
signs along the graph. The result is the table of $192$ entries in the solid file, and the twelve
magnitudes $1,\dots,12$ are the twelve components. Because every component uses a distinct
magnitude, two features fit face to face only if they have the same magnitude and opposite sign.

\subsection{Boundary and topology}\label{app:boundary}

% ledger: O2, O3, O4; Lean boundary_sphere; verify/boundary_sphere.py; imports Moise, Rourke-Sanderson
The boundary of $\solid$ is a triangulated surface with $2{,}138$ vertices, $6{,}408$ edges and
$4{,}272$ triangles (\cref{tab:mesh}), checked triangle by triangle against the construction in \cref{sec:solid}. Every edge lies in exactly two triangles, the triangle graph is connected, the link of
every vertex is a single cycle, and $V-E+F=2$; these are finite facts, verified in Lean and
independently by a script. Two classical results then apply, and we state them exactly as used:
a connected closed triangulated surface with Euler characteristic $2$ is homeomorphic to the
sphere, since an orientable surface of genus $g$ has $\chi=2-2g$ and a non-orientable one with $k$
crosscaps has $\chi=2-k$; and, by the polyhedral Schoenflies theorem of Alexander, a polyhedral
$2$-sphere in $\R^3$ bounds a piecewise-linear $3$-ball \cite{Moi77,RS72}. Hence $\solid$, the
closure of the bounded domain its boundary mesh bounds, is a closed topological $3$-ball. The same
conclusion follows from the tube construction in \cref{app:tubes}, which exhibits an ambient homeomorphism
carrying $\carrier$ to $\solid$.

\begin{table}[ht]
  \centering
  \caption{Mesh facts of $\solid$, computed from the solid file with exact arithmetic and
  cross-checked against its recorded counts; the verified statements are \lean{boundary\_sphere},
  \lean{mesh\_angle\_audit} and \texttt{verify/boundary\_sphere.py} (FIGURES.md, Table 5).}
  \label{tab:mesh}
  \input{generated/mesh_table}
\end{table}

\subsection{The solid has no nonidentity self-isometry}\label{app:asymmetry}

% ledger: O5, O11, R12; Lean planar_area_carrier_recovery_holds (T1, E5; formal proof by diameter endpoints); carrier_feature_frame_reduction_holds (T1n); no_native_symmetry (T1); paper/scripts/planar_areas.py (T2); external review of the manuscript 2026-09-09, finding m2
$\solid$ has no nonidentity self-isometry. The argument has three parts, and we say once here at what level
each is verified. First, an isometry $g$ with $g\solid=\solid$ preserves the carrier: the nine
large planar boundary regions of $\solid$, one for each coordinate plane at heights $0$, $1$ and
$2$, carry planar areas $2492/625$, $623/625$ and $1869/625$ respectively (the area-four, notch
area-one and area-three regions less their feature bases; exact values from the mesh,
\file{paper/scripts/planar_areas.py}), each more than $9/10$, whereas all the non-axis pyramid
facets together have area less than $1212/15625<1/4$. So $g$ permutes those nine planes; the
three distinct area types identify the height-$0$, height-$1$ and height-$2$ plane triples, hence
$g$ fixes the origin, the corner $(2,2,2)$ and the notch, and maps $\carrier$ to itself
(the Lean theorem \lean{planar\_area\_carrier\_recovery\_holds}, \tier{T1}, proves the same
statement by a different route: a point in a feature tube is at squared distance at most
$1089/100<12$ from every point of $\solid$, so the diameter pairs of $\solid$ lie in the
undeformed carrier and are its antipodal corner pairs, which $g$ permutes; the six corners
recover the centre and the axes, and the preserved missing corner excludes a reversal; the
planar-area calculation above is the written proof). Second, an isometry that preserves both $\solid$
and the chair has a signed permutation matrix as its linear part (the chair's own symmetry group
is the six coordinate permutations) and must carry the table of feature centres and heights to
itself; this reduction to one of the $48$ signed frames is a Lean theorem using
\lean{native\_decide}. Third, among the $48$ signed frames only the identity preserves the
table of $192$ features (a Lean theorem by kernel \lean{decide}). Consequently the stabilizer of
a tiling's family of placements equals its symmetry group, as noted in \cref{sec:terminology}.

\subsection{Feature angles}\label{app:angles}

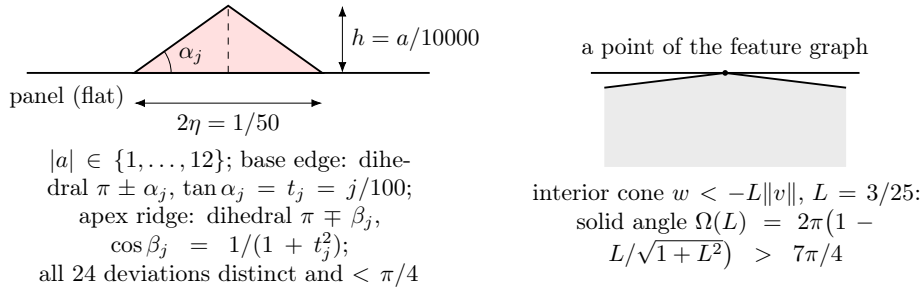
\begin{figure}[ht]
  \centering
  \resizebox{\linewidth}{!}{\input{fig_feature_geometry}}
  \caption{Schematic feature geometry, with exaggerated height: a pyramid of height $j/10000$,
  base half-width $1/100$, its deviation angles, and the interior cone with slope bound $L=3/25$
  used in \cref{sec:companions}.}
  \label{fig:feature}
\end{figure}

% ledger: O6, O7; Lean native_features_length, native_feature_centers_nodup, deviations_distinct, mesh_angle_audit
The $192$ features play $192$ distinct roles: their centres are distinct points of
$\tfrac18\Z^3$, and the twelve magnitudes give two sequences of twelve deviation angles, at the
base edges and at the apex edges of a pyramid of slope $t_j=j/100$, all $24$ distinct: a
coincidence between the two sequences would force $1+(i/100)^2=(1+(j/100)^2)^2$, that is
$10000\,i^2=20000\,j^2+j^4$, which has no solution with $1\le i,j\le 12$ (a Lean theorem by
kernel \lean{decide}). Every one of the $6{,}408$ mesh edges carries a declared dihedral angle,
flat or orthogonal except at the features; the $1{,}536$ feature edges fall into $48$ classes of
$32$, one class per deviation and sign (\cref{tab:mesh}).

\subsection{Local deformation and retained cores}\label{app:tubes}

% ledger: A16; Lean per_tube_homeomorphisms_holds, feature_tube_maps_glue_holds, feature_tube_map_carries_carrier_holds (T1n); disjointness and boundary identity T1
Around each feature take the normal tube of half-height $\rho=1/100$ over its base. The map
$(u,z)\mapsto(u,\,z+\chi(z)f(u))$, where $f$ is the signed tent function of the pyramid and
$\chi(z)=\max(0,1-|z|/\rho)$, is strictly increasing in $z$ because $\|f\|_\infty\le 12/10000<\rho$,
equals the identity on the boundary of the tube, and carries the flat base to the pyramid. The
$192$ tubes are pairwise disjoint (the marks on a panel are at least $1/8$ apart in one tangent
coordinate and more than $1/5$ from every panel edge, while the bases are $1/50$ across), so the
maps glue to a homeomorphism of $\R^3$ that is the identity outside the tubes and carries
$\carrier$ onto $\solid$. In particular every unit cube of the chair keeps its core
$[\rho,1-\rho]^3$ inside $\solid$, a fact used in \cref{sec:registration}.

\subsection{The open parameter family}\label{app:family}

% ledger: T14; source proof/PROOFS_registered.md R§11; form 2303.10798 L80 (Tile(a,b))
Finally, the integer heights are a convenient member of a family. Let the twelve component
heights be real numbers $c_1,\dots,c_{12}$, keep the signs of the graph, and use height
$c_j/10000$. With $t_j=|c_j|/100$, the whole proof goes through unchanged provided every $c_j$ is
nonzero and the magnitudes $|c_j|$ are distinct; $1+t_i^2\ne(1+t_j^2)^2$ for all $i,j$; every
$t_j<1$, every $(1+t_j^2)^2<2$, and $63\max_j t_j^2<1$; and the heights stay below $\rho$ and
$\eta$. These are finitely many strict inequalities satisfied by $c_j=j$ (\cref{tab:family}), so
\cref{thm:full} holds on an open twelve-parameter family of solids with the same architecture.
This is not stability under arbitrary perturbation of the boundary, and we claim none.

\begin{table}[ht]
  \centering
  \caption{The parameter family: the exact open conditions of R\S11 and their values at the
  explicit member $c_j=j$, checked with exact rationals (FIGURES.md, Table 7).}
  \label{tab:family}
  \small
  \input{generated/parameter_family_table}
\end{table}

%% file: generated/mesh_table.tex
% generated by paper/scripts/mesh_table.py; from solid/r44_solid.json, exact arithmetic, cross-checked against its counts
\begin{tabular}{L{0.66\linewidth} r}
\toprule
quantity & value \\
\midrule
vertices $V$ & $2{,}138$ \\
edges $E$ & $6{,}408$ \\
triangles $F$ & $4{,}272$ \\
$V-E+F$ & $2$ \\
volume (exact, divergence formula) & $7$ \\
features (square pyramids) & $192$ \\
feature edges; classes (magnitude, base/ridge, sign) $\times$ edges per class & $1{,}536$; $48\times32$ \\
base half-width $\eta$; height unit & $1/100$; $1/10000$ \\
\bottomrule
\end{tabular}

%% file: fig_feature_geometry.tex
% fig_feature_geometry.tex — Fig. 2.3: a feature in cross-section (schematic; heights exaggerated,
% the exact numbers are printed). Used by sec2_solid.tex.
\begin{tikzpicture}[x=1cm,y=1cm,font=\small,>={Latex[length=1.6mm]}]
  % left: pyramid cross-section through the apex, along a tangent axis: base [-eta, eta], apex height h
  \begin{scope}
    \draw[thick] (-3.0,0) -- (3.0,0);
    \node[anchor=north] at (-2.4,-0.05) {panel (flat)};
    \draw[thick,fill=red!12] (-1.4,0) -- (0,1.0) -- (1.4,0) -- cycle;
    \draw[<->] (-1.4,-0.45) -- (1.4,-0.45) node[midway,below] {$2\eta = 1/50$};
    \draw[<->] (1.7,0) -- (1.7,1.0) node[midway,right] {$h = a/10000$};
    \draw[dashed] (0,0) -- (0,1.0);
    \draw (-1.4,0) ++(0.55,0) arc (0:35.5:0.55);
    \node at (-0.55,0.25) {$\alpha_j$};
    \node[anchor=north,align=center,text width=6.4cm] at (0,-1.0)
      {$|a|\in\{1,\dots,12\}$; base edge: dihedral $\pi\pm\alpha_j$, $\tan\alpha_j=t_j=j/100$;\\
       apex ridge: dihedral $\pi\mp\beta_j$, $\cos\beta_j=1/(1+t_j^2)$;\\ all $24$ deviations distinct and $<\pi/4$};
  \end{scope}
  % right: the Lipschitz cone at a point of the feature graph
  \begin{scope}[shift={(7.4,0)}]
    \draw[thick] (-2.0,0) -- (2.0,0);
    \fill[black!8] (0,0) -- (-1.8,-0.22) -- (-1.8,-1.4) -- (1.8,-1.4) -- (1.8,-0.22) -- cycle;
    \draw[thick] (0,0) -- (-1.8,-0.22);
    \draw[thick] (0,0) -- (1.8,-0.22);
    \fill (0,0) circle (1.2pt);
    \node[anchor=south] at (0,0.08) {a point of the feature graph};
    \node[anchor=north,align=center,text width=6.4cm] at (0,-1.5)
      {interior cone $w<-L\|v\|$, $L=3/25$:\\ solid angle $\Omega(L)=2\pi\bigl(1-L/\sqrt{1+L^2}\bigr)>7\pi/4$};
  \end{scope}
\end{tikzpicture}

%% file: generated/parameter_family_table.tex
% generated by paper/scripts/parameter_family_table.py; R\S11 conditions verified at c_j = j with exact rationals
\begin{tabular}{L{0.5\linewidth} L{0.42\linewidth}}
\toprule
condition (R\S11) & at $c_j=j$ \\
\midrule
(1) every $c_j\ne0$, magnitudes $|c_j|$ distinct & $|c_j|=1,\dots,12$, distinct \\
(2) $1+t_i^2\ne(1+t_j^2)^2$ for all $i,j$ ($t_j=|c_j|/100$) & holds; smallest gap $39/390625$ \\
(3) $t_j<1$; $(1+t_j^2)^2<2$; $63\max_j t_j^2<1$ & $\max t_j=3/25$; $(1+t^2)^2=401956/390625$; $63t^2=567/625$ \\
(4) heights $|c_j|/10000$ below $\rho=1/100$ and $\eta=1/100$ & $\max=3/2500$ \\
\bottomrule
\end{tabular}

%% file: notation.tex
\section{Notation}\label{app:notation}

% ledger: D1, D3, D5, O1, A11; the definitions of sec:solid, sec:substitution, sec:companions, sec:registration, sec:hierarchy
\begin{table}[ht]
  \centering
  \caption{Notation used throughout, with the section that defines it.}
  \label{tab:notation}
  \small
  \begin{tabular}{l L{0.62\linewidth}}
    \toprule
    symbol & meaning \\
    \midrule
    $\carrier$ & the seven-cube chair carrier, volume $7$ (\cref{sec:solid}) \\
    panel & one of the $24$ exposed unit squares of $\carrier$ (\cref{sec:solid}) \\
    feature & one of the $192$ square pyramids, eight per panel (\cref{sec:solid}) \\
    $\eta$, $\rho$ & base half-width $1/100$ of a feature; tube half-height and retained-core margin $1/100$ (\cref{sec:solid,app:tubes}) \\
    $\solid$ & the solid: the carrier with its $192$ features (\cref{sec:solid}) \\
    role, profile & a feature position indexed by panel and mark; the vector of $192$ signed feature coefficients (\cref{sec:solid}) \\
    $t_j$, $\alpha_j$, $\beta_j$ & slope $j/100$ of a feature of magnitude $j$; its base-edge and ridge deviation angles (\cref{lem:dihedral}) \\
    $E(F)$ & the closed edge graph of a feature $F$: four base edges and four ridges (\cref{sec:companions}) \\
    frame & a signed permutation of the coordinate axes; registered frames are the $24$ proper ones (\cref{sec:substitution}) \\
    registered pose & a proper cubic frame followed by an integer translation (\cref{sec:substitution}) \\
    $\atlas$ & the $44$ legal neighbour poses relative to a root chair in the identity pose (\cref{sec:substitution}) \\
    $F_0$, $F_1$, $\mathrm{Partners}(u)$ & the $6{,}862$ complete-feature companion poses; those not overlapping the root; those in $F_1$ matching the root feature $u$ (\cref{sec:registration}) \\
    $T$, $\Per(T)$, $\Sym(T)$ & a tiling; its translation periods; its symmetry group (\cref{sec:terminology}) \\
    handedness & the determinant $\pm1$ of the linear part of a placement (\cref{sec:aperiodicity}) \\
    supertile of level $n$ & a $2^n$-scaled chair in a registered pose (\cref{sec:hierarchy}) \\
    coarsening $D$ & grouping into unique parents, halving parent origins and scale, and realizing fresh copies of $\solid$ (\cref{thm:coarsening}) \\
    \tier{T1} & kernel-checked Lean proof using only standard axioms \\
    \tier{T1n} & kernel-checked Lean proof with the recorded native compiler hooks \\
    \tier{T2} & exact finite Python replay \\
    \tier{T3} & written proof or cited external import \\
    \bottomrule
  \end{tabular}
\end{table}